\documentclass[11pt, a4paper]{amsart}
\usepackage[utf8]{inputenc}
\usepackage[usenames,dvipsnames]{xcolor}
\usepackage{amsmath}
\usepackage{amsthm}
\usepackage{amsfonts}
\usepackage{amssymb}
\usepackage[foot]{amsaddr}
\usepackage{array}  
\usepackage{graphicx} 
\usepackage{color}
\usepackage{mathrsfs}
\usepackage{graphicx}
\usepackage{caption}
\usepackage{subcaption}
\usepackage{dsfont}
\usepackage{geometry}
\usepackage[final]{hyperref} 
\usepackage{enumerate} 

\hypersetup{
    linktoc=page,
    linkcolor=red,         
    citecolor=blue,        
    filecolor=blue,        
    urlcolor=cyan,
    colorlinks=true      } 

\theoremstyle{plain}
\newtheorem{thm}{Theorem}[section]

\newtheorem{claim}[thm]{Claim}
\newtheorem{theorem}[thm]{Theorem}
\newtheorem{corollary}[thm]{Corollary}
\newtheorem{proposition}[thm]{Proposition}
\newtheorem{conjecture}[thm]{Conjecture}
\newtheorem{question}[thm]{Question}
\newtheorem{lemma}[thm]{Lemma}

\theoremstyle{definition}
\newtheorem{definition}[thm]{Definition}

\newtheorem{remark}[thm]{Remark}
\newtheorem{example}[thm]{Example}

\newtheorem{assumption}[thm]{Assumption}

\theoremstyle{remark}

\newcommand{\E}{\mathbb{E}}

\newcommand{\N}{\mathbb{N}}

\newcommand{\R}{\mathbb{R}}

\newcommand{\Z}{\mathbb{Z}}

\def\calC{\mathcal{C}}
\def\calD{\mathcal{D}}

\def\calF{\mathcal{F}}

\def\calM{\mathcal{M}}

\newcommand{\un}{\mathds{1}}

\renewcommand{\textbf}[1]{\begingroup\bfseries\mathversion{bold}#1\endgroup}
\def\var{\mathop{\mathrm{Var}}}

\def\prob{\mathbb{P}}
\newcommand{\Pro}{\mathbb{P}}
\def\eps{\varepsilon}

\newcommand{\Cross}{\text{\textup{Cross}}}
\newcommand{\Ann}{\text{\textup{Ann}}}

\newcommand{\Circ}{\text{\textup{Circ}}}

\def\thr{\textup{T}}
\def\sad{\textup{S}}

\def\br#1{\left(#1\right)}
\def\brb#1{\left[#1\right]}

\numberwithin{equation}{section}
\definecolor{alscolor}{rgb}{0.84, 0.04, 0.33}

\usepackage{etoolbox}

\makeatletter
\patchcmd{\@setauthors}{\MakeUppercase}{}{}{}
\makeatother

\begin{document}

\title[The phase transition for planar Gaussian percolation without FKG]{The phase transition for planar Gaussian \\ percolation models without FKG}
\author{\uppercase{Stephen Muirhead$^{\dagger}$}}
\author{\uppercase{Alejandro Rivera$^{\dagger\dagger}$}}
\author{\uppercase{Hugo Vanneuville$^{\dagger\dagger\dagger}$}\\ $\,$ \\  with an appendix by \uppercase{Laurin Köhler-Schindler$^{\dagger\dagger\dagger}$}}
\address{\hfill\protect\parbox{0.975\linewidth}{$^\dagger$School of Mathematical Sciences, Queen Mary University of London (Current address: School of Mathematics and Statistics, University of Melbourne). SM was partially supported by the Australian Research Council (ARC) Discovery Early Career Researcher Award DE200101467. \\ $^{\dagger\dagger}$Institute of Mathematics, EPFL \\ $^{\dagger\dagger\dagger}$Department of Mathematics, ETH Zürich. HV and LKS were supported by the SNF Grant No 175505. LKS received funding from the European Research Council (ERC) under the European Union's Horizon 2020 research and innovation program (grant agreement No 851565).}}
\email{smui@unimelb.edu.au}
\email{alejandro.rivera@epfl.ch}
\email{hugo.vanneuville@math.ethz.ch}
\email{\hspace{0.3cm} laurin.koehler-schindler@math.ethz.ch}

\begin{abstract}
We develop techniques to study the phase transition for planar Gaussian percolation models that are \textit{not} (necessarily) positively correlated. These models lack the property of positive associations (also known as the `FKG inequality'), and hence many classical arguments in percolation theory do not apply. More precisely, we consider a smooth stationary centred planar Gaussian field $f$ and, given a level $\ell \in \R$, we study the connectivity properties of the excursion set $\{f \geq -\ell\}$. We prove the existence of a phase transition at the critical level $\ell_{crit}=0$ under only symmetry and (very mild) correlation decay assumptions, which are satisfied by the \textit{random plane wave} for instance. As a consequence, all non-zero level lines are bounded almost surely, although our result does not settle the boundedness of zero level lines (`no percolation at criticality').

To show our main result: (i) we prove a general sharp threshold criterion, inspired by works of Chatterjee, that states that `sharp thresholds are equivalent to the delocalisation of the threshold location'; (ii) we prove threshold delocalisation for crossing events at large scales -- at this step we obtain a sharp threshold result but without being able to locate the threshold -- and (iii) to identify the threshold, we adapt Tassion's RSW theory replacing the FKG inequality by a sprinkling procedure. Although some arguments are specific to the Gaussian setting, many steps are very general and we hope that our techniques may be adapted to analyse other models without FKG.
\end{abstract}
\date{\today}
\keywords{Percolation; Gaussian fields; phase transition}
\subjclass[2010]{60K35; 60G60}
\maketitle

\vspace{-0.4cm}
\tableofcontents

\vspace{-0.3cm}
\section{Introduction}\label{s:intro}

In this paper we study the phase transition for a class of planar percolation models which lack:
\begin{itemize}
\item the property of `positive associations', also known as the Fortuin--Kasteleyn--Ginibre (FKG) inequality, and
\item other structural properties common in statistical physics such as finite energy, spatial independence at large scales\footnote{In the models we consider pointwise correlations do decay but extremely slowly.}, integrability, or the domain Markov property.
\end{itemize}
\smallskip
The inputs we use are mainly
\[
\textit{planarity}, \quad \textit{ergodicity}, \quad \textit{symmetries}, \quad \text{and at some steps \textit{Gaussianity}},
\]
though this last property could perhaps be replaced by \textit{hypercontractivity}. We refer the reader interested in applying our methods to other models to Section~\ref{ss:strategy}, in which we describe a general strategy to establish the phase transition from these four properties.

\smallskip
More concretely, we study percolation models given by the excursion sets of smooth stationary centred Gaussian fields on the plane. Although these models have been studied before, previous work has considered fields that satisfy the FKG inequality and/or for which correlations decay relatively quickly. In this paper we prove the existence of a phase transition at the critical level $\ell_\text{crit} = 0$ assuming neither of these properties.

\smallskip
Our work belongs to the study of \textit{sharp thresholds}. Indeed, the core of the paper consists of proving that the probability of `crossing events' at large scales jumps from close to $0$ to close to $1$ over a small interval of levels. A general approach to proving sharp thresholds (see \cite{rus82}) uses the insight that, as described in \cite{tal94}, an event satisfies such a property if it `depends little on any given coordinate'. There are different ways to formalise this, and in the present work we propose the following new interpretation: an event depends little on any given coordinate if the `threshold location delocalises' (see Section \ref{ss:strategy}).

\subsection{Level set percolation for planar Gaussian fields}\label{ss:gaussian_perco}
Let $f: \mathbb{R}^2 \to \R$ be a continuous stationary centred Gaussian field, and let $\kappa(x)=\E[f(0)f(x)]$ denote its covariance kernel. In recent years there has been an extensive investigation into the geometric and topological properties of the level and excursion sets\footnote{We consider $\{f \geq -\ell\}$ instead of $\{f \le \ell\}$ since the former has the advantage of being both increasing in $f$ and in $\ell$; of course, by symmetry, these sets have the same law.}
\[ \{ f = \ell \} := \{ x \in \R^2 : f(x) = \ell\} \quad \text{and} \quad \{f \geq -\ell\} :=\{x \in \R^2 : f(x) \geq -\ell\}  \ , \quad \ell \in \R\, . \]
For example, it has been proven for a wide family of fields that geometric quantities such as the length/area of the level/excursion sets satisfy central limit theorems (see for instance \cite{kl01,kv18, npr19}), and topological quantities such as the number of connected components of the level/excursion sets have been shown to satisfy laws of large numbers and concentration of measure (see for instance \cite{ns09, ns16}).

\smallskip
In this paper we are interested in \textit{percolation properties} of the level/excursion sets of Gaussian fields. It has long been believed (see for instance  \cite{dyk70,zs71,is92,al96}) that (under very mild assumptions) the connectivity exhibits a  phase transition at the critical level $\ell_{crit} = 0$ analogous to the phase transition in many planar percolation models:
\begin{equation}\label{e:pt1}
\hspace{-1cm} \text{For } \ell \le 0, \; \{f \geq -\ell\} \text{ has bounded connected components almost surely;}
\end{equation} 
\vspace{-0.7cm}
\begin{equation}\label{e:pt2}
\text{For } \ell > 0, \; \{f \geq -\ell\} \text{ has a unique unbounded connected component almost surely.}
\end{equation}
\noindent Note that \eqref{e:pt1} is analogous to Harris' theorem \cite{ha60} in Bernoulli percolation, while \eqref{e:pt2} is analogous to Kesten's theorem \cite{ke80}.

\smallskip
Recently the phase transition \eqref{e:pt1}--\eqref{e:pt2} has been proven for a class of planar Gaussian fields whose correlations satisfy the conditions of (i) positivity ($\kappa \ge 0$), and (ii) integrability ($\kappa \in L^1$); see \cite{bg17, bm18, rv19, mv20} and \cite{rv20, mv20, ri19, gv19} for quantitative versions of \eqref{e:pt1} and \eqref{e:pt2} respectively. These two conditions are satisfied for many natural fields, for instance the \textit{Bargmann--Fock field} (see \cite{bg17}) and the \textit{ (discrete) massive Gaussian free field} (see \cite{rod17}), but are not satisfied in other important examples, such as the \textit{random plane wave} (RPW) introduced in Example \ref{ex:rpw} below. 

\smallskip
If only one of these conditions is satisfied then partial results are available. For instance, \eqref{e:pt1} is known under the positivity condition (and some mild extra conditions) \cite{al96}. Moreover, if the correlations are integrable then one can prove that the critical level $\ell_\text{crit}$ is finite (see for instance \cite{ms83a, ms83b}). However, if neither condition is satisfied then it was not even known before the present work that the critical level $\ell_\text{crit}$ was finite, let alone its exact value $\ell_\text{crit} = 0$.

\smallskip
From the perspective of percolation theory, one can highlight two main obstacles to establishing \eqref{e:pt1}--\eqref{e:pt2} in full generality:
\begin{enumerate}
\item \textbf{Lack of positive associations / FKG.} `Positive associations' (or the `FKG inequality', proved by Harris \cite{ha60} for Bernoulli percolation) refers to the property that events that are increasing with respect to the field are positively correlated. For Gaussian fields, positive associations is known to be equivalent to $\kappa \ge 0$ (see \cite{pi82}).

\smallskip
Positive association is a central tool in percolation theory, in particular for `gluing of paths' constructions, and not having this property limits the applicability of many classical techniques.

\smallskip
\item \textbf{Lack of quasi-independence.} If correlations decay sufficiently rapidly one can prove that the level/excursion sets satisfy a certain \textit{quasi-independence} property: percolation events on domains of scale $R$ that are separated by a distance of order $R$ are asymptotically independent. 

\smallskip
Satisfying quasi-independence is believed to be equivalent to belonging to the universality class of Bernoulli percolation, in the sense that the model shares large-scale connectivity properties (e.g.\ critical exponents, conformal invariant scaling limits etc.) with critical Bernoulli percolation. Moreover, although quasi-independence is conjectured to be true if $|\kappa(x)| \ll |x|^{-3/2}$ (see \cite{wei84,bmr20}, and also \cite{rv19,mv20,bmr20} for rigorous results in the case of integrable correlations), it is conjectured to fail if $\kappa$ is positive and decays more slowly than $|x|^{-3/2}$ (although oscillations in the covariance mean that it \textit{can} hold if $\kappa$ decays more slowly, for instance for the RPW).

\smallskip
Since in general smooth Gaussian fields also do not satisfy `domain Markov' or `finite energy' properties, the lack of spatial independence severely limits the applicability of many techniques from classical percolation theory. 
\end{enumerate}

\smallskip
In this paper we show that, if we restrict our attention to non-critical levels $\ell\neq 0$, we can circumvent these obstacles to establish the existence of the phase transition at $\ell_\text{crit} = 0$ for a very wide class of Gaussian fields (see Theorem \ref{t:main} below). We emphasise that this result avoids the use of positive associations, and is not limited to a perturbative regime.

\subsection{The phase transition at the zero level}\label{ss:phase_transition}
To state our results we need the following very mild smoothness, non-degeneracy and correlation decay assumptions:

\begin{assumption} $\,$
\label{as:main}
\begin{itemize}
\item (Smoothness) The field $f$ is almost surely $C^{3}$-smooth;
\item (Non-degeneracy) For each $x \in \R^2 \setminus \{0\}$, the Gaussian vector
\[  (f(0), f(x), \nabla_0 f, \nabla_x f ) \in \R^6 \]
is non-degenerate;
\item (Correlation decay) $\kappa(x) \to 0$ as $|x| \to \infty$.
\end{itemize}
\end{assumption}

\begin{remark}
\label{r:assump}
The field $f$ is almost surely $C^3$-smooth if $\kappa$ is of class~$C^8$ \cite[Appendix A.9]{ns16}. Moreover, the non-degeneracy condition is satisfied if the support of the spectral measure contains an open disc or a circle centred at the origin \cite[Lemma A2]{bmm19}. The assumption $\kappa(x) \to 0$ implies that the field is ergodic (see for instance \cite[Theorem 6.5.4]{adler10}).
\end{remark}

Since $f$ is assumed $C^3$-smooth, we may view $f$ as a random variable in the set $C^3(\R^2)$ equipped with its Borel $\sigma$-algebra (which is also the $\sigma$-algebra generated by the projections $u \in C^3(\R^2) \mapsto u(x)$ for every $x \in \R^2$, see for instance \cite[Lemma A.1]{ns16}). We immediately complete this $\sigma$-algebra and work with its completion (a.k.a. the Lebesgue $\sigma$-algebra) in the rest of the paper.

\smallskip
We shall also need the following two notions of \textit{symmetry} for the field $f$:
\begin{itemize}
\item ($D_4$-symmetry) $f$ is \textit{$D_4$-symmetric} if its law is invariant with respect to reflections in the horizontal and vertical axes and with respect to rotations by $\pi/2$.
\item (Isotropy) $f$ is \textit{isotropic} if its law is invariant with respect to all rotations; in particular this implies that $f$ is also $D_4$-symmetric.
\end{itemize}

We can now state the main result of the paper:
\begin{theorem}[The phase transition at the zero level]
\label{t:main}
Let $f$ be a Gaussian field satisfying Assumption \ref{as:main}.
\begin{itemize}
\item Suppose that $f$ is $D_4$-symmetric. Then, for each $\ell<0$, the set $\{f \geq -\ell\}$ has bounded connected components almost surely. In particular the level lines at levels $\ell \neq 0$ are bounded almost surely.
\item Suppose in addition that $f$ is isotropic, and there exists $\delta > 0$ such that, as $|x| \to \infty$,
\begin{equation}
\label{e:decaymain}
 |\kappa(x)| (\log\log|x|)^{2+\delta} \to 0. 
 \end{equation}
Then, for each $\ell > 0$, the set $\{ f \geq -\ell\}$ has a unique unbounded connected component almost surely.
\end{itemize}
\end{theorem}

\begin{example}[The random plane wave]\label{ex:rpw}
As a motivating example, consider the \textit{random plane wave} (RPW) (also known as the `monochromatic random wave'), which is the smooth stationary centred Gaussian field $f$ with covariance kernel $\kappa(x) = J_0(|x|)$, where $J_0$ is the zeroth Bessel function. Since, as $r \to \infty$,
\[ J_0(r)  = \sqrt{\frac{2}{\pi r}} \cos(r-\pi/4) + O(1/r),  \]
the covariance kernel $\kappa$ is \textit{neither positive nor integrable}. However, it is easy to check that the RPW satisfies all the conditions in Theorem \ref{t:main} since it is smooth, isotropic, and the support of its spectral measure is the unit circle (see Remark \ref{r:assump}).

\smallskip
Percolation properties of the RPW are of particular interest since it is has been conjectured that the set $\{f \geq 0\}$ lies in the universality class of Bernoulli percolation \cite{bs02, bds06, bs07}. Although we are unable to say anything about the percolation of $\{ f = 0\}$, our main result proves the existence of a \textit{phase transition} at $\ell_\text{crit}=0$ between the absence and presence of percolation (see Figure \ref{f:rpw}).

\begin{figure}[h!]
\begin{center}
\hspace{0.3cm}
\begin{minipage}{0.4\textwidth}
\includegraphics[width = 0.8\textwidth]{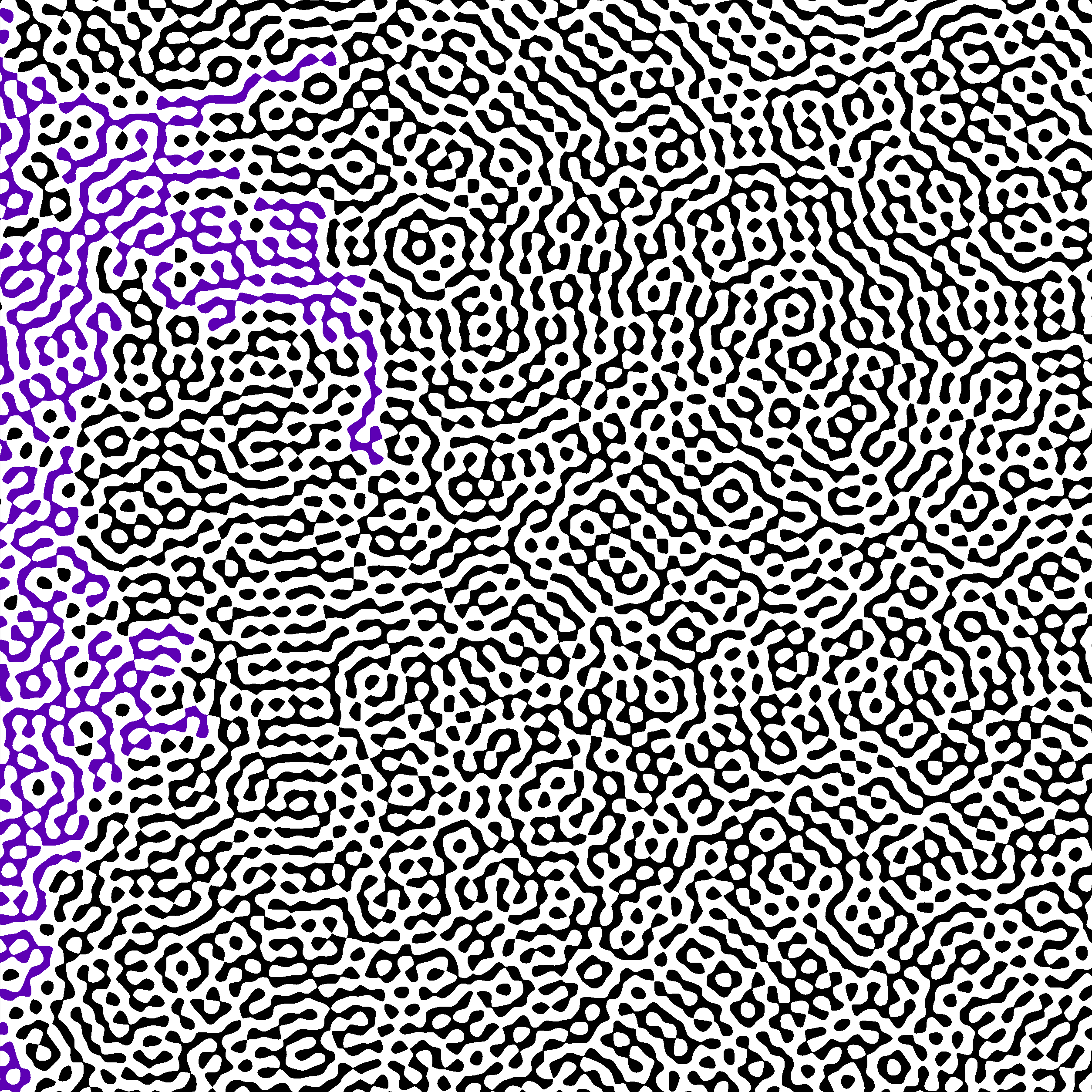}  \\ \phantom{lllllllllllllllll} $\ell = -0.1$
\end{minipage}
\hspace{0.2cm} 
\begin{minipage}{0.4\textwidth}
\includegraphics[width = 0.8\textwidth]{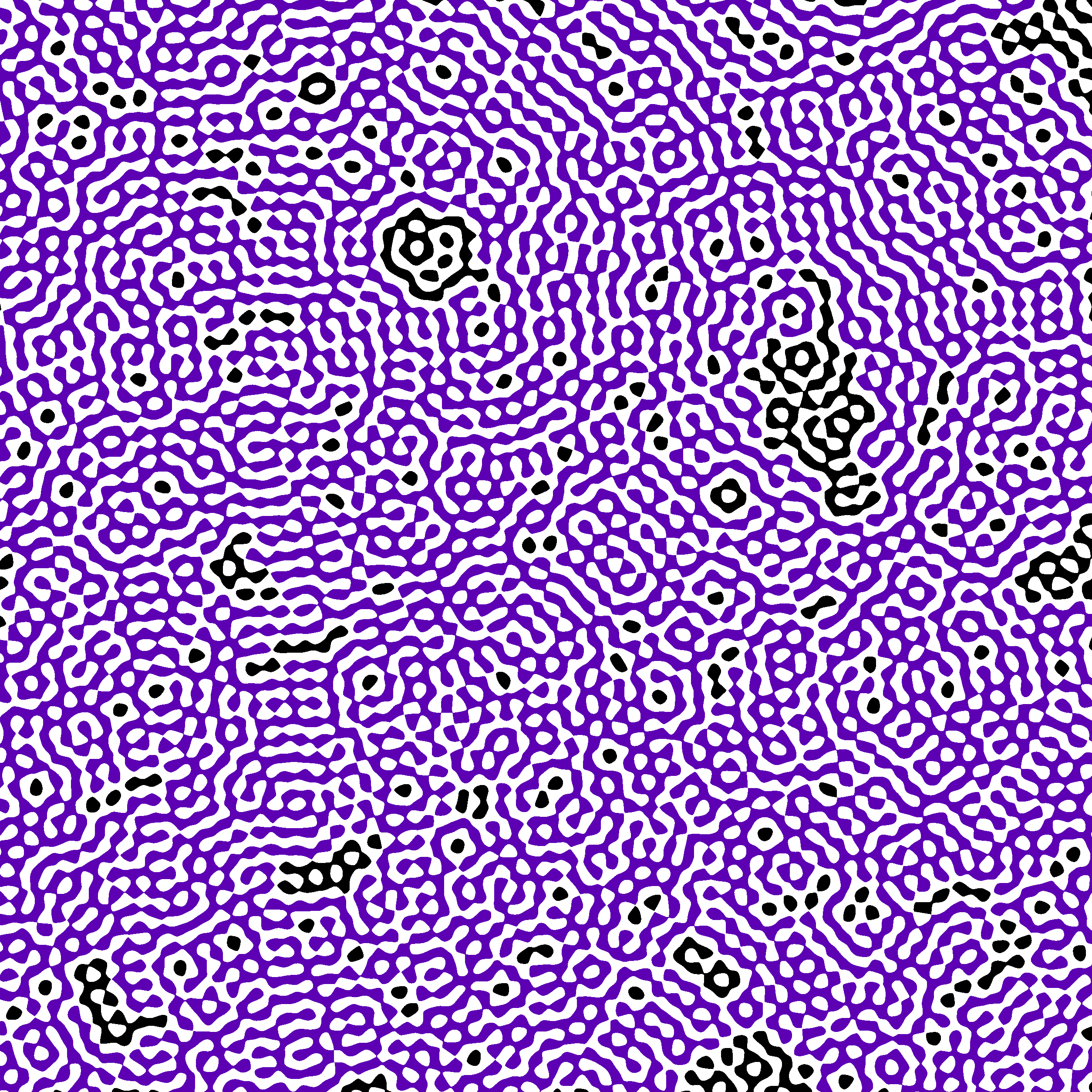} \\ \phantom{llllllllllllllllllll} $\ell = 0.1$
\end{minipage}
\end{center}
\caption{The phase transition for the RPW. The set $\{ f \geq -\ell \}$ is coloured black; the set $\{ f \leq -\ell \}$ is coloured white; the purple set is the black component of the left side of the square.\label{f:rpw}}
\end{figure}
\end{example}

\begin{remark}[Possible extensions]\label{rk:assumptions_for_main_result}
We expect the conclusion of Theorem \ref{t:main} in the supercritical regime $\ell > 0$ to be true without the extra symmetry and correlation decay assumptions, and it would be interesting to remove them (see Section \ref{s:open}). As explained in Remark \ref{r:bootstrap}, isotropy can be replaced by $D_4$-symmetry if $\kappa$ decays sufficiently rapidly.
\end{remark}

\subsection{The sharpness of the phase transition}
\label{ss:sharpness}
We next address the \textit{sharpness} of the phase transition. For Bernoulli percolation, the probability of crossing events in the subcritical regime decays exponentially in the scale. While we suspect this to also be true for a large family of Gaussian fields, and in particular for the RPW, the techniques of the present paper only provide a much weaker conclusion. 

\smallskip
Let us first introduce the aforementioned crossing events:

\begin{definition}[Crossing events]
\label{d:conn}
For $a,b > 0$ and level $\ell \in \R$, let $\Cross_\ell(a, b)$ denote the `rectangular crossing event' that $\{f \geq -\ell \} \cap ([0,a] \times [0,b])$ contains a path that intersects both $\{0\} \times [0,b]$ and $\{a\} \times [0, b]$.
\end{definition}

\begin{theorem}[Sharpness of the phase transition]
\label{t:sharp}
Let $f$ be a Gaussian field satisfying Assumption \ref{as:main}.
\begin{itemize}
\item Suppose that $f$ is $D_4$-symmetric. Then there exists an unbounded sequence $R_n \to \infty$ such that, for every $\ell < 0$ and $a>0$, as $n \to \infty$,
\begin{equation}
\label{e:hprsw1}
\prob[ \Cross_\ell(R_n, aR_n) ]  \to 0.
 \end{equation}
\item Suppose in addition that $f$ is isotropic, and there exists $k \geq 1$ such that, as $|x| \to \infty$,
\[ |\kappa(x)|\log^{(k)}(|x|) \to 0, \]
where $\log^{(k)} (x) := \log \log \cdots \log x$ denotes the $k$-fold composition of the logarithm. Then for every $\ell < 0$ and $a>0$ there exist $c_1,c_2 > 0$ such that, for every $R \ge 1$,
\begin{equation}\label{e:hprsw2}
\prob[ \Cross_\ell(R, aR) ] < c_1 e^{-c_2 \sqrt{ \min\{ \log R, 1/ \overline \kappa (\sqrt{R}) \} } }  ,
 \end{equation}
 where 
 \begin{equation}\label{eq:kappa_bar}
\overline \kappa(r) := \sup_{|x| \ge r} |\kappa(x)|.
\end{equation}
\end{itemize}
In particular, if $f$ is the RPW then, for every $\ell < 0$ and $a > 0$ there exist $c_1,c_2 > 0$ such that, for every $R \ge 1$,
\[ \prob[ \Cross_\ell(R, aR) ] < c_1 e^{-c_2 \sqrt{\log R} }  . \]
\end{theorem}

\begin{remark}
We expect that \eqref{e:hprsw1} holds for any sequence of scales without further assumptions (for more about this conjecture, and a connection to the recent work by Köhler-Schindler and Tassion \cite{kt20}, see Section \ref{s:open}). Our proof actually shows that the sequence $(R_n)_{n \ge 1}$ can be chosen so that it eventually satisfies
\[ R_{n+1} \le R_n 4^{\log^\ast_{2^{1/4}} (R_n)}, \]
where $\log^\ast_b$ denotes the \textit{base-b iterated logarithm}; see \eqref{e:iteratedlog} for a precise definition (but here we simply note that it grows slower than $\log^{(k)}$ for any $k \ge 1$). Similarly, we could weaken the hypothesis `$|\kappa(x)|\log^{(k)}(|x|) \to 0$ for some $k \geq 1$' by substituting $(\log^\ast(|x|))^2$ in place of $\log^{(k)}(|x|)$, but we have chosen the present formulation for simplicity (and because this is already much weaker than \eqref{e:decaymain}).
\end{remark}

\begin{remark}
\label{r:bootstrap}
If correlations decay sufficiently rapidly, it is possible to `bootstrap' the results in Theorem \ref{t:sharp} to achieve a faster decay of crossing probabilities (for instance, combining the mixing estimate in \cite[Corollary 1.2]{bmr20}, \cite[Theorem 1.12]{rv19} or \cite[Theorem~4.2]{mv20} with the arguments in \cite[Theorem 6.1]{mv20}); in some cases it can even be shown that crossing probabilities decay exponentially \cite[Theorem 6.1]{mv20}. For simplicity, and since these arguments appear elsewhere \cite{rv20, mv20}, we refrain from stating a precise version of this result.

As announced in Remark \ref{rk:assumptions_for_main_result}, one consequence of the availability of bootstrapping methods is that one may bypass the part of the proof of Theorem \ref{t:main} that relies on isotropy. As such, whenever correlations decay fast enough the isotropy assumption can be weakened to $D_4$-symmetry, which would enable our techniques to apply to lattice models.
\end{remark}

As a consequence of Theorem \ref{t:sharp} we deduce the non-existence of `giant components' of $\{f \geq -\ell\}$ for $\ell<0$. For the RPW, this answers a question of Sodin \cite[Question 6]{sodin16}:

\begin{corollary}
\label{c:sodin}
Let $f$ be an isotropic field satisfying Assumption \ref{as:main}, and suppose there exists $k \geq 1$ such that, as $|x| \to \infty$,
\[ |\kappa(x)|\log^{(k)}(|x|) \to 0 . \]
Then for each $\ell < 0$ and $\varepsilon > 0$, as $R \to \infty$,
\[ \prob[  \{f \geq -\ell\} \cap  B_0(R)   \text { has a component with diameter greater than } \varepsilon  R  \, ] \to  0 \]
where $B_0(R)$ is the Euclidean ball of radius $R$ centred at $0$.
\end{corollary}

\subsection{A general strategy for planar models without FKG}
\label{ss:strategy}

There are many natural statistical physics models that lack positive associations (e.g.\ FK models with $q < 1$, certain regimes of $O(n)$ loop models, anti-ferromagnetic Ising models, random current models, Boolean models on non-Poisson point processes etc.), and in general their phase transitions are poorly understood. In this section we provide an informal description of our proof strategy in the hope that it might eventually be adapted to a wider class of models without positive associations or any spatial independence, domain Markov or finite energy properties.

\smallskip
For clarity we present the strategy in the simpler setting of Bernoulli percolation on $\Z^2$, defined by erasing independently each edge with probability $1-p$ for some parameter $p \in [0,1]$. A famous result of Kesten \cite{ke80} is that the critical parameter is $p_{crit}=1/2$: there is almost surely an infinite connected component if and only if $p > 1/2$. To draw a closer link to the present work, we prefer the following equivalent definition of Bernoulli percolation: associate an independent standard normal random variable $X_e$ to each edge $e$ and erase edges for which $X_e < -\ell$. Then the parameter is a level $\ell \in \R$ and the critical level is $\ell_{crit}=0$.
 
\smallskip
The proof of Kesten's theorem relies on the analysis of crossing events (see Definition~\ref{d:conn}), in particular of scaled copies of rectangles. Kesten's proof -- and other more recent proofs (e.g.\ \cite{rus82,br06, bd12}) -- proceeds roughly as follows:
\begin{enumerate}
\item[(a)] Use gluing arguments, which rely crucially on the FKG inequality, to prove that crossing events at level $\ell=0$ are non-degenerate (this is known as `Russo--Seymour--Welsh (RSW) theory'); 
\item[(b)] Apply a differential formula and/or an abstract sharp threshold result to prove that crossing events have a sharp threshold, meaning that 
\[ \qquad \quad \ell \mapsto\prob_\ell[\textup{crossing of rectangles at scale } R ] \]
approximates a step function as $R \to \infty$; part (a) allows one to identify that the `step' occurs at $\ell=0$ and hence to establish that crossing probability tends to $0$ if $\ell<0$ and to $1$ if $\ell>0$;
\item[(c)] Use `bootstrapping' to make the convergence quantitative, and conclude by using a Borell--Cantelli argument to construct an unbounded connected component.
\end{enumerate}

The essence of our strategy is to \textit{invert the order of steps (a) and (b)}. Precisely, we first prove a sharp threshold result \textit{without locating the level at which the thresholds occurs}, and then observe that the existence of sharp thresholds permits us to dispense with the FKG inequality in the RSW theory. Moreover, in order to prove the sharp threshold result (step (b) above), we propose a new general criterion: \textit{sharp thresholds occur if and only if the `threshold location delocalises'}. We now describe this strategy in more detail.

\subsubsection{Sharp thresholds from the `delocalisation of the threshold location'} As mentioned above, a general approach to establishing sharp thresholds \cite{rus82} is to prove that an event `depends little on every coordinate' (here, the coordinates are the edges). There are several ways to formalise this:
\begin{itemize}
\item[(1)] \textit{Influences.} Russo's original formalisation \cite{rus82} uses the notion of \textit{influence}, which in Bernoulli percolation refers to the probability $\text{Inf}_e^\ell(A)$ that an edge $e$ is `pivotal' for an event $A$, i.e.\ changing the state of $e$ modifies the outcome of $A$. `Russo's approximate $0$-$1$ law' states that $\sup_{\ell, e} \text{Inf}_e^\ell(A) \ll 1$ implies a sharp threshold. An alternate proof of the $0$-$1$ law is given by the BKKKL theorem which exploits hypercontractive properties of the Boolean hypercube; this was used by \cite{br06} to give a new proof of Kesten's theorem. However, proving that influences are small seems delicate without FKG or finite energy, and moreover existing proofs of Russo's approximate $0$-$1$ law rely either on strong positive associations \cite{gg06} or strong independence properties \cite{rv20}.

\smallskip
\item[(2)] \textit{Decision trees.} A second formalisation is via decision trees, and in fact the existence of a decision tree with small `revealment' implies a sharp threshold (as quantified for instance by the OSSS inequality). However, again this formalisation has only been applied successfully, thus far, in models with strong positive associations \cite{dcrt19}.

\smallskip
\item[(3)] \textit{Threshold location.} We propose a new and third formalisation using the \textit{threshold location}. To define this, consider an increasing event $A$ and let $\thr_A := \inf\{ \ell : A \text{ holds at level } \ell\}$; we call this the \textit{threshold height} of the event (see \cite{as17} for a study of the threshold height for Boolean functions). Then the \textit{threshold location} $\sad_A$ is the (random) edge $e$ such that $X_e=\thr_A$; see Figure~\ref{fig:Bernou}. Our sharp threshold criterion states that `sharp thresholds are equivalent to the delocalisation of the threshold location', where the latter means that $\max_e\mathbb{P}[\sad_A = e] \ll 1$. 

\smallskip
We derive this criterion by adapting works of Chatterjee \cite{ch08, cha14} on the `superconcentration' of the maximum of a Gaussian vector. The proof relies on the hypercontractivity of the Ornstein--Uhlenbeck semigroup (just like the BKKKL theorem relies on hypercontractivity for the Boolean hypercube) and is robust enough to apply to strongly correlated Gaussian fields; note that this is the only place in the proof where we use Gaussianity.

\smallskip
In the context of Bernoulli percolation, the criterion is a consequence of Talagrand's inequality \cite{tal94,cel12} applied to $\thr_A$, and one can prove that
\[  \text{Var}(\thr_A) \leq c / \log( \max_e \Pro [ \sad_A =e ] ), \]
which is the analogue of Theorem \ref{t:varcon} below.

\smallskip
Talagrand's inequality was used in \cite{ri19} to prove a sharp threshold inequality for Gaussian fields, also using the notion of threshold location $\sad_A$. However, the sharp threshold inequality from \cite{ri19} was proven in a more restrictive Gaussian setting (see Remark \ref{rk:compare} for more details) and only for transitive events (and one had to use the FKG inequality to deduce sharp threshold results for more general percolation events).
\end{itemize}

To use our sharp threshold criterion (described in (3) above), we establish the delocalisation of the threshold location for crossing events on large scales using only ergodicity and $D_4$-symmetry; in the bulk we use a variant of the Burton--Keane argument and on the boundary a very general argument of Harris \cite{ha60}. While this is sufficient for qualitative delocalisation, to obtain a quantitative result we exploit rotational invariance; this is the only place in the proof that isotropy is needed. The upshot is that we obtain a sharp threshold result \textit{without locating the level of the threshold} (except for some special symmetric events such as square crossings).

\begin{figure}[h!]
\centering
\includegraphics[scale=0.8]{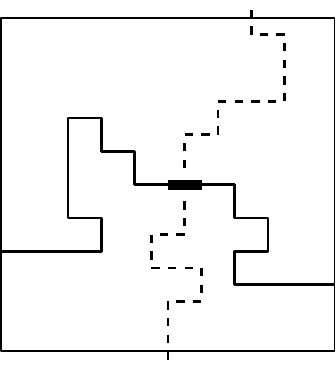}
\caption{The threshold location $\sad_A$ for the left-right crossing of the square is the bold edge; it satisfies $X_{\sad_A}=\thr_A$, the edges in full line satisfy $X_e<\thr_A$, and the edges in dashed line are dual to edges satisfying $X_e > \thr_A$. Note that there is a macroscopic `discrete saddle point' at $\sad_A$; this is reminiscent of a `pivotal point' but with two important differences: (i) the threshold location is only defined for a coupling over all levels while a pivotal point is defined at a fixed level; (ii) there is only \textit{one} threshold location but there may be many pivotal points; this suggests it is easier to control $\max_e \Pro [ \sad_A=e ]$ than the analogue $\max_{e,\ell} \text{Inf}^\ell_e(A)$ for pivotal points, and indeed we achieve this using simple arguments that do not seem to apply to $\text{Inf}^\ell_e(A)$, see Section \ref{ss:4arm}.\label{fig:Bernou}}
\end{figure}

\subsubsection{Sprinkling instead of FKG in gluing arguments}\label{sss:spink} The `gluing' arguments in classical RSW theory are ultimately based on the following elementary observation: for any two continuous paths on the plane that intersect each other, the union of these paths contains a path joining any pair of their four endpoints. As a result, for many crossing events $A$ and $B$ there is a third crossing event $C$ of interest such that $A\cap B \implies C$. If the FKG inequality is available then
\begin{equation}
\label{e:fkgex}
\prob[C]\geq\prob[A\cap B]\geq\prob[A]\prob[B] .
\end{equation}
Often \eqref{e:fkgex} is only used to say that if $A$ and $B$ are not negligible, then $C$ is also non-negligible. RSW theory, later enhanced by Tassion \cite{tas16} (see also Appendix \ref{a:rsw} by K\"ohler-Schindler), uses subtle combinations of this elementary observation to prove that crossing events at large scales are non-degenerate at level $\ell=0$. The classical theory relies on $D_4$-symmetry, FKG, and independence. During the elaboration of the present paper, K\"ohler-Schindler and Tassion \cite{kt20} have removed the independence assumption, see Section \ref{s:open} for more details and for connections to our work.

\smallskip
In order to dispense with the FKG inequality, we replace it by \textit{sprinkling} and the \textit{union bound}. Suppose that at a level $\ell$ crossing events $A$ and $B$ have probability bounded from below. By using the sharp threshold theorem and slightly increasing the level (known as `sprinkling'), the events $A$ and $B$ become very likely, so the union bound implies this is also the case for $A \cap B$, and so also for the event $C$. Unfortunately, the use of sprinkling prevents us from proving results at the critical parameter; in particular, our methods do not prove that there is `no percolation at criticality' (see Conjecture \ref{c:crit}). We note that this step of the argument is very robust, as it uses only $D_4$ symmetry and self-duality at $\ell=0$ (as well as the existence of sharp thresholds for crossing events).

\subsubsection{Constructing the infinite cluster}
At this point we are able to deduce the absence of percolation at $\ell < 0$, however we still need to construct the infinite cluster for $\ell > 0$. The classical approach (step (c) above) consists of using bootstrapping to deduce that the rectangles are crossed with probability converging exponentially to $1$ in the scale; it is then easy to construct an infinite cluster by gluing dyadic rectangles. Since bootstrapping arguments are not available in our setting (due to a lack of spatial independence or the domain Markov property; see however Remark \ref{r:bootstrap}), we instead rely on a quantitative version of the sharp threshold result which exploits isotropy.

\subsection{Open problems and conjectures}\label{s:open}

\subsubsection{The absence of percolation at criticality}

The most fundamental question still to be answered for planar Gaussian percolation is whether the \textit{nodal lines} (i.e.\ the $\ell = 0$ level lines) are bounded (equivalent to the boundedness of $\{ f \geq 0 \}$). This is known in the case that $\kappa \ge 0$ \cite{al96}, but not for several important examples such as the RPW. 
  
\begin{conjecture}
\label{c:crit}
Let $f$ be a Gaussian field satisfying Assumption \ref{as:main} that is $D_4$-symmetric (e.g.\ $f$ is the RPW). Then $\{f = 0\}$ has no unbounded connected components.
\end{conjecture}

The difficulty in proving this conjecture is the absence of a finite size criterion for percolation. As explained in Section \ref{ss:strategy}, our methods rely crucially on sprinkling (i.e.\ small increases of the level $\ell$), so they are unable to provide such a criterion. The situation is similar to Bernoulli percolation on $\Z^3$, where the absence of percolation at criticality is a fundamental open question, and where the use of sprinkling (for instance in \cite{gm90}) is an obstacle to proving a finite size criterion.

\subsubsection{Weakening the assumptions}

We do not believe that all the assumptions in Theorem~\ref{t:main} are necessary, and it would be interesting to know the extent to which they can be relaxed. 

\smallskip
In regards to symmetry, although $D_4$-symmetry is crucial to our arguments, isotropy is only used at one point in the proof (see also Remark \ref{r:bootstrap}) and it would be nice to remove it. Similarly, one can ask whether the correlation decay assumption \eqref{e:decaymain} can be relaxed or whether there might be counterexamples to \eqref{e:pt2} if $\kappa \to 0$ sufficiently slowly.

\begin{question}
Let $f$ be a Gaussian field satisfying Assumption \ref{as:main} that is $D_4$-symmetric. Does \eqref{e:pt2} hold? Does it hold even if the assumption of $D_4$-symmetry is removed? What about if the assumption $\kappa(x) \to 0$ is weakened to ergodicity?
\end{question}

Finally, it would be interesting to have an analogue of Theorem \ref{t:main} for rough fields.

\begin{question}
Let $f$ be a Gaussian field satisfying the second and third conditions of Assumption \ref{as:main}, and suppose that $f$ is almost surely of H\"older class $C^\nu$ for some $\nu\in (0,1)$. Are the conclusions of Theorem \ref{t:main} still true?
\end{question}

 \subsubsection{Exponential decay of crossing probabilities}
In Theorem \ref{t:sharp} we prove bounds on the decay of crossing probabilities in the subcritical regime. As explained in Remark \ref{r:bootstrap}, if correlations decay sufficiently rapidly these bounds can be `bootstrapped' to give exponential decay. It would be interesting to know if this holds under weaker conditions. 

\begin{question}
Let $f$ be a Gaussian field satisfying Assumption \ref{as:main} that is $D_4$-symmetric. Under what extra conditions is it true that, for each $\ell < 0$, there exists $c_1, c_2 > 0$ such that
\begin{equation}
\label{e:exp}
\prob\brb{\textup{Cross}_\ell(R,R)} \le c_1 e^{-c_2 R}  \ ? 
\end{equation}
Is \eqref{e:exp} true for the RPW?
\end{question}

\begin{remark}
This question has been partially answered by the first author and Severo in \cite{ms22} (after the first version of the present paper appeared). In particular, they show that, if for every $\alpha>0$ we let $F_\alpha$ be the planar Gaussian field with Cauchy kernel of parameter $\alpha$ (which means that $\E[F_\alpha(0)F_\alpha(x)] = (1+|x|^2)^{-\alpha/2}$), then \eqref{e:exp} holds if and only if $\alpha>1$.
\end{remark}

\subsubsection{Simplifications and extensions of our proofs by using the new proof of RSW by Köhler-Schindler and Tassion}\label{sss:kt}

In \cite{kt20} (that has been written during the elaboration of the present paper), Köhler-Schinlder and Tassion have proven a RSW theorem (see Section \ref{sss:spink}) for arbitrary scales by only assuming symetries and positive association (so in particular without assuming any sort of quasi-independence property). It would be interesting to replace the RSW results used in the present paper (that also come from works by these two authors -- see Propositions \ref{p:gc1} and \ref{p:gc2}) by results from \cite{kt20}. If one manages to do this, then we would obtain analogues of these two propositions \textit{for arbitrary scales} rather than specific sequences $(R_n)_n$, but for a different choice of ``building blocks'' (i.e.\ with domains different from the $\calD (R;a,b)$'s). Possible consequences could be: (a) the simplification of some steps of the proofs from the present paper (e.g.\ in Section \ref{ss:concentration_to_phase_transition}), (b) that Item~1 of Theorem \ref{t:sharp} holds for arbitrary scales and (c) that Corollary~\ref{c:sodin} holds without the isotropy and quantitative decay assumptions (but with the $D_4$-symmetry assumption).

We do not expect the adaptation of the techniques from \cite{kt20} to our context to be easy but it may be tractable. For instance, the following are two new difficulties: (1) In \cite{kt20}, the ``building blocks'' of the RSW arguments are crossings from boundary of a rectangle to a square included in this rectangle. Establishinig the delocalisation of the threshold location would thus require to deal with nodal lines in the $3/4$ - rather than half - plane, where the geometry of nodal lines is more complicated. (2) The proof from \cite{kt20} relies on a cascade argument. As a result, it does not seem clear that one can really obtain an analogue of Propositions \ref{p:gc1} and \ref{p:gc2} with some \textit{bounded} $N$.

\subsubsection{What about higher dimensions?}

We end this section with the following question: What about higher dimensions? We first note that it is expected -- and proven in some cases such as the Bargmann-Fock field \cite{DRRV} -- that $\ell_c<0$ when the dimension is $\ge 3$. However, analogues of Theorem \ref{t:sharp} (with $\ell=0$ replaced by $\ell=\ell_c$) are expected to hold when $d \ge 3$. Concerning our techniques, our general result Theorem~\ref{t:varcon} extends to all dimensions. However, in our geometric arguments, we use both planar (e.g.\ in the Russo--Seymour--Welsh and Harris arguments) and more general (e.g.\ in the Burton--Keane argument) tools.

\subsection{Outline of the paper}
In Section~\ref{s:proof} we implement the strategy described in Section~\ref{ss:strategy} above. In particular, we state our result that `sharp thresholds are equivalent to delocalisation of the threshold location' (see Theorem \ref{t:varcon}) and we prove the main results of the paper assuming this result and that the threshold location delocalises for a certain class of crossing events. In Section \ref{s:con} we prove the sharp threshold result, and in Section \ref{s:deloc} we prove the delocalisation of the threshold location. The appendix contains auxiliary results on Gaussian fields and Morse functions, and also includes a section written by Laurin K\"{o}hler-Schindler, containing his work on RSW theory.

\subsection{Acknowledgments} 
We are grateful to Laurin K\"{o}hler-Schindler for sharing his work on RSW theory with us, and for kindly agreeing to write Appendix \ref{a:rsw}. We also thank Vincent Tassion for general discussions about quantitative RSW theory, Michael McAuley and Jeff Steif for help with references, Matthis Lehmkühler for help with ergodic theory, Gábor Pete for interesting discussions about superconcentration theory and Thomas Letendre for providing the proof of Lemma \ref{lem:as_main_implies_cond} to us. Finally, we wish to thank an anonymous referee for helpful comments.

\medskip
\section{Proof of the main results}
\label{s:proof}

In this section we give the proof of Theorems \ref{t:main} and \ref{t:sharp} assuming intermediate statements on the existence of sharp thresholds; these statements are proven in the following two sections. We follow the general strategy described in Section \ref{ss:strategy} above. We shall assume throughout this section that the conditions in Assumption \ref{as:main} hold (although some intermediate results do not require all the conditions).

\subsection{Crossing domains and the threshold map}\label{ss:threshold_map}

As described in Section \ref{ss:strategy}, our study of the phase transition rests on an analysis of the \textit{threshold height} and \textit{threshold location}, and we begin by making these concepts precise.

\begin{definition}[Three stratified domains]\label{def:strat_domain}
In this paper a \textit{stratified domain} will be a couple $(D,\calF)$, where $D\subset\R^2$ is a compact domain and $\calF$ is a finite partition of $D$, in one of the following three cases:
\begin{itemize}
\item The set $D$ is a closed rectangle. The partition $\calF$ consists of (i) the interior of $D$, (ii) a finite number of open intervals of the smooth part of $\partial D$, and (iii) a finite number of points (which must include the corners of $D$).
\item The set $D$ is a closed annulus $\text{Ann}(a,b) = \{a \le |x| \le b\}$ for some $0 < a < b$. The partition $\calF$ consists of the interior of $D$ and the two connected components of $\partial D$.
\item The set $D$ is the Euclidean ball $B_0(R)=\{ |x| \leq R \}$ for some $R>0$. The partition $\calF$ consists of the interior of $D$ and the boundary $\partial D$.
\end{itemize}
The elements of $\calF$ will be called faces; notice that they are all smooth submanifolds of $\R^2$. For technical reasons, we want to consider functions defined on a neighborhood of $D$. So for each $D$ we (arbitrarily) fix $D^{++} \supsetneq D^+ \supsetneq D$ be two compact sets with smooth boundary whose interior contains $D$. 
\end{definition}

The third case above will not be used until Section \ref{s:deloc} so the reader can ignore it for the moment. In the two first cases, we define a \textit{crossing domain} as follows:

\begin{definition}[Crossing domains]\label{def:conn_domain}
A \textit{crossing domain} is a triple $\mathfrak{D}=(D,\calF,A)$, where $(D,\calF)$ is a stratified domain with $D$ either a rectangle or an annulus, and $A \subset C^0(D^+)$ is defined as:
\begin{itemize}
\item If $D$ is a rectangle, fix $S_0,S_2\subset \partial D$ both homeomorphic to non-empty open intervals\footnote{In particular, we allow $S_0$ and $S_2$ to `wrap around corners' of $D$ whenever it is a rectangle. The reason for this notation is that we will later denote by $S_1$ and $S_3$ the two connected components of $\partial D \setminus (\overline{S}_0 \cup \overline{S}_2)$.} which are unions of elements of $\calF$ such that $\overline{S_0} \cap \overline{S_2} = \emptyset$. The sets $S_0,S_2$ are called the \textit{distinguished sides}, and a continuous function $u:D^+\rightarrow\R$ belongs to $A$ if and only if there exists a path $\gamma:[0,1]\rightarrow D\cap\{u\geq 0\}$ such that $\gamma(0)\in \overline{S}_0$ and $\gamma(1)\in \overline{S}_2$.
\item If $D$ is an annulus, a continuous function $u:D^+\rightarrow\R$ belongs to $A$ if and only if there exists a circuit in $D\cap\{u \geq 0\}$ that separates the inner disc from infinity.
\end{itemize}
Note that the set $A$ is increasing in the sense that if $u\in A$ and $v\in C^0(D^+)$ is non-negative then $u+v\in A$.
\end{definition}

\begin{remark}\label{r:top}
The results of this subsection, as well as Theorem \ref{t:varcon} and Proposition \ref{p:expcon} below, actually hold in the more general setting in which $(D,\calF)$ is a stratified set in $\R^d$, $d \ge 1$, and $A$ is an increasing topological event on $D$ (in the sense of \cite{bmr20}) with roughly the same proof. Moreover, they also hold for discrete models such as Bernoulli percolation (see Section \ref{ss:strategy}) or more generally Gaussian vectors on Euclidean lattices with non-degenerate covariance matrix. In the latter case, analogues of Theorem \ref{t:varcon} and Proposition \ref{p:expcon} below hold for \textit{any} increasing event that depends on the sign of the coordinates.
\end{remark}

We next define the \textit{threshold map} (relative to a crossing domain) whose two components are the \textit{threshold height} and the \textit{threshold location}. To define the first component we consider a function $u\in C^0(\R^2)$ (later we will substitute realisations of $f$ for $u$). 

\begin{definition}[The threshold height]
Let $\mathfrak{D}=(D,\calF,A)$ be a crossing domain and $u \in C^0(D^+)$. Since $A$ is increasing and $D$ is compact, the set of $\ell\in\R$ such that $u+\ell\in A$ is an interval of the form $[\thr_A(u),\infty)$ or $(\thr_A(u),\infty)$ where $\thr_A(u)  \in \R$ is finite (generically the interval is closed, see Appendix \ref{sec:app_morse}, but we will not use this property in the present section). We call $\thr_A(u)$ the \textit{threshold height} of $u$ relative to $A$.
\end{definition}

The second component of the threshold map is not well defined for arbitrary $u \in C^0(D^+)$, and so we first restrict ourselves to the generic class of perfect Morse functions.

\begin{definition}
Consider a stratified domain $(D,\calF)$ and a function $u \in C^1(D^+)$. For each $F \in \mathcal{F}$ and each $x \in F$, we say that $x$ is a \textit{stratified critical point} if it is a critical point of $u_{|F}$ (i.e.\ $\nabla_x(u_{|F})=0$). If $F$ is a point, the convention is that this point is always a stratified critical point. If $x$ is a stratified critical point, $u(x)$ is called a \textit{stratified critical value} of $u$.
\end{definition}

\begin{definition}[Perfect Morse functions]\label{d:morse}
Consider a stratified domain $(D,\calF)$. We say that $u\in C^2(D^+)$ is a $(D,\calF)$-perfect Morse function if:
\begin{itemize}
\item If $x_1 \neq x_2$ are distinct stratified critical points then $u(x_1) \neq u(x_2)$.
\item For each $F\in\calF$, the critical points of $u|_F$ are non-degenerate (i.e.\ the Hessian is invertible at every critical point).
\item For each $F,F'\in\calF$ such that $F\neq F'$ and $F\subset\overline{F'}$ (i.e.\ $F \subset \partial F'$ since the faces are disjoint) and for each $x\in F$, we have $\nabla_x (u_{|\overline{F'}})\neq0$.\footnote{In several proofs, it is sufficient to say that, if $\nabla_x (u_{|\overline{F'}}) = 0$, then it is non-degenerate (as a critical point of $u_{|F'}$). However, since the fields considered in this paper also satisfy this third item -- and to be consistent with the literature about Morse functions -- we have chosen to always exclude all such critical points rather than just degenerate ones.} This last boils down to asking that a point of a face of dimension less than $2$ cannot be critical if seen as a critical point of a face of larger dimension.
\end{itemize}
Let $\calM(D,\calF)$ be the set of $(D,\calF)$-perfect Morse functions; we note that it is an open subset of $C^2(D^+)$ (Lemma \ref{lem:morse_is_open}) and we equip it with the $C^2$-topology.
\end{definition}

In the rest of the subsection we work with a fixed crossing domain $\mathfrak{D}=(D,\calF,A)$. Within the function class $\calM(D,\calF)$ the second component of the threshold map is defined as follows:

\begin{definition}[The threshold location]
Let $u \in \calM(D,\calF)$ and suppose that $-\ell\in\R$ is not a stratified critical value for $u$. Then the level set $\{u+\ell=0\}$ is the intersection of $D$ with a $C^1$-smooth manifold that intersects each $F\in\calF$ transversally. Thus, for each $t>0$ small enough, $\{u+\ell+t \geq 0\}$ isotopically retracts to $\{u+\ell-t \geq 0\}$ in a way that preserves the sets $F\in\calF$.\footnote{By this, we mean that there exists a continuous map $H : [0,1] \times D \rightarrow D$ such that (a) $H(0,\cdot)=id_D$, (b) $\forall s, H(s,\cdot)$ is a homeomorphism, (c) $H(s,F)=F$ for every $F \in \calF$ and every $s$, and (d)
\[
H\big(1, \{u+\ell+t \ge 0\}\cap D\big)= \{u + \ell - t \ge 0\}\cap D.
\]
One can prove that such an $H$ exists by using locally the implicit function theorem. For more about operations from stratified Morse theory, see \cite[e.g.\ Section 3.2 of Part I]{morse}.} In particular, $u+\ell+t\in A$ if and only if $u+\ell-t\in A$ and so $\thr_A(u)\neq -\ell$. We conclude that there exists a unique $x \in D$ that is a stratified critical point with $u(x)=-\thr_A(u)$. We denote this point by $\sad_A(u)$. This defines a map
\[\sad_A:u \in \calM(D,\calF) \mapsto \sad_A(u) . \]
We call the stratified critical point $\sad_A(u)$ the \textit{threshold location}.
\end{definition}

\begin{remark}
As we will see (see Lemma \ref{lem:saddles} and Figure \ref{fig:saddles}), and as we already saw in the discrete case in Section \ref{ss:strategy}, the threshold location $\sad_A$ is not only a `local critical point' but a `global saddle point'.
\end{remark}

All in all, the \textit{threshold map} is the map $\calM(D,\calF) \rightarrow \R\times D$ defined by 
\[ u \mapsto(\thr_A(u),\sad_A(u)) . \]
Henceforth we shall view $\thr_A$ and $\sad_A$ as random variables by evaluating them on $u = f$. The following lemma verifies that this is well-defined:

\begin{lemma}
\label{l:morse}
Let $f$ be a Gaussian field satisfying Assumption \ref{as:main}. Then $f  \in \calM(D,\calF)$ almost surely, and so in particular $\sad_A(f)$ is well-defined almost surely. Moreover, $(\thr_A(f),\sad_A(f)) \in \R \times D$ is measurable.
\end{lemma}
\begin{proof}
The first statement is a consequence of Lemmas \ref{lem:cond_implies_morse} and \ref{lem:as_main_implies_cond} proven in the appendix. The measurability of $(\thr_A, \sad_A)$ is implied by the continuity of $\thr_A$ and $\sad_A$ in the $C^2$-topology, see Lemma~\ref{lem:diff_of_threshold} below.
\end{proof}

Let us make a link between crossing domains and the crossing events $\Cross_\ell(a,b)$ defined in Definition~\ref{d:conn}. First we define, similarly to $\Cross_\ell(a,b)$, the `annular circuit' event $\Circ_\ell(a,b)$, $0 < a < b$, that there exists a circuit in $\{ f \geq -\ell \} \cap \Ann(a,b)$ that separates the inner disc of $\Ann(a,b)$ from infinity. Next we observe that to every crossing domain $\mathfrak{D} = (D, \calF, A)$ we can associate a family of events, indexed by levels $\ell \in \R$, via
\begin{equation}
\label{e:crossA}
\Cross_\ell(A) := \{f + \ell \in A \}\,  .
\end{equation}
Then the crossing events $\Cross_\ell(a,b)$ (from Definition \ref{d:conn}) and $\Circ_\ell(a, b)$ are both of the form $\Cross_\ell(A)$ for some choice of crossing domain (for which $D$ is, respectively, a rectangle and an annulus), and that the threshold height $\thr_A$ has the property that
\[ \{f \in \Cross_\ell(A)  \}  ,  \ \text{ if } \ell > \thr_A,  \qquad \text{and} \qquad  \{f \notin \Cross_\ell(A)  \}  ,  \  \text{ if } \ell <  \thr_A  .  \]

\begin{remark}
Since $\sad_A$ is almost surely a stratified critical point of $f$, and critical points of stationary fields have density with respect to the Lebesgue measure, it can be seen that $\sad_A$ has density with respect to the sum of the Lebesgue measures on the elements of $\calF$, and in fact, the joint law of $(\sad_A, \thr_A)$ has a density with respect to this measure. Since we will not need this fact we do not prove it rigorously.
\end{remark}

\subsection{Sharp thresholds are equivalent to delocalisation of the threshold location}\label{ss:delocalization_is_concentration}

As explained in Section \ref{ss:strategy}, we aim to prove that `sharp thresholds are equivalent to the delocalisation of the threshold location'; this is inspired by works of Chatterjee \cite{ch08, cha14} who demonstrated a similar phenomenon for the maximum of a Gaussian vector. We now state a precise version of this equivalence. 

\smallskip
Let us first quantify the notion of the threshold location being delocalised. Again we fix a crossing domain $\mathfrak{D} = (D,\calF,A)$ (see Definition \ref{def:conn_domain}) for the rest of the subsection. Recall that $B_x(r)$ denotes the ball of radius $r$ centred at $x$. For $r > 0$ let
\begin{equation}\label{eq:sigma}
\sigma_A(r) := \sup_{x \in \R^2} \prob\brb{S_A \in B_x(r)}
\end{equation}
to be the maximal probability, over all balls of radius $r$, that the threshold location $S_A$ lies in this ball. Note that $r \mapsto\sigma_A(r)$ is non-decreasing.

\begin{theorem}[Sharp thresholds are equivalent to delocalisation]
\label{t:varcon}
Let $f$ be a Gaussian field satisfying Assumption \ref{as:main}, and recall that $\overline \kappa(r) = \sup_{|x| \ge r} |\kappa(x)|$. There is a constant $c = c(\kappa) > 0$ depending only on $\kappa$ such that
\[  \textup{Var}(\thr_A) \le c \inf_{r > 0} \overline{M}(r) , \]
where
\begin{equation}
\label{e:overM}
\overline{M}(r) =  \max \big\{\overline{\kappa}(r),|\log( \sigma_A(r) )|^{-1} \big\} .
\end{equation}
Moreover, if $D$ is a rectangle, we also have
\[c^{-1} \sup_{r\geq e} \underbar{M}(r)\leq \textup{Var}(\thr_A)\]
where
\begin{equation}\label{eq:anticon}
\underbar{M}(r) = \sigma_A(r)^3 r^{-4} \big(\max \big\{ \log r,|\log(\sigma_A(r))| \big\} \big)^{-2}\, .
\end{equation}
\end{theorem}

We expect that the lower bound \eqref{eq:anticon} also holds for annuli under some conditions on their inner radius. In any case, the lower bound will not be used elsewhere in the paper. 

\smallskip
Since we work under the assumption that $\kappa(x) \to 0$, we deduce the following corollary:

\begin{corollary} 
\label{c:con}
Consider a sequence $(D_R,\calF_R,A_R)_{R>0}$ of crossing domains. Then,
\[ \big(  \forall r  > 0 , \  \lim_{R \to \infty}  \sigma_{A_R}(r)=0  \big) \quad \Leftrightarrow \quad \big(  \exists r  > 0 , \  \lim_{R \to \infty}  \sigma_{A_R}(r)=0  \big) \quad \Rightarrow  \quad   \lim_{R \to \infty}\textup{Var}(\thr_{A_R}) \to 0    . \]
If furthermore each $D_R$ is a rectangle then
\[ \big(  \forall r  > 0 , \  \lim_{R \to \infty}  \sigma_{A_R}(r)=0  \big) \quad \Leftrightarrow \quad \big(  \exists r  > 0 , \  \lim_{R \to \infty}  \sigma_{A_R}(r)=0  \big) \quad \Leftrightarrow \quad   \lim_{R \to \infty}\textup{Var}(\thr_{A_R}) \to 0    . \]
\end{corollary}
\begin{proof}
The first equivalence comes from the monotonicity of $r \mapsto \sigma_A(r)$ and the fact that, by the union bound, $\sigma_A(R) \leq c(R/r)^2\sigma_A(r)$ for all $R \geq r>0$ and a universal constant $c>0$. The implication (or the second equivalence) is a consequence of Theorem \ref{t:varcon}.
\end{proof}

\begin{remark}\label{rk:compare}
Some inequalities on $\textup{Var}(\thr_A)$ were proven in \cite{ri19}, also using the notion of threshold location $\sad_A$. However, the inequalities from \cite{ri19} were only proven for transitive events and for Gaussian fields with fast decay of correlation and an underlying white noise product space.

Let us more generally point out one important difference between Theorem \ref{t:varcon} and the abstract sharp threshold results used in previous works on Gaussian field percolation \cite{rv20, mv20, ri19, gv19}, which were respectively based on the BKKKL, OSSS, Talagrand, and Schramm--Steif inequalities. The previous approaches suffered from one of two disadvantages -- either the abstract threshold results were applied to a discretisation of the model (as in \cite{rv20, mv20}), or they were applied to a `white noise product space' that generates the model (as in \cite{mv20, ri19, gv19}) -- which restricted their applicability to special classes of Gaussian fields. Moreover, the application of these sharp threshold results to (general, non-transitive) percolation events required the FKG inequality. The sharp threshold result in Theorem \ref{t:varcon} is both continuous and `coordinate free', applies naturally to all Gaussian fields, and as we will see in Section \ref{s:deloc}, its application to general percolation events does not require the FKG inequality.
\end{remark}

While Corollary \ref{c:con} suffices to prove absence of percolation in the subcritical regime $\ell < 0$, to study the supercritical regime $\ell > 0$ we need a certain `large deviation' extension of Theorem~\ref{t:varcon} (inspired by \cite{tan15}, which gave the analogous result for the maximum of a Gaussian vector):

\begin{proposition}
\label{p:expcon}
There is a constant $c = c(\kappa) > 0$ depending only on $\kappa$ such that, for $t > 0$,
\[ \prob\brb{|\thr_A -\E[\thr_A]| \ge t} \le \inf_{r > 0} \, 6 e^{-ct/\sqrt{\overline{M}(r)}}  \]
where $\overline{M}(r)$ is defined as in \eqref{e:overM}.
\end{proposition}

We prove Theorem \ref{t:varcon} and Proposition \ref{p:expcon} in Section \ref{s:con}. Although specific to the Gaussian setting, in essence these results rely on the hypercontractivity of the Ornstein--Uhlenbeck semigroup, and as such we expect that similar results may hold for other models to which can be associated natural hypercontractive dynamics. 

\subsection{The threshold location delocalises}\label{ss:threshold_delocalizes} 

To deduce a sharp threshold result for crossing events it remains to show that the threshold location $\sad_A$ delocalises for crossing domains on large scales.

\begin{proposition}[Delocalisation of the threshold location]
\label{prop:delocalization}
Let $\mathfrak{D}=(D,\calF,A)$ be a crossing domain, and recall the definition of $\sigma_A$ in \eqref{eq:sigma}.
\begin{itemize}
\item There exists a positive function $\eta(x) \to 0$ as $x \to \infty$, depending only on the field, such that the following holds. Suppose that $D$ is a rectangle and $f$ is $D_4$-symmetric. Then, for all $r \geq 1$,
\begin{equation}
\label{e:sigrec}
\sigma_A(r) \le r^2 \eta(d_0) ,
\end{equation}
where $d_0$ is the minimum among the distance between $S_0$ and $S_2$, and the distance between $S_1$ and $S_3$, where $S_0$ and $S_2$ are the two distinguished sides of the crossing domain and $S_1$ and $S_3$ are the two components of $\partial D \setminus (\overline{S}_0 \cup \overline{S}_2)$, see Figure~\ref{fig:faces}.
\item Suppose $D$ is an annulus and $f$ is isotropic. Then there exists a universal constant $c > 0$ such that, for all $r >0$, 
\begin{equation}
\label{e:sigann}
 \sigma_A(r) \le  cr d_0^{-1},
 \end{equation}
where $d_0$ is the inner radius of the annulus.
\end{itemize}
\end{proposition}

\begin{figure}[h!]
\centering
\includegraphics[scale=0.4]{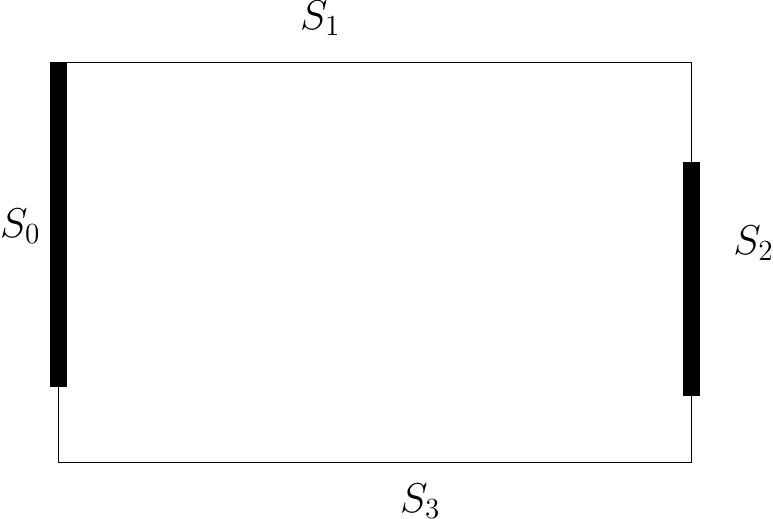}
\caption{The distinguished sides $S_0$ and $S_2$ and the sets $S_1$ and $S_3$ (which are unions of $0$-dimensional and $1$-dimensional faces). In this example $d_0$ is the distance between $S_1$ and $S_3$, which is also the length of $S_2$.\label{fig:faces}}
\end{figure}

Since the proof of \eqref{e:sigann} is short, we prove it immediately:

\begin{proof}[Proof of \eqref{e:sigann}]
If $r>d_0/100$ then the result follows by taking $c$ sufficiently large. Assume that $r \leq d_0/100$ and let $x \in \R^2$ be such that $B_x(r)$ intersects the annulus. Next, note that there exist $N>d_0/(100r)$ disjoint balls $B_{x_1}(r), \ldots, B_{x_N}(r)$ such that $|x_i| = |x|$. By rotational invariance of both $f$ and the event $A$, we have
\[ \forall i, \,  \Pro [ \sad_A \in B_{x_i}(r) ] =  \Pro [ \sad_A \in B_x(r) ] . \]
Since the events $\{ \sad_A \in B_{x_i}(r) \}$ are disjoint, we have $\Pro [ \sad_A \in B_x(r) ] \leq 1/N < 100r/d_0$.
\end{proof}

The proof of \eqref{e:sigrec} is more complicated and we defer it to Section \ref{s:deloc}. The proof relies on the following two claims, which may be of independent interest since they do not depend on the Gaussian setting:
\begin{enumerate}
\item Almost surely the field $f$ has no saddle point whose four `arms' (level lines at the level of the saddle point) connect the saddle point to infinity; see Corollary \ref{cor:4arm}.
\item Suppose $f$ is $D_4$-symmetric and let $H = \{(x,y) : y \ge 0\}$ be the half-plane. Then almost surely the field $f|_H$ has no unbounded level lines that intersect $\partial H$; see Corollary~\ref{cor:positiveinhalfplane}.
\end{enumerate}

Note that, unlike the previous steps of the proof (see Remark \ref{r:top}), Proposition \ref{prop:delocalization} is specific to the class of crossing domains of Definition \ref{def:conn_domain}, and does not extend immediately to a more general class of increasing `topological' events. 

\subsection{From sharp thresholds to the phase transition}
\label{ss:concentration_to_phase_transition}

To complete the proof of Theorems \ref{t:main} and \ref{t:sharp} we adapt the RSW theory of Tassion \cite{tas16}, and its extension provided by K\"ohler-Schindler in Appendix \ref{a:rsw}, by replacing the FKG inequality with a sprinkling procedure.

\smallskip
Let us introduce the RSW theory, beginning with notation for the crossing domains that are the `building blocks' of the theory. For $R > 0$ and $0 \le a < b \le R$, let $\mathfrak{D}(R; a, b)$ be the crossing domain $(D(R), \mathcal{F}(R; a, b), A(R; a, b))$ defined by applying the first point of Definition \ref{def:conn_domain} to the square $D(R) = [0,R]^2$ with distinguished sides $S_0 = \{0\} \times (0, R)$ and $S_2 = \{R\} \times (a, b)$. For each level $\ell \in \R$, the associated crossing event $\Cross_\ell(A)$ is the event that $\{f \geq -\ell\}$ contains a path inside the square $[0, R]^2$ that intersects both the left-hand side $\overline{S}_0$ and the subinterval of the right-hand side $\overline{S}_2$ (see Figure \ref{f:hevent}).

\begin{figure}[h!]
\centering
\includegraphics[scale=0.4]{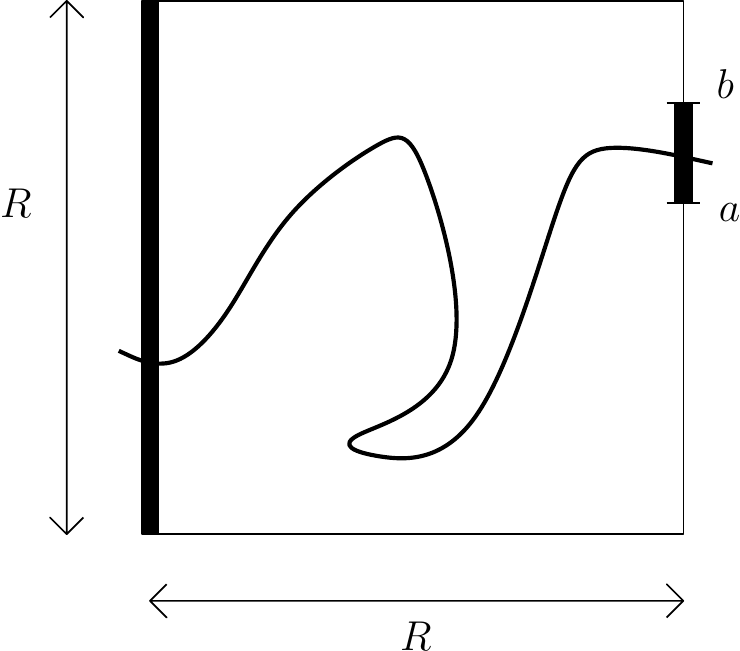}
\caption{The event $\Cross_\ell(A)$ for $A = A(R;a,b)$.
\label{f:hevent}}
\end{figure}

\smallskip
For the remainder of this section we assume that the field $f$ is $D_4$-symmetric; by a simple symmetry argument this guarantees that, at level $\ell = 0$, squares are crossed with probability exactly $1/2$ (see Lemma \ref{l:sc}). The RSW theory in \cite{tas16} rests on the following geometric construction that states that, on an unbounded sequence of `good scales' $(R_n)_{n \ge 1}$, crossings of a long rectangle $\Cross_\ell(\lambda R_n, R_n)$ are implied by events $\Cross_\ell(A)$ for $A$ of the form $A(R_n; a, b)$ or $A(3R_n/4;a,b)$. 

\begin{proposition}[See \cite{tas16}]
\label{p:gc1}
For every $\lambda > 0$ there exists a sequence of scales $R_n \to \infty$ as $n\to\infty$, a constant $N \in \N$, and a sequence $\alpha_n \to \infty$ as $n\to \infty$, such that the following holds for each $n \in \N$. There is a set of $N$ crossing domains $(\mathfrak{D}^n_i)_{i \le N} = ((D^n_i, \calF^n_i, A^n_i))_{i \le N}$, that are translations, rotations by $\pi/2$, and reflections in the vertical and horizontal axes, of crossing domains in the collection
\[
\bigcup_{r \in \{ R_n, 3R_n /4\}, b - a\ge \alpha_n} \mathfrak{D}(r; a, b)
 \]
 such that:
\begin{itemize}
\item For each $\ell \in \R$,
\[    \bigcap_{i \le N} \Cross_\ell(A^n_i) \subset \Cross_\ell(\lambda R_n, R_n)  ; \]
\item For each $i \in \{1,\ldots,N\}$, $\mathbb{P}[ \Cross_0(A^n_i) ] \ge 1/8$.
\end{itemize}
\end{proposition}

The proof of the above statement is essentially contained in \cite{tas16}, although we add an extra ingredient to prove that $b - a$ can be chosen to exceed a quantity $\alpha_n \rightarrow \infty$ (which relies on our arguments in Section \ref{s:deloc}, specifically Corollary \ref{cor:positiveinhalfplane}). For completeness, a proof of Proposition \ref{p:gc1} is included in Appendix \ref{a:rsw} (written by L.\ K\"{o}hler-Schindler).

\smallskip
The relevance of Proposition \ref{p:gc1} to RSW theory is that, under the assumption of positive associations, the events $\Cross_\ell(A^n_i)$ are positively correlated, so it follows that 
\[ \liminf_{n \to \infty} \prob[ \Cross_0(\lambda R_n, R_n) ] > 0  ;\]
this is the `weak RSW' theorem of \cite{tas16}. In our setting we replace this argument with a sprinkling procedure as described in Section \ref{ss:strategy}.

\smallskip
While Proposition \ref{p:gc1} is all that we need to prove that $\{f \geq -\ell\}$ does not percolate in the subcritical regime $\ell < 0$, to prove percolation in the supercritical regime $\ell > 0$ we need quantitative control on the gaps between the `good scales'  $R_n$ which is not implied by the arguments in \cite{tas16}. This is provided by the following extension of Proposition~\ref{p:gc1}:

\begin{proposition}
\label{p:gc2}
For every $\lambda > 0$ there exists a sequence of scales $R_n \to \infty$ as $n\to\infty$ satisfying, for every $k \ge 1$,
\begin{equation}
\label{e:rn+1}
 R_{n+1} \le   R_n \log^{(k)}(R_n) ,
  \end{equation}
  eventually as $n \to \infty$, a constant $N \in \N$, and a sequence $\alpha_n \to \infty$ as $n\to\infty$ such that the following holds for each $n \in \N$. There is a set of $N$ crossing domains $(\mathfrak{D}^n_i)_{i \le N} = ((D^n_i, \calF^n_i, A^n_i))_{i \le N}$, that are translations, rotations by $\pi/2$, and reflections in the vertical and horizontal axes, of crossing domains in the collection
\[
\bigcup_{r \in [R_n/4,R_n]  , b - a \ge  \alpha_n} \mathfrak{D}(R; a, b)
 \]
 such that:
\begin{itemize}
\item For each $\ell \in \R$,
\[    \bigcap_{i \le N} \Cross_\ell(A^n_i) \subset \Cross_\ell( \lambda R_n,  R_n)  ,  \]
\item For each $i \in \{1,\ldots,N\}$, $\mathbb{P}[ \Cross_0(A^n_i) ] \ge 1/8$.
\end{itemize}
\end{proposition}

Note that in Proposition \ref{p:gc2} we take a union on a whole interval $[R_n/4,R_n]$ while in Proposition \ref{p:gc1} the possible scales of the crossing domains are only $R_n$ and $3R_n/4$. The proof is provided in Appendix \ref{a:rsw} (written by L.\ K\"{o}hler-Schindler).

\begin{remark}
During the elaboration of the present work, Köhler-Schindler and Tassion \cite{kt20} have proven an analogue of Proposition \ref{p:gc2} for arbitray sequences $(R_n)_n$. We refer to Section \ref{sss:kt} for more about this and for connections to the present paper.
\end{remark}

\smallskip
We are now ready to prove the main results of the paper (assuming Corollary \ref{c:con} and Propositions \ref{p:expcon}--\ref{p:gc2}):

\begin{proof}[Proof of Theorem \ref{t:main}]

We consider the subcritical regime $\ell < 0$ and the supercritical regime $\ell > 0$ separately; the proof is simplest in the subcritical regime. 

\smallskip
\noindent \textit{Subcritical regime $\ell < 0$.} Fix a level $\ell'>0$, let $\lambda = 5$ (this choice is mainly for concreteness), and let $(R_n)_{n \ge 1}$ be the unbounded sequence guaranteed to exist by Proposition~\ref{p:gc1}. We first argue that, as $n \to \infty$,
\begin{equation}
\label{e:rncross}
  \prob [ \Cross_{\ell'}(5 R_n, R_n) ] \to 1 .
  \end{equation}
Consider the finite set $(\mathfrak{D}^n_i)_{i \le N}$ of crossing domains in the statement of Proposition \ref{p:gc1} as well as the sequence $\alpha_n \to \infty$. By definition, for all $n \in \N$,
\begin{equation}
\label{e:mainproof1}
\min_i \prob[\Cross_0(A^n_i)]  \ge 1/8. 
\end{equation}
Moreover the crossing domain $\mathfrak{D}_i^n$ is defined via distinguished sides $S_0$ and $S_2$ that satisfy the following: the minimum among the distance between $S_0$ and $S_2$, and the distance between $S_1$ and $S_3$, is at least $\alpha_n \rightarrow \infty$, where as before $S_1$ and $S_3$ are the two components of $\partial D_i^n \setminus ( \overline{S}_0 \cup \overline{S}_2)$. Combining this with Corollary \ref{c:con} and Proposition \ref{prop:delocalization}, we see that the threshold heights for these events are (uniformly) asymptotically concentrated, i.e., as $n \to \infty$,
\begin{equation}
\label{e:mainproof2}
 \max_i \textup{Var}(\thr_{A^n_i})  \to 0 .
 \end{equation}
Combining \eqref{e:mainproof1} and \eqref{e:mainproof2} we deduce that, as $n \to \infty$, 
\[ \min_i \prob[\Cross_{\ell'}(A^n_i) ] \to 1  , \]
(recall the notation \eqref{e:crossA}), and \eqref{e:rncross} then follows from Proposition \ref{p:gc1} and the union bound.

\smallskip
The remainder of the argument is classical. Recall that $\Circ_{\ell'}(a,b)$ denotes the event that there exists a circuit in $\{ f \geq -\ell' \} \cap \Ann(a,b)$ that separates the inner disc of $\Ann(a,b)$ from infinity. Choose $\delta > 0$ sufficiently small so that, by gluing constructions and the union bound, for every $R > 0$ we have
\[ \prob[\Cross_{\ell'}(5 R, R) ] > 1 -\delta  \quad  \Longrightarrow \quad \prob[\Circ_{\ell'}(5R, 10R) ] > 1/2  . \]
Hence we deduce from \eqref{e:rncross} that, as $n \to \infty$, eventually
\[ \prob[\Circ_{\ell'}(5R_n, 10R_n) ]  > 1/2 .\]
Since $R_n$ is unbounded, this implies in particular that every compact domain $D \subset \R^2$ is surrounded by a circuit in $\{f \geq -\ell' \}$ with probability at least $1/2$. By ergodicity (see, e.g., the `box lemma' of \cite{gkr88}), in fact every compact domain $D \subset \R^2$ is surrounded by a circuit in $\{f \geq -\ell'\}$ almost surely, and so $\{f < -\ell' \}$ has only bounded connected components almost surely. Since $f$ and $-f$ are equal in law, we have proven that, for every $\ell < 0$, $ \{ f > -\ell \}$ has only bounded connected components almost surely. Since $\{f > -2\ell \} \subset \{f \geq -\ell \}$, the first statement of the theorem follows.

\smallskip
\noindent \textit{Supercritical regime $\ell > 0$.} Fix $\ell > 0$ and $\lambda = 5$, and define $(R_n)_{n \ge 1}$ to be the sequence guaranteed to exist by Proposition~\ref{p:gc2}. Note that, by taking a subsequence, we can assume that
\begin{equation}\label{eq:increasing_good}
R_{n+1} \geq 2R_n.
\end{equation} 
Repeating the arguments from the subcritical regime with the geometric construction in Proposition \ref{p:gc2} in place of Proposition \ref{p:gc1}, we deduce that, for $n \ge 1$ sufficiently large, 
\begin{equation}
\label{eq:circbound}
 \prob[\Circ_{\ell}(5R_n, 10R_n) ]  > 1/2 . 
 \end{equation}
Let $\mathfrak{D}(R)$ be the crossing domain $(\Ann(5R,10R),\calF(R),A(R))$ where $\calF(R)$ and $A(R)$ are defined as in the second item of Definition \ref{def:conn_domain}. In particular, \eqref{eq:circbound} can be rephrased as
\begin{equation}\label{eq:1/2}
\Pro [ \thr_{A(R_n)} \leq \ell ] > 1/2 .
\end{equation}
By Proposition \ref{p:expcon}, there is a constant $c_1 > 0$ such that, for all~$n$,
\[ \Pro [ | \thr_{A(R_n)} - \E [ \thr_{A(R_n)} ] | > \ell/2 ] < \inf_{r > 0} 6 e^{-c_1\ell / \sqrt{\overline{M}(r)} } \]
where $\overline{M}(r) =  \max\{ \overline{\kappa}(r), |\log (\sigma_{A(R_n)}(r))|^{-1} \}$. Since $f$ is isotropic, by Proposition \ref{prop:delocalization} there is a $c_2> 0$ such that
\[ \sigma_{A(R_n)}(r) < c_2 r R_n^{-1} . \]
Choosing $r = R_n^\nu$ for $\nu = 3/4$ (the precise choice of $\nu \in (0,1)$ is mainly for concreteness but it will be convenient in the proof of Theorem \ref{t:sharp} below) we have in fact
\begin{equation}\label{eq:tanguy_Mbar}
\Pro [ | \thr_{A(R_n)} - \E [ \thr_{A(R_n)} ] | > \ell/2 ] < 6 e^{- c_3 \ell \sqrt{ \min\{ \log R_n,  1/ \overline{\kappa}(R_n^\nu) \} } }
\end{equation}
for some $c_3>0$. We need the following claim:

\begin{claim}\label{cl:3/2}
For all $n \geq 1$ sufficiently large,
\[ \E [ \thr_{A(R_n)} ] \leq 3\ell/2. \]
\end{claim}

\begin{proof}[Proof of Claim \ref{cl:3/2}]
Take $n$ sufficiently large so that \eqref{eq:1/2} holds, and let us assume by contradiction that $\E [ \thr_{A(R_n)} ] > 3\ell/2$. Then by \eqref{eq:1/2} and \eqref{eq:tanguy_Mbar},
\[ 1/2 < \Pro [ \thr_{A(R_n)} \leq \ell ] \leq \Pro [ | \thr_{A(R_n)} - \E [ \thr_{A(R_n)} ] | > \ell/2 ] < 6 e^{- c_3 \ell \sqrt{  \min\{\log R_n,  1/ \overline{\kappa}(R_n^\nu) \} } }, \]
which cannot be true when $n$ is large enough since $R_n \to \infty$ and $\kappa \to 0$.
 \end{proof}

Claim \ref{cl:3/2} and \eqref{eq:tanguy_Mbar} imply that, if $n\geq 1$ is sufficiently large, then
\[ \Pro [ \thr_{A(R_n)} > 2\ell ] <  6 e^{- c_3 \ell \sqrt{ \min\{ \log R_n,  1/ \overline{\kappa}(R_n^\nu) \} } }, \]
which can be rephrased as
\[  \prob [ \Circ_{2\ell}(5R_n, 10R_n) ] > 1 -6 e^{- c_3 \ell \sqrt{  \min\{ \log R_n, 1/ \overline{\kappa}(R_n^\nu) \} } } .   \]
By gluing constructions (see Figure \ref{fig:circ_gluing}), this gives
\begin{equation}
\label{e:quantcross1}
 \prob [ \Cross_{2\ell}(40 R_n, 20R_n) ] > 1 -  c_4 e^{- c_3 \ell \sqrt{ \min\{ \log R_n, 1/ \overline{\kappa}({R_n^\nu}) \} } } .   
 \end{equation}
for some $c_4>0$ and for $n$ sufficiently large. Since we assume that $\kappa(x) ( \log \log |x|)^{2+ \delta} \to 0$, this gives, for all $n \geq 1$ sufficiently large,
\begin{equation}\label{e:quantcross2}
\prob [ \Cross_{2\ell}(40 R_n, 20R_n) ] > 1- e^{- \ell  ( \log \log R_n)^{1+\delta/2} }.  
\end{equation}

\begin{figure}[h!]
\centering
\includegraphics[scale=1]{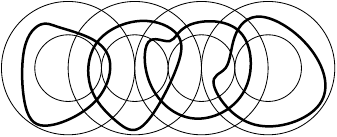}
\caption{Gluing of circuits.
\label{fig:circ_gluing}}
\end{figure}
 
The remainder of the argument is again classical. By gluing constructions, one can interlink $c_5 R_{n+1} / R_n$ translations and rotations by $\pi/2$ of the event $\Cross_{2\ell}(40R_n, 20R_n)$ to create $\Cross_{2\ell}( 20R_{n+1}, 20 R_n )$, for some universal constant $c_5 > 0$. Hence we deduce from \eqref{e:rn+1} (with $k=1$), \eqref{e:quantcross2}, and the union bound that, for sufficiently large $n$,
\begin{align*}
 1 - \prob[\Cross_{2\ell}(20R_{n+1}, 20 R_n ) ] & < \frac{ c_5 R_{n+1}}{R_n} e^{- \ell  ( \log \log R_n)^{1+\delta/2} }\\
& < c_6  \log(R_n) e^{- \ell  ( \log \log R_n)^{1+\delta/2} }\\ 
& <  e^{- (\ell/2)  ( \log \log R_n)^{1+\delta/2} }.  
\end{align*}
Since \eqref{eq:increasing_good} implies that $R_n \ge R_0 2^n$, this gives
\[ \sum_{n \ge 1} (  1 -  \prob[\Cross_{2\ell}( 20R_{n+1} , 20R_n   ) ] ) < \infty,  \]
and so by the Borel--Cantelli lemma the following two events occur for all sufficiently large $n$ almost surely: (i) the event $\Cross_{2\ell}( 20R_{n+1} , 20R_n )$, and (ii) the event that there is a top-bottom crossing of $[0,20R_n] \times [0,20R_{n+1}]$ by a path included in $\{ f \geq -2\ell\}$. Indeed, the latter event is obtained by translation and $\pi/2$-rotation of $\Cross_{2\ell}( 20R_{n+1} , 20R_n )$. This implies the existence of an unbounded connected component of $\{f \geq -2\ell\}$, see Figure \ref{fig:infinite} (in fact the first event for $n$ odd and the second event for $n$ even suffices).

\smallskip
It only remains to prove that the unbounded component is unique. This follows since (by the first part of the proof Theorem \ref{t:main}) almost surely every compact set is surrounded by a circuit included in $\{ f \geq -2\ell \}$, so there cannot be two unbounded components in $\{f \geq -2\ell\}$.
 \end{proof}
 
 \begin{figure}[h!]
\centering
\includegraphics[scale=0.4]{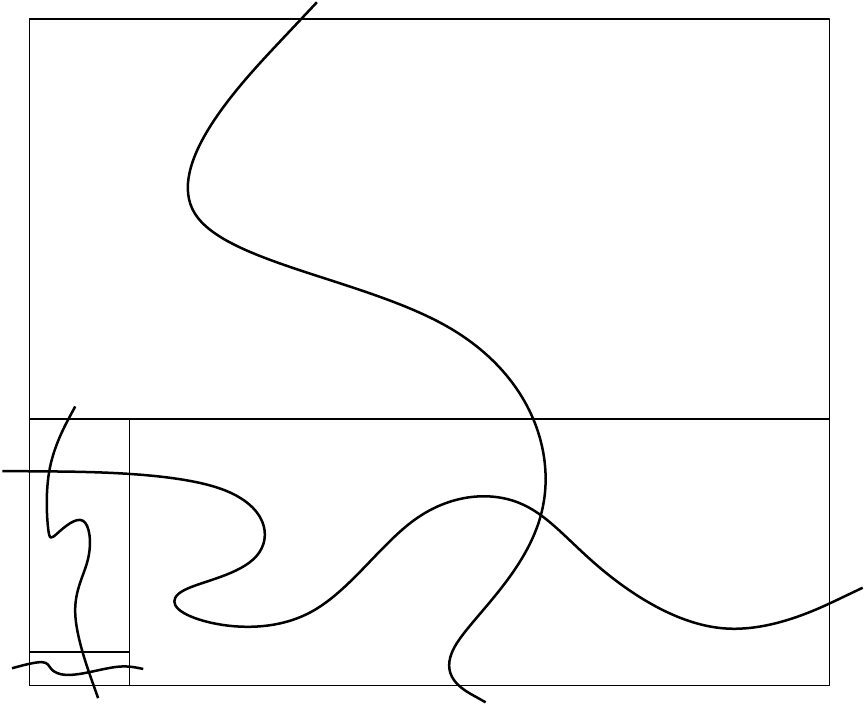}
\caption{Construction of an unbounded component by gluings.
\label{fig:infinite}}
\end{figure}
 
 \begin{proof}[Proof of Theorem \ref{t:sharp}]
In the proof of Theorem \ref{t:main} we showed that, assuming $f$ is $D_4$-symmetric, there exists an unbounded sequence $(R_n)_{n \ge 1}$ such that
 \[ \prob [ \Cross_{\ell'}(5 R_n, R_n) ] \to 1 \]
for every $\ell'>0$. By standard gluing constructions and the union bound (and by using $D_4$-symmetry) this gives that, for any $a>0$, the probability that there is a top-bottom crossing of $[0,R_n] \times [0,aR_n]$ by a path included in $\{ f \geq -\ell' \}$ goes to $1$. Since $-f$ and $f$ have the same law, this implies the first result of Theorem \ref{t:sharp} for any $\ell < -\ell'$.

\smallskip 
For the second result, recall that in \eqref{e:quantcross1} we showed that, for $\ell' > 0$ and $\nu = 3/4$,
\[   \prob[\Cross_{\ell'}(40 R_n, 20R_n) ] > 1 -  c_4 e^{- c_3\ell'  \sqrt{ \min\{  \log R_n, 1/ \overline{\kappa}(R_n^\nu) \} } }   \]
along a subsequence $R_n$ satisfying \eqref{e:rn+1} and \eqref{eq:increasing_good}. Let $R>0$ and $a>0$ and assume that $R$ is sufficiently large so that there exists $m \geq 1$ such that
\[ 20R_{m} \leq R \le 20R_{m+1} \le R_m^{3/2}. \]
Note in particular that
\[ R_m^\nu \ge R^{2 \nu /3} = \sqrt{R} . \]
Again interlinking $c_5 R_{m+1} / R_{m}$ translations and reflections of $\Cross_{\ell'}( 40 R_{m}, 20 R_{m})$ to create $\Cross_{\ell'}(aR, R )$, we deduce that
\begin{align*}
\prob[\Cross_{\ell'}(aR,  R) ] &  >  1 -  \frac{ c_5 R_{m+1}}{R_{m}} e^{-c_3 \ell' \sqrt{ \min\{ \log R_m, 1/ \overline{\kappa}(R_m^\nu) \} }  }  \\
&  >   1 -   \frac{ c_5 R_{m+1}}{R_{m}} e^{-c_6 \ell' \sqrt{ \min\{ \log R, 1/ \overline{\kappa}(\sqrt{R} ) \} }    }  .
\end{align*}
By \eqref{e:rn+1} (and since $|\kappa(x)|\log^{(k)}(|x|) \to 0$ for some $k \ge 1$), if $R$ is sufficiently large the above is at least
\[ 1 - e^{- (c_6/2) \ell' \sqrt{ \min\{ \log R, 1/ \overline{\kappa}( \sqrt{R}) \} }  }. \]
By $D_4$-symmetry, $ \prob[\Cross_{\ell'}(aR,  R) ]$ equals the probability that there is a top-bottom crossing of $[0,R] \times [0,aR]$ by a path included in $\{ f \geq -\ell' \}$. Since $-f$ has the same law as $f$, this implies the result for any $\ell < -\ell'$ for $R$ larger than some constant (that depends on $\ell$), and the result for $R$ less than this constant is direct by taking $c_1$ sufficiently large.
\end{proof}

\begin{proof}[Proof of Corollary \ref{c:sodin}]
Let $\varepsilon > 0$. The result follows from Theorem \ref{t:sharp} by noting that there is some constant $N=N(\varepsilon) \in \N$ and, for each $R>0$, a collection of $N$ translations and rotations by $\pi/2$ of the event
\[ \Cross_\ell(\varepsilon R/100,\varepsilon R/50) \]
such that, if $\{ f \geq -\ell \} \cap B_0(R)$ has a component of diameter larger than $\varepsilon R$, then one of these events hold.
\end{proof}

\medskip
\section{Sharp thresholds are equivalent to the delocalisation of~the~threshold~location}
\label{s:con}

In this section we fix a crossing domain $\mathfrak{D}=(D,\calF,A)$ as in Definition \ref{def:conn_domain} and prove Theorem \ref{t:varcon} and Proposition \ref{p:expcon}. Recall the notations $D^{++},D^+$ from Definition \ref{def:strat_domain}. First, in Section \ref{ss:covariance_formula}, we establish a general formula for the variance of $h(\thr_A(f))$ where $h(t)$ grows at most exponentially (Lemma \ref{lem:variance_as_saddles}). Then in Section \ref{ss:covariance_applications} we apply it with $h(t)=t$ and $h(t)=e^{\theta t/2}$ to prove, respectively, Theorem \ref{t:varcon} and Proposition \ref{p:expcon}.

\subsection{A covariance formula for the threshold height}
\label{ss:covariance_formula}

The goal of this subsection is to establish the following formula for the variance of certain functionals of the threshold height, in terms of the threshold height itself and the threshold location:

\begin{lemma}\label{lem:variance_as_saddles}
Let $f : D^{++} \to \mathbb{R}$ be a (not necessarily stationary) Gaussian field satisfying:
\begin{enumerate}
\item $f \in C^3(D^{++})$ almost surely;
\item For each $x,y\in D$ distinct, the random vector 
\[ (f(x),f(y),\nabla_xf,\nabla_yf) \in \R^6 \]
 is non-degenerate;
\item For each $x\in D$, the random vector 
\[ (\nabla_x f,\nabla^2_x f) \in \R^2\times\textup{Sym}_2(\R)  \] 
is non-degenerate.
\end{enumerate}

\noindent Let $\widetilde{f}$ be an independent copy of $f$, and for each $t\geq 0$, define $f_t :=e^{-t}f+\sqrt{1-e^{-2t}}\widetilde{f}$. Then for every $h\in C^1(\R)$ such that $|h'(t)|\leq c e^{c|t|}$ for some $c>0$ and all $t\in\R$,\begin{equation}\label{eq:variance_as_saddles}
\var(h(\thr_A(f)))=\int_0^\infty\E\brb{K(\sad_A(f),\sad_A(f_t))h'(\thr_A(f))h'(\thr_A(f_t))}e^{-t}dt
\end{equation}
where $K$ is the covariance function of $f$. In particular,
\begin{equation}\label{eq:simple_variance}
\var(\thr_A(f))=\int_0^\infty\E\brb{K(\sad_A(f),\sad_A(f_t))}e^{-t}dt .
\end{equation}
\end{lemma}

\begin{remark}
Note that, under the assumptions of Lemma \ref{lem:variance_as_saddles}, $f \in \calM(D,\calF)$ almost surely and $(\thr_A(f), \sad_A(f))$ is well-defined and measurable (see Lemmas \ref{lem:cond_implies_morse} and \ref{lem:diff_of_threshold}).
\end{remark}

\begin{remark}
Lemma \ref{lem:variance_as_saddles} is a spatial version of the Gaussian covariance interpolation formula presented for instance in \cite{cha14}. It is also closely related to the main result of \cite{bmr20} and Piterbarg's formula, discussed therein.
\end{remark}

Before proving Lemma \ref{lem:variance_as_saddles} we first compute the derivative of the threshold height in terms of the threshold location.

\begin{lemma}\label{lem:diff_of_threshold}
For each $u\in\calM(D,\calF)$, the maps $v\in\calM(D,\calF)\mapsto \sad_A(v)\in\R^2$ and $v \in \calM(D,\calF)\mapsto \thr_A(v)\in\R$ are continuous at $u$ in the $C^2$-topology. Moreover, for each $v\in C^2(D^+)$, the map $t\mapsto\thr_A(u+tv)$ is differentiable at $0$ with derivative $v(\sad_A(u))$.
\end{lemma}
\begin{proof}
We first recall that the set $\calM(D,\calF)$ is open in $C^2(D^+)$ (Lemma \ref{lem:morse_is_open}). As a result, the continuity of $\thr_A$ comes for instance from the inequality
\begin{equation}\label{eq:thr_is_Lip}
|\thr_A(u+v)-\thr_A(u)|\leq \|v\|_{C^0(D)}.
\end{equation}
Let us prove that $\sad_A$ is continuous. Let $u \in \calM(D,\calF)$ and $F \in \calF$ be such that $\sad_A(u)\in F$. Since $u$ has finitely many stratified critical points, with distinct critical values, the infimum over $F' \in \calF \setminus \{F\}$ and $x \in F'$ of $|u(x)-\thr_A(u)|+|\nabla_x(u|_{F'})|$ is positive. This and the continuity of $\thr_A$ imply that, if $v$ is sufficiently $C^1$-close to $u$, then $\sad_A(v) \in F$. By making the same observation on $F \setminus B_{\sad_A(u)}(\varepsilon)$ rather than $F'$, for every $\varepsilon>0$, we deduce that $\sad_A$ is continuous at $u$. (Note that we have actually proved that $\thr_A$ and $\sad_A$ are continuous, respectively, in the $C^0$ and $C^1$ topologies.)
\smallskip

Let $u,v \in \calM(D,\calF)$ and let $F\in\calF$ be such that $\sad_A(u)\in F$. Since $v$ and $u$ are both $C^2$-smooth, the map $h:(x,t)\mapsto \nabla_x(u|_F+tv|_F)$ defined on $F\times\R$ is differentiable and, since $\sad_A(u)$ is a non-degenerate critical point of $u|_F$, the linear map $\partial_xh(\sad_A(u),0)$ is invertible. By the implicit function theorem, there exists a differentiable map $t\mapsto x_t\in F$ defined for $|t|$ small enough such that $x_0=\sad_A(u)$ and ($h(x,t)=0 \Leftrightarrow x=x_t$) in a neighbourhood of $(\sad_A(u),0)$ (in particular, $x_t$ is a stratified critical point of $u+tv$). By continuity of $\sad_A$ (and since, as shown above, $\sad_A(v) \in F$ if $v$ is sufficiently $C^2$-close to $u$), $x_t=\sad_A(u+tv)$ if $|t|$ is small enough. We next note that
\begin{equation}\label{eq:diff_of_threshold}
\frac{d(u+tv)(x_t)}{dt} \Big|_{t=0}=\langle \nabla_{x_0}u,\dot{x}_0\rangle+v(x_0)=v(\sad_A(u))
\end{equation}
because $x_0=\sad_A(u)$ is a critical point of $u|_F$. By \eqref{eq:diff_of_threshold} and since $x_t=\sad_A(u+tv)$,
\[ \frac{d}{dt} \Big|_{t=0}\thr_A(u+tv)=\frac{d}{dt} \Big|_{t=0}(u+tv)(x_t)=v(\sad_A(u))  .\qedhere \]
\end{proof}

\begin{proof}[Proof of Lemma \ref{lem:variance_as_saddles}]
Let us start with a short technical remark: the reader can note that, although the desired formula only involves $f_{|D}$, the field $f$ considered in this lemma is defined on a neighbourhood $O$ of $D$. This assumption could probably be removed and will be useful only in the last paragraph of the proof (Step 3/3). Until this last paragraph, we only work on $D$ and write $f$ instead of $f_{|D}$.

\medskip
\textbf{Step 1/3.} We first assume that there exists a finite-dimensional subspace $H\subset C^3(D^+)$, equipped with a scalar product $\langle\cdot,\cdot\rangle$, such that $f$ has the law of a standard Gaussian vector on $H$, and prove the desired formula in this case. By applying the formula in \cite[Lemma 3.3]{ch08} we have
\begin{equation}
\label{e:cha}
 \var(h(\thr_A(f)))=\int_0^\infty\E\brb{\langle\nabla\thr_A(f),\nabla\thr_A(f_t)\rangle h'(\thr_A(f))h'(\thr_A(f_t))}e^{-t}dt .
 \end{equation}
Here $\textup{T}_A$ is seen as a map from $H$ to $\R$ and $\nabla \textup{T}_A$ is its gradient, which is a.s.\ defined at $f$. To justify the application of \cite[Lemma 3.3]{ch08} we need that $\nabla (h \circ \thr_A) = (h' \circ \thr_A) \nabla \thr_A$ is in $L^2$, which follows since $h'$ has at most exponential growth and $|\thr_A|$ is bounded by $\sup_D |f|$ which is sub-Gaussian (by the Borell--TIS inequality for instance; \cite[Theorem 2.1.1]{adler_taylor}).\footnote{Although \cite[Lemma 3.3]{ch08} as stated also requires $h \circ \thr_A$ to be absolutely continuous (which it is not in general), a truncation argument shows this requirement to be unnecessary.} Then note that, since $\langle \nabla\thr_A(f),\cdot\rangle$ is the evaluation map at $\sad_A(f)$ (by Lemma \ref{lem:diff_of_threshold}), and by the reproducing property of the covariance kernel \eqref{eq:reproducing_kernel}, we have that
\[ \langle\nabla\thr_A(f),\nabla\thr_A(f_t)\rangle = \nabla \thr_A(f_t)(\sad_A(f)) = \langle K(\sad_A(f), \cdot), \nabla \thr_A(f_t) \rangle =K(\sad_A(f),\sad_A(f_t)) ,\]
and so we have proven the result under the finite-dimensional assumption.

\medskip
\textbf{Step 2/3.} To extend to the general case we use standard approximation arguments (see \cite{ri19} for similar arguments) and the fact that, by Lemma \ref{lem:diff_of_threshold}, both $\thr_A$ and $\sad_A$ are continuous in the $C^2$-topology. In this step, we prove the result under the assumption that $f$ is almost surely $C^\infty$-smooth. To this purpose, it is sufficient to construct a sequence of Gaussian fields $(f_n)_n$ such that:
\begin{enumerate}[(i)]
\item $f_n$ is  supported on a finite-dimensional subspace of $C^3(D^+)$;
\item for $n$ sufficiently large, $f_n$ satisfies the assumptions of Lemma \ref{lem:variance_as_saddles};
\item there exists a sub-Gaussian random variable $X$ such that $|\thr_A(f_n)| \leq X$ for all $n$; and
\item $\thr_A(f_n)$ and $\sad_A(f_n)$ converge almost surely to $\thr_A(f)$ and $\sad_A(f)$ respectively.
\end{enumerate}
Indeed, suppose we could find such a sequence. Then \eqref{eq:variance_as_saddles} is true for $f_n$ almost surely, and the integrands on both sides of \eqref{eq:variance_as_saddles} applied to $f_n$ converge almost surely to the same expression with $f_n$ replaced by $f$. Moreover, the existence of the random variable $X$ together with the growth condition on $h'$ implies that (by the dominated convergence theorem applied to both sides of the equation) \eqref{eq:variance_as_saddles} is also true for $f$. 

\smallskip
So let us exhibit such a sequence. Let $H^{10} = H^{10}(D^+)$ be the $L^2$-Sobolev space of order $10$ on~$D$ (here $10$ is arbitrary but we need it to be $>4$). Then $H^{10} \subset C^3(D^+)$ and the injection is continuous. Let $(e_k)_{k\geq 0}$ be an orthonormal basis for $H^{10}$. Since $f$ is $C^\infty$-smooth, $f\in H^{10}$ almost surely and (by the Borell--TIS inequality \cite[Theorem 2.1.1]{adler_taylor}) there exists a $c_1 > 0$ such that, for each $t>c_1$,
\begin{equation}\label{eq:gaussian_tails_norm}
\prob\brb{\|f\|_{H^{10}}>t}\leq c_1e^{-t^2/c_1}  .
\end{equation}
Now let $\Pi_n$ be the projector onto the sub-space $H_n\subset H^{10}$ generated by $e_1,\dots,e_n$, and define $f_n=\Pi_n(f)$. Then $f_n$ converges to $f$ almost surely in $H^{10}$ as $n\to \infty$, and so in particular $f_n$ converges almost surely to $f$ in $C^2(D^+)$ (even in $C^3(D^+)$).
\smallskip

We now verify (i)--(iv). Clearly $f_n$ belongs to a finite-dimensional subspace of $C^3(D^+)$ by definition. Moreover, by Lemma \ref{lem:cond_implies_morse}, for $n$ sufficiently large almost surely $f_n$ satisfies the assumptions of Lemma \ref{lem:variance_as_saddles}. Next notice that $|\thr_A(f_n)|\leq\|f_n\|_{C^0(D)}$ almost surely, and moreover, there is a constant $c_2=c_2(D^+) > 0$ such that
\[ \|f_n\|_{C^0(D)} \leq \|f_n\|_{C^2(D)} = \|\Pi_n f\|_{C^2(D)} \le \|\Pi_n f\|_{C^2(D^+)}\leq c_2\|\Pi_n f\|_{H^{10}}\leq c_2\|f\|_{H^{10}} .\]
Since, by \eqref{eq:gaussian_tails_norm}, $\|f\|_{H^{10}}$ is sub-Gaussian, we conclude that (iii) holds with $X=\|f\|_{H^{10}}$. Finally, by Lemma \ref{lem:diff_of_threshold}, $\thr_A$ and $\sad_A$ are both continuous on $\calM(D,\calF)$, so (iv) holds.

\medskip
\textbf{Step 3/3.} We thus have the result for $f \in C^\infty(D^+)$. To remove the $C^\infty$-smoothness assumption, we use that our field is actually defined on $D^{++}$ and we proceed as follows. One may construct a sequence of convolutions of $f$ by smooth approximations of the Dirac which are compactly supported in $B_0(\varepsilon/2)$, where $\varepsilon$ is the distance between $D^+$ and $\partial D^{++}$, and reason as in the previous approximation step (indeed, one can easily choose the approximations of the Dirac so that the convolution operations define $C^\infty$-smooth fields $f_n$ on $D$ such that $f_n$ converges to $f$ almost surely in $C^2(D^+)$, and by the convolution inequality $\|\varphi_1*\varphi_2\|_\infty\leq \|\varphi_1\|_1\|\varphi_2\|_\infty$, $\|f_n\|_{\infty,D^+} \le \|f\|_{\infty,D^{++}}$).
\end{proof}

\subsection{Applications of the covariance formula}\label{ss:covariance_applications}
We now apply Lemma \ref{lem:variance_as_saddles} to prove Theorem~\ref{t:varcon} and Proposition~\ref{p:expcon}. In so doing we return to the setting of stationary fields, replacing $K(x,y)$ by $\kappa(x-y)$, although we stress that stationarity is not essential. We work under Assumption \ref{as:main} so that $f$ satisfies the assumptions of Lemma \ref{lem:variance_as_saddles} (by Lemma \ref{lem:as_main_implies_cond}).

\smallskip
We begin with the proof of Theorem \ref{t:varcon}. Recall the definitions of $\overline{\kappa}(r)$ and $\sigma_A(r)$ in \eqref{eq:kappa_bar} and \eqref{eq:sigma} respectively.

\begin{proof}[Proof of Theorem \ref{t:varcon}; upper bound]
Starting from Lemma \ref{lem:variance_as_saddles}, we follow an argument due to Chatterjee (see \cite[Section 4]{ch08}). Fix $r>0$, let $(C_j)_j$ be a collection of pairwise disjoint (partially closed) squares of side length $r$ that cover $D$, and for each $j$, let $C'_j$ be the union of $C_j$ and the eight squares surrounding it. Observe that for each $j$ there are at most $25$ indices $k$ such that $C_j'\cap C_k' \neq  \emptyset$. Moreover, if $x,y$ are two points and $j$ is an index such that $x \in C_j$ and $y\notin C_j'$, then $x$ and $y$ must be at distance at least $r$ from each other.

Now, for each $t\geq 0$,
\[ \E\brb{\kappa(\sad_A(f)-\sad_A(f_t))} \leq \sup_{|x|\geq r}|\kappa(x)|+ \kappa(0) \sum_j\prob\brb{\sad_A(f)\in C_j;\, \sad_A(f_t)\in C'_j} . \]
We now use the hypercontractivity property of the Ornstein--Uhlenbeck semigroup (see \cite[Theorem 5.8]{janson}, and also \cite[Example 4.7]{janson} for definitions). In our context, this property can be stated as follows: for every $\varphi :\, C^3(\R^2) \rightarrow \R$ such that $\varphi(f)$ is a (measureable) $L^2$ random variable and $t \geq 0$,
\[\E [ \varphi(f) \varphi(f_t) ]^{1/2} \leq \E [ \varphi(f)^{p(t)} ]^{1/p(t)},\]
where $p(t)=1+e^{-t}$. This property implies that
\begin{multline}
\label{eq:simple_concentration}
 \prob\brb{\sad_A(f)\in C_j ;\,  \sad_A(f_t)\in C'_j} \leq \prob\brb{\sad_A(f)\in C_j' ;\,  \sad_A(f_t)\in C'_j}
\leq \prob\brb{\sad_A(f)\in C'_j}^{2/p(t)}.
 \end{multline}
 Hence,
\begin{align*}
\sum_j\prob\brb{\sad_A(f)\in C_j;\, \sad_A(f_t)\in C'_j} & \leq \sum_j\prob\brb{\sad_A(f)\in C'_j}^{2/p(t)} \\
& \leq   \sigma(3r)^{2/p(t) - 1}  \sum_j \prob\brb{\sad_A(f)\in C'_j} \leq 25 \sigma(3r)^{2/p(t) - 1} .
\end{align*}
Note that $2/p(t) - 1=\textup{tanh}(t/2)$, and that, for any $\alpha \in (0,1)$,
\begin{equation}
\label{e:tanh}
 \int_0^\infty \alpha^{\textup{tanh}(t/2)} e^{-t}\, dt \le \int_0^\infty \alpha^{\textup{tanh}(t/2)} e^{-t/2}  \, dt \le  \frac{2}{|\log(\alpha)|} 
 \end{equation}
(as can be seen from the inequality $\textup{tanh}(t/2) \ge 1 - e^{-t/2}$ for instance). Combining with \eqref{eq:simple_variance} we have
\begin{equation}
 \var(\thr_A(f))\leq \overline{\kappa}(r) + 25 \kappa(0) \int_0^\infty\sigma_A(3r)^{\textup{tanh}(t)/2}e^{-t}dt \le \overline{\kappa}(r) + \frac{50 \kappa(0)}{|\log(\sigma_A(3r))|} . 
\end{equation}
This concludes the proof since, by the union bound, there exists a universal $c>0$ such that $\sigma_A(3r) \leq c\sigma_A(r)$.
\end{proof}

\begin{proof}[Proof of Theorem \ref{t:varcon}; lower bound] We first note that we can assume without loss of generality that the faces $F \in \calF$ consist of (i) the interior of $D$, (ii) the $2$ distinguished sides and at most $6$  other intervals in $\partial D$, and (iii) the at most $8$ endpoints of the boundary faces; in particular, we can assume $|\calF| \leq 17$. Fix $x_0\in D$ and $r\geq e$ and abbreviate $B_r=B_{x_0}(r)$. Denoting $\ell_A=\E[\thr_A]$ we observe that, for each $\eps>0$, the event $|\thr_A-\ell_A|\geq \eps$ is implied by the intersection of (i) $\sad_A\in B_r$, and (ii) the event that there exists no stratified critical point $x\in B_r$ of $f$ with critical value between $\ell_A-\eps$ and $\ell_A+\eps$. In particular,
\begin{equation}\label{eq:var_lower_bound_1}
\eps^{-2}\var(\thr_A)\geq\prob[\sad_A\in B_r]-\sum_{F\in\calF}\prob\brb{\exists x\in B_r\cap F,\ \nabla_x(f|_F)=0, |f(x)-\ell_A|<\eps}  .
\end{equation}
We proceed by bounding the probabilities on the right-hand side of \eqref{eq:var_lower_bound_1} and then optimising over $\eps$. To this end we fix $F\in\calF$. By the Borell--TIS inequality (see \cite[Theorem 2.1.1.]{adler_taylor}), there exists a constant $c_1=c_1(\kappa) \ge 1$ such that, if $M_r\geq \sqrt{c_1\log r}$, then
\[ \prob \Big[\sup_{x\in B_r}(|\nabla_xf|+|\nabla^2_xf|)\geq M_r \Big] \leq 2e^{-M_r^2/c_1^2} . \]
On the event $\sup_{x\in B_r}(|\nabla_xf|+|\nabla^2_xf|) < M_r$, the existence of a critical point $x\in B_r\cap F$ of $f|_F$ with critical value between $\ell_A-\eps$ and $\ell_A+\eps$ implies that the volume of the set
\[ E(\eps) :=\{y\in B_r\cap F\, :\, |f(y)-\ell_A|\leq 2\eps, |\nabla_y(f|_F)|\leq 2\eps\} \]
is at least $\eps^{d_F}/M_r^{d_F}$, where $d_F=\textup{dim}(F)$. Hence, using Markov's inequality and stationarity, there exists $c_2=c_2(\kappa) > 0$ such that
\begin{align}\label{eq:var_lower_bound_4}
\nonumber \prob\brb{\exists x\in B_r\cap F,\ \nabla_x(f|_F)=0, |f(x)-\ell_A|<\eps} & \leq M_r^{d_F}\eps^{-d_F}\E[\text{vol}(E(\eps))]+2e^{-M_r^2/c_1^2} \\
 & \leq c_2r^{d_F}M_r^{d_F}\eps+2e^{-M_r^2/c_1^2} .
\end{align}
In particular, applying \eqref{eq:var_lower_bound_4} to each $F$ in the right-hand side of \eqref{eq:var_lower_bound_1}, and since $rM_r\geq 1$,
\[ \var(\thr_A)\geq\eps^2 \Big( \prob\brb{\sad_A\in B_r}-17c_2r^2M_r^2\eps-2e^{-M_r^2/c_1^2} \Big) \]
(recall that we may assume $|\calF| \le 17$). Hence there exists $c_3=c_3(\kappa)\geq 1$ such that, setting 
\[ M_r=c_3 \big( \max \left\{|\log(\prob\brb{\sad_A\in B_r})|, \log r \right\} \big)^{1/2} \quad \text{and} \quad \eps=c_3^{-1}\prob\brb{\sad_A\in B_r}r^{-2}M_r^{-2} , \]
we get
\[ \var(\thr_A)\geq\frac{\prob\brb{\sad_A\in B_r}^3}{c_4 r^4 \big(\max\{\log r,|\log(\prob\brb{\sad_A\in B_r})|\} \big)^2 }  ,\]
where $c_4 = c_4(\kappa) > 0$. Taking the supremum over $x_0$ and $r \ge e$ gives the result.
 \end{proof}

Proposition \ref{p:expcon} is a refinement of the upper bound of Theorem \ref{t:varcon} proved above. The idea is to replace $h(t)=t$ in Lemma \ref{lem:variance_as_saddles} by $h(t)=e^{\theta t/2}$ and optimise over $\theta$. This idea was used by Tanguy in \cite{tan15} (see Theorem 5 therein) to study the maxima of Gaussian vectors, and we include a brief proof for completeness (and since our setting is slightly different).

\begin{proof}[Proof of Proposition \ref{p:expcon}]
As in the proof of Theorem \ref{t:varcon}, fix $r>0$, let $(C_j)_j$ be a collection of pairwise disjoint (partially closed) squares of side length $r$ that cover $D$, and for each $j$, let $C'_j$ be the union of $C_j$ and the eight squares surrounding it. By Lemma \ref{lem:variance_as_saddles}, for each $\theta\in\R$,
\begin{equation}\label{eq:tanguy_start}
\var \big( e^{\theta\thr_A(f)/2} \big) =\frac{\theta^2}{4}\int_0^{\infty}\E \Big[ \kappa(\sad_A(f)-\sad_A(f_t))e^{\theta\thr_A(f)/2}e^{\theta\thr_A(f_t)/2} \Big] e^{-t}dt .
\end{equation}
Abbreviating $\sad$ (resp.\ $\sad_t,\thr,\thr_t$) for $\sad_A(f)$ (resp.\ $\sad_A(f_t),\thr_A(f),\thr_A(f_t)$), \eqref{eq:tanguy_start} is bounded by
\begin{align*}
&\frac{\theta^2}{4}\sum_{j}\int_0^{\infty}  \br{ \kappa(0) \E \Big[ \un_{[\sad,\sad_t\in C'_j]}e^{\theta\thr/2}e^{\theta\thr_t/2} \Big]+\overline{\kappa}(r)\E\brb{e^{\theta(\thr+\thr_t)/2}}}e^{-t}dt\\
& \qquad \leq \frac{\theta^2}{4}\sum_{j }\br{\int_0^{\infty} \kappa(0) \E\brb{M_j(\theta)M_j^t(\theta)}e^{-t}dt+\overline{\kappa}(r)\E[e^{\theta\thr} ]} ,
\end{align*}
where $M_j(\theta)=\un_{\sad\in C'_j}e^{\theta\thr/2}$, $M_j^t$ is defined analogously, and the last step uses the Cauchy--Schwarz inequality. By the hypercontractivity of the Ornstein--Uhlenbeck semigroup (see the proof of Theorem \ref{t:varcon}), 
\[ \E [M_j(\theta)M_j^t(\theta) ] \leq\E [M_j(\theta)^{p(t)}]^{2/p(t)} , \]
 where $p(t)=1+e^{-t}$. By the above and H\"older's inequality (applied to $p'=2/p(t)$ and $q'=2/(2-p(t))$) we have
\begin{multline*} \sum_{j }\brb{M_j(\theta)M_j^t(\theta)} \leq \sum_j \E [M_j(\theta)^{p(t)}]^{2/p(t)} \leq \sum_{j }\prob\brb{\sad\in C'_j}^{\frac{2-p(t)}{p(t)}}\E \big[ M_j(\theta)^2 \big]\\
= \sum_{j }\prob\brb{\sad\in C'_j}^{\frac{2-p(t)}{p(t)}}\E \big[\un_{\sad\in C'_j}e^{\theta\thr} \big] \leq 25\sup_{j}\prob\brb{\sad\in C'_j}^{\frac{2-p(t)}{p(t)}}\E[ e^{\theta\thr} ]  . 
\end{multline*}
Since $(2-p(t))/(p(t)) =\textup{tanh}(t/2)$, by integrating over $t$ (recall \eqref{e:tanh}) we deduce that
\[ \var \big( e^{\theta\thr/2} \big) \leq \frac{\theta^2}{4} \bigg( \frac{50 \kappa(0)}{|\log{\sup_{j }\prob [\sad\in C'_j]}|}+\overline{\kappa}(r) \bigg) \E \big[e^{\theta\thr } \big]  . \]
To conclude, we use that $\sup_{j}\prob[\sad\in C'_j]\leq \sigma_A(3r)$ and the general fact (see \cite[Lemma 6]{tan15}, and \cite[Page 51]{ledoux01} for the proof) that for any random variable $Z$ and constant $K>0$,
\[ \forall |\theta|\leq 2\sqrt{K},\ \var \big( e^{\theta Z/2} \big) \leq K\frac{\theta^2}{4}\E [e^{\theta Z} ] \, \Longrightarrow\, \forall t>0,\prob\brb{|Z-\E[Z]|\geq t}\leq 6 e^{-ct/\sqrt{K}} , \]
where $c>0$ is a universal constant. (Note that we have actually proven the result with $\sigma_A(3r)$ instead of $\sigma(r)$, but this is equivalent since $\sigma_A(r) \leq \sigma_A(3r) \leq c'\sigma_A(r)$ for a universal $c'>0$.)
\end{proof}

\medskip
\section{Delocalisation of the threshold location}
\label{s:deloc}

In this section we prove Proposition~\ref{prop:delocalization} (or rather \eqref{e:sigrec}, since \eqref{e:sigann} has already been proved) on the delocalisation of the threshold location $\sad_A$. For this purpose, we first study macroscopic saddle points (which will control delocalisation in the bulk of the crossing domain) and then we study connection properties of the model in a half-plane (which will control delocalisation on the boundary). These two cases are treated very differently.

\subsection{On macroscopic four-arm saddles}\label{ss:4arm}

Recall that $B_0(R) :=\{|x|\leq R\}$. We shall call a function \textit{$R$-perfect Morse} if it is a $(B_0(R),\mathcal{F})$-perfect Morse function where $\mathcal{F}=\{ B_0(R),\partial B_0(R) \}$, see Definition \ref{d:morse}. Recall that, for every $R > 0$, $f$ is $R$-perfect Morse almost surely by Lemma \ref{l:morse}.

\begin{definition}
Let $R>0$, $u \in C^2(\R^2)$ and let $x$ be a critical point of $u$. We say that $x$ is an \textit{$R$-saddle point} if there exist four injective paths $\gamma_1^x,\ldots,\gamma_4^x$ from $x$ to $\partial B_x(R)$, intersecting pairwise only at $x$, such that $u_{|\cup_i \gamma_i^x}$ is constant.
\end{definition}

An analogue of the following proposition appears in \cite{bmm20} (see Lemma 4.5 therein). We have chosen to include a (different, Burton--Keane-type) proof for completeness.

\begin{proposition}\label{prop:interior_saddle}
Let $u\in C^2(\R^2)$ be an $R$-perfect Morse function. Then the number of $2R$-saddle points of $u$ in $B_0(R)$ is less than or equal to
\[ \max \{ 0 , \text{ number of critical points of } u_{|\partial B_0(R)} - 3 \}. \]
In particular, it is less than or equal to the number of critical points of $u_{|\partial B_0(R)}$.
\end{proposition}

We first use Proposition \ref{prop:interior_saddle} to show the following: 

\begin{corollary}\label{cor:4arm}
Let $f$ satisfy Assumption \ref{as:main}. Then there exists $c>0$ such that the probability that there is an $R$-saddle point of $f$ in $B_0(1)$ is less than $c/R$.
\end{corollary}

\begin{proof}
First note that it is sufficient to prove the result for $R=2n$ where $n$ is a positive integer. Given $D \subset \R^2$, let $N^s_D(2n)$ denote the number of $2n$-saddle points in $D$. By Proposition \ref{prop:interior_saddle},
\begin{align*}
(n^2/100) \Pro \big[ N^s_{B_0(1)}(2n) \geq 1 \big] & \leq (n^2/100) \E \big[ N^s_{B_0(1)}(2n) \big]  \leq \E \big[ N^s_{B_0(n)}(2n) \big] \\
& \leq \E \big[ \text{number of critical points of } f_{|\partial B_0(n)} \big],
\end{align*}
where in the second to last inequality we have used translation invariance and the fact that $B_0(n)$ contains more than $n^2/100$ disjoint Euclidean balls of radius $1$. The result now follows from the fact that the last term equals $O(n)$, which is for instance a direct consequence of the Kac--Rice formula (see Lemma \ref{lem:O(n)}).
\end{proof}
\begin{proof}[Proof of Proposition \ref{prop:interior_saddle}]
Let $Y$ denote the set of critical points of $u_{|\partial B_0(R)}$. Note that each $2R$-saddle point in $B_0(R)$ induces a four-partition of $Y$, i.e.\ a partition of $Y$ in four non-empty sets. This can be done as follows: (i) consider, for each path $\gamma_i^x$ in the definition of a $2R$-saddle point, the first intersection point with $\partial B_0(R)$; (ii) note that these four points of $\partial B_0(R)$ cut $\partial B_0(R)$ in four pieces; (iii) since (by using that $u$ is $R$-perfect Morse), these four points do not belong to $Y$, these four pieces of $\partial B_0(r)$ indeed induce a four-partition of $Y$; (iv) finally, each of these four subsets of $Y$ is non empty by Rolle's lemma.

\smallskip
Now observe that if $x \neq x'$ are two $2R$-saddle points in $B_0(R)$ then the two induced four-partitions $\Pi=\{P_1,\ldots,P_4\}$ and $\Pi'=\{P_1',\ldots,P_4'\}$ are compatible in the sense that there exists an ordering of their elements such that $P_1 \supseteq P_2' \cup P_3' \cup P_4'$. This comes from the fact that, since $u$ is $R$-perfect Morse, $u(x) \neq u(x')$ so that for $i,j\in\{1,\dots,4\}$, $\gamma_i^x\cap\gamma_j^{x'}=\emptyset$. By \cite[Lemma 8.5 ]{gri99} the number of such partitions is at most the cardinality of $Y$ minus $3$ which proves the result (note that, although \cite[Lemma  8.5]{gri99} treats three-partitions, the proof is exactly the same for four-partitions with $|Y|-3$ for the latter replacing $|Y|-2$ for the former).
\end{proof}

\subsection{On unbounded nodal lines in the half-plane}

We next prove that there are no unbounded nodal lines (i.e.\ components of $\{f = 0\}$) in the half-plane that intersect the boundary (Proposition \ref{prop:half}), inspired by arguments in \cite{ha60}. As we will see (Corollary \ref{cor:positiveinhalfplane}), this implies the non-existence of unboundedness components in the half-plane that intersect the boundary, for each of $\{ f \ge 0\}$, $\{f \le 0\}$ and $\{f = \ell\}$, $\ell \in \mathbb{R}$.

\smallskip 
We start with an elementary lemma:

\begin{lemma}[Smoothness of nodal set]\label{l:smoothC1mani}
Let $L \subset \mathbb{R}^2$ be a line. Then the following holds almost surely:
\begin{itemize}
\item The set $\{ f = 0\}$ is a $C^1$-smooth one-dimensional manifold that is not tangent to $L$;
\item The sets $\{ f \geq 0 \}$ and $\{ f \leq 0\}$ are two $C^1$-smooth two-dimensional manifolds with boundary and the boundary of both these manifolds equals $\{ f = 0 \}$.
\end{itemize}
\end{lemma}
\begin{proof}
This follows from Bulinskaya's lemma. More precisely, it suffices to apply \cite[Lemma 11.2.10]{adler_taylor} to $x\mapsto (g(x),\nabla g(x))$ where $g$ is either $f$ or its restriction to $L$.
\end{proof}

We will use the first item of Lemma \ref{l:smoothC1mani} several times in the sequel without further mention.

\smallskip 
Let us recall that the hypothesis `$\kappa \rightarrow 0$' from Assumption \ref{as:main} implies that $f$ is ergodic with respect to the translations (see for instance \cite[Theorem 6.5.4]{adler10}). In the sequel when we say ``by ergodicity'' we mean that we use implicitly the following direct consequence of the Birkhoff–Khinchin theorem: if $T_n$ is the translation by $(0,n)$ or by $(n,0)$, and if $A$ satisfies $\Pro [ f \in A ] > 0$, then almost surely there exist infinitely many positive integers $n$ and infinitely many negative integers $n$ such that $f \in T_n^{-1}(A)$. 

\smallskip
As in \cite{ha60}, we begin by studying the case of a slab:

\begin{lemma}\label{lem:slab}
Let $f$ satisfy Assumption \ref{as:main}, fix $a > 0$, and consider the slab $S = [0,a] \times \R$. Then almost surely the set $\{ f = 0 \} \cap S$ has no unbounded connected component.
\end{lemma}

\begin{proof}
By Lemma \ref{l:smoothC1mani}, $\Pro [ \Cross(a,a)]$ equals the probability that there is a left-right crossing of $[0,a]^2$ by a path included in $\{ f > 0 \}$ (note the strict inequality). Since $\Pro [ \Cross(a,a) ]  = 1/2 > 0$ (Lemma \ref{l:sc}), by ergodicity almost surely there exist infinitely many positive integers $n$ and infinitely many negative integers $n$ such that the box $[0,a] \times [na,(n+1)a]$ is crossed from left to right by a continuous path included in $\{ f > 0 \}$. This prevents the existence of an unbounded component in $\{ f = 0 \} \cap S$.
\end{proof}

The next lemma replaces the slab with a quarter-plane but with a slightly weaker conclusion:

\begin{lemma}\label{lem:quarter}
Let $f$ satisfy Assumption \ref{as:main} and consider the quarter-plane $Q=\R_+ \times \R_+$. Then almost surely the set $\{ f = 0 \} \cap Q$ has no unbounded connected component that intersects~$\partial Q$.
\end{lemma}

\begin{remark}
We believe that $\{f=0\}\cap Q$ does not contain any unbounded components at all but the above statement is sufficient for our purposes.
\end{remark}

\begin{proof}[Proof of Lemma \ref{lem:quarter}]
Assume for the sake of contradiction that $\{ f = 0 \} \cap Q$ has an unbounded component that intersects~$\partial Q$ with positive probability. By symmetry and rotation by $\pi/2$, there exists $C>0$ such that this is still the case if we ask furthermore that this component intersects $[0,C] \times \{0\}$; let $A_C$ denote the event with this additional property. By ergodicity and $D_4$-symmetry, almost surely there exists some $n \in \N$ larger than $1$ such that the event obtained from $A_C$ by reflecting along the $y$-axis and translating by the vector $(nC,0)$ holds (this event is illustrated in Figure \ref{fig:quarter}). However almost surely either (i) this event prevents $A_C$ from holding, or (ii) there exists an unbounded component in $\{ f = 0 \} \cap ([0,nC] \times \R)$. Since the latter event has probability $0$ by Lemma \ref{lem:slab}, we have $\Pro [ A_C ] = 0$, which is the required contradiction.
\end{proof}

\begin{figure}[h!]
\centering
\includegraphics[scale=0.4]{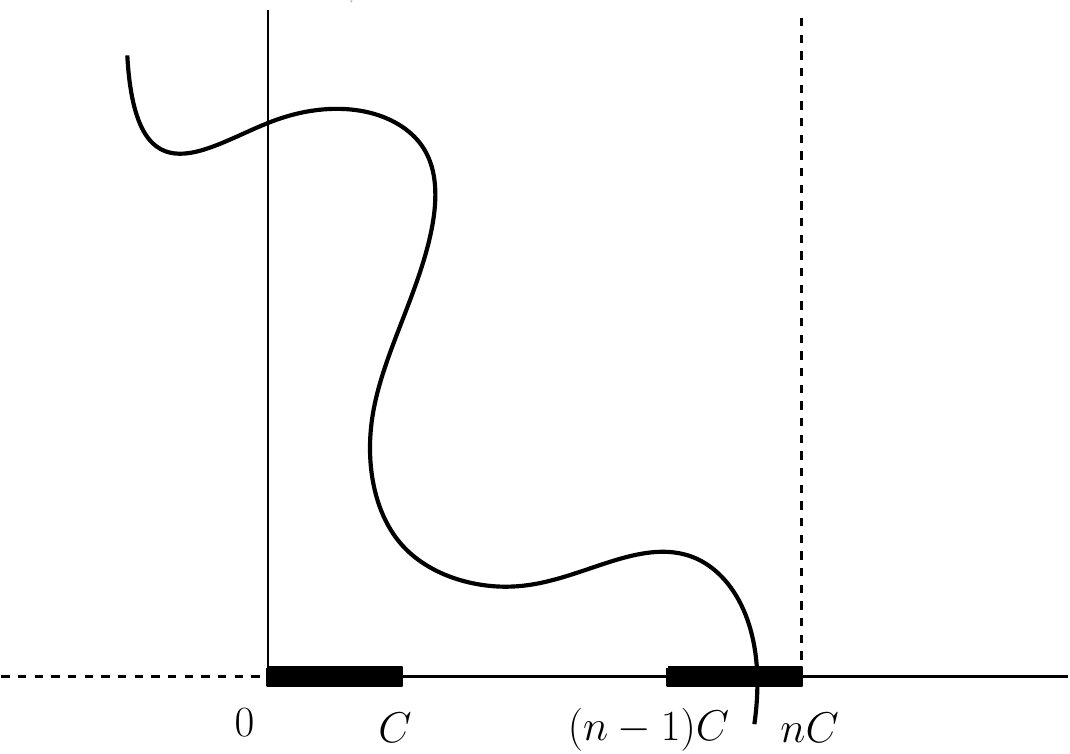}
\caption{The event obtained from $A_C$ by applying some symmetries prevents $A_C$ from holding.\label{fig:quarter}}
\end{figure}

Finally we prove the result analogous to Lemma \ref{lem:quarter} replacing the quarter-plane with the half-plane:

\begin{proposition}\label{prop:half}
Let $f$ satisfy Assumption \ref{as:main} and consider the half-plane $H=\R \times \R_+$. Then almost surely the set $\{ f = 0 \} \cap H$ has no unbounded connected component that intersects~$\partial H$.
\end{proposition}
\begin{proof}
First note that it is sufficient to prove that almost surely there is no continuous function $\gamma \, : \, \R_+ \rightarrow \R^2$ that satisfies the following three properties:
\begin{enumerate}[(P1)]
\item $\gamma(\R_+) \subset \{ f = 0 \} \cap H$;
\item $\gamma(t) \in \partial H \Longleftrightarrow t = 0$;
\item $\gamma(t) \to \infty \text{ as } t \to \infty$.
\end{enumerate}
The proof is then a direct consequence of Claims \ref{cl:infinitly_close} and \ref{cl:escape} immediately below.
\end{proof}

\begin{claim}\label{cl:infinitly_close}
Let $C>0$. Almost surely there is no continuous function $\gamma \, : \, \R_+ \rightarrow \R^2$ that satisfies (P1)--(P3) as well as: 

\vspace{0.1cm}
\noindent (P4) There exist two sequences $s_k \to \infty$ and $t_k \to \infty$ such that, for all $k$, 
\[ \text{dist}(\gamma(s_k), \R_- \times \{0\}) < C \quad \text{ and } \quad \text{dist}(\gamma(t_k), \R_+ \times \{0\}) < C . \]
\end{claim}

\begin{claim}\label{cl:escape}
Almost surely there is no continuous function $\gamma \, : \, \R_+ \rightarrow \R^2$ that satisfies (P1)--(P3) as well as:
\vspace{-0.1cm}
\[ \text{(P4')} \qquad  \text{dist} (\gamma(t), \R_- \times \{ 0 \})  \to  \infty \quad \text{ as } t \to \infty. \qquad \]
\end{claim}

\begin{proof}[Proof of Claim \ref{cl:infinitly_close}]
Assume for the sake of contradiction that such a $\gamma$ exists with positive probability. Then (by translation invariance) this is also the case if we ask furthermore that $\gamma(0) \in [-1,1] \times \{0\}$. Let $B$ denote the event with this additional property and let $B_n$ denote the event $B$ translated by the vector $(0,n)$. By ergodicity, almost surely there exist infinitely many positive integers $n$ such that $B_n$ holds, so we obtain a contradiction if $B$ and $B_n$ are almost surely disjoint for $n>C$. To show this, let $n>C$, assume $B$ holds, and consider a number $T>0$ such that $\gamma(t) \notin [-1,1] \times [0,n]$ for every $t>T$ (such a $T$ exists by (P3)). By (P4) there exist $t^*>s^*>T$ such that $\gamma(s^*) \in (-\infty,-1) \times \{ n \}$ and $\gamma(t^*) \in (1,\infty) \times \{ n \}$, and then the path $(\gamma(u))_{s^* \leq u \leq t^*}$ prevents $B_n$ from holding, see Figure~\ref{fig:halffirst}.
\end{proof}

\begin{figure}[h!]
\centering
\includegraphics[scale=0.4]{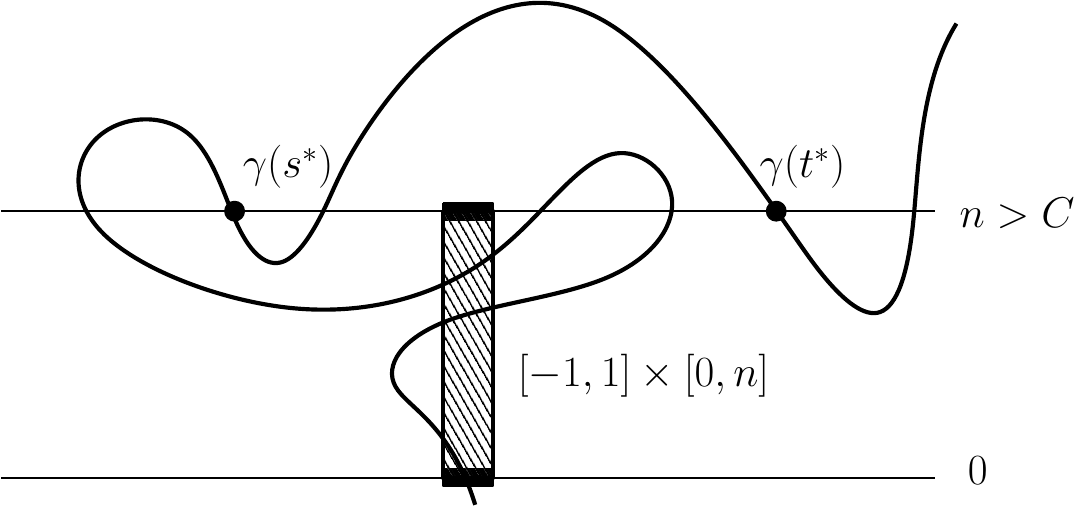}
\caption{The event $B$ prevents $B_n$ from holding if $n>C$.\label{fig:halffirst}}
\end{figure}

\begin{proof}[Proof of Claim \ref{cl:escape}]
Let $E$ be the event that such a $\gamma$ exists and assume for the sake of contradiction that $\Pro [E]>0$. Then the event $E^{\pi/2}$ obtained by rotating in the origin by $\pi/2$ also holds with positive probability. For every $D>0$, let $E^{\pi/2}_D$ be the event obtained from $E^{\pi/2}$ by asking furthermore that $\gamma(0) \in \{ 0 \} \times [-D,D]$, and  let $E^{-\pi/2}_D$ be the event obtained from $E^{\pi/2}_D$ by reflecting along the $y$-axis. By ergodicity, there exists $D>0$ such that $E^{\pi/2}_D \cap E^{-\pi/2}_D$ occurs with positive probability (this event is illustrated in Figure \ref{fig:half} a)).

\smallskip
Fix an $N \in \N$. Since we have assumed that with positive probability there exists a path $\gamma$ that satisfies (P1)--(P3), by ergodicity there exists an $M>0$ such that, with positive probability, there are in fact $N$ disjoint such paths $\gamma_i$ such that $\gamma_i(0)$ belongs to $[-M,M] \times \{ 0 \}$. Fix such an $M > 0$ and let $F_N(M)$ denote this event.

\smallskip
We now use the positivity of the probabilities of $E^{\pi/2}_D \cap E^{-\pi/2}_D$ and $F_N(M)$ to find a contradiction. We first note that, by Lemma \ref{lem:quarter} and outside an event of probability $0$, the paths induced by the events $E^{\pi/2}_D$ and $E^{-\pi/2}_D$ necessarily hit the line $\R \times \{ - n \}$ for each $n \in \N$. Moreover, (P4') implies that there exists some (random) $n_0$ such that, for all $n>n_0$, the points at which these paths hit $\R \times \{ - n \}$ for the first time are at distance larger than $M$ from the boundary of the corresponding half-plane, see Figure \ref{fig:half} a). We also note that by ergodicity there almost surely exist an $n > n_0$ such that $F_N(M)$ translated by $(0,-n)$ holds. This implies that $f$ has at least $N+2$ zeros on the segment $\{ 0 \} \times [-D,D]$, see Figure \ref{fig:half} b). Since $N$ was chosen arbitrarily, we obtain that if $E^{\pi/2}_D \cap E^{-\pi/2}_D$ holds then almost surely $f$ has infinitely many zeros on $\{ 0 \} \times [-D,D]$. Since this event has probability $0$ (by Lemma \ref{l:smoothC1mani}), we have a contradiction.
\end{proof}

\begin{figure}[h!]
\centering
\includegraphics[scale=0.4]{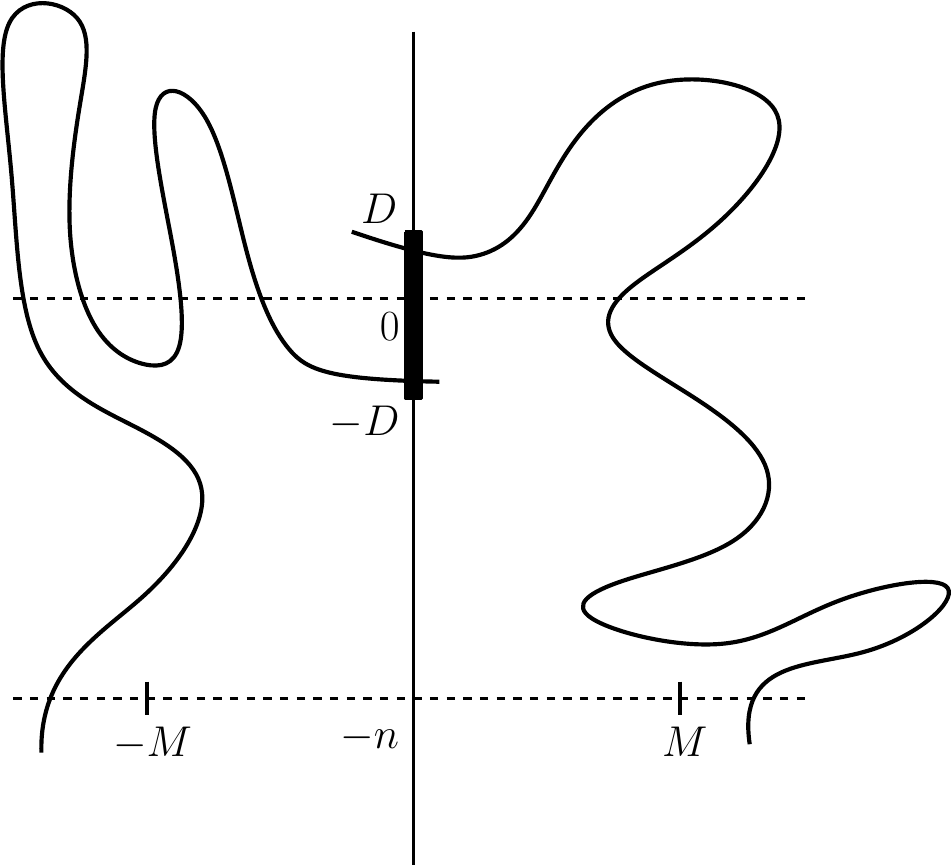}
\hspace{1cm}
\includegraphics[scale=0.4]{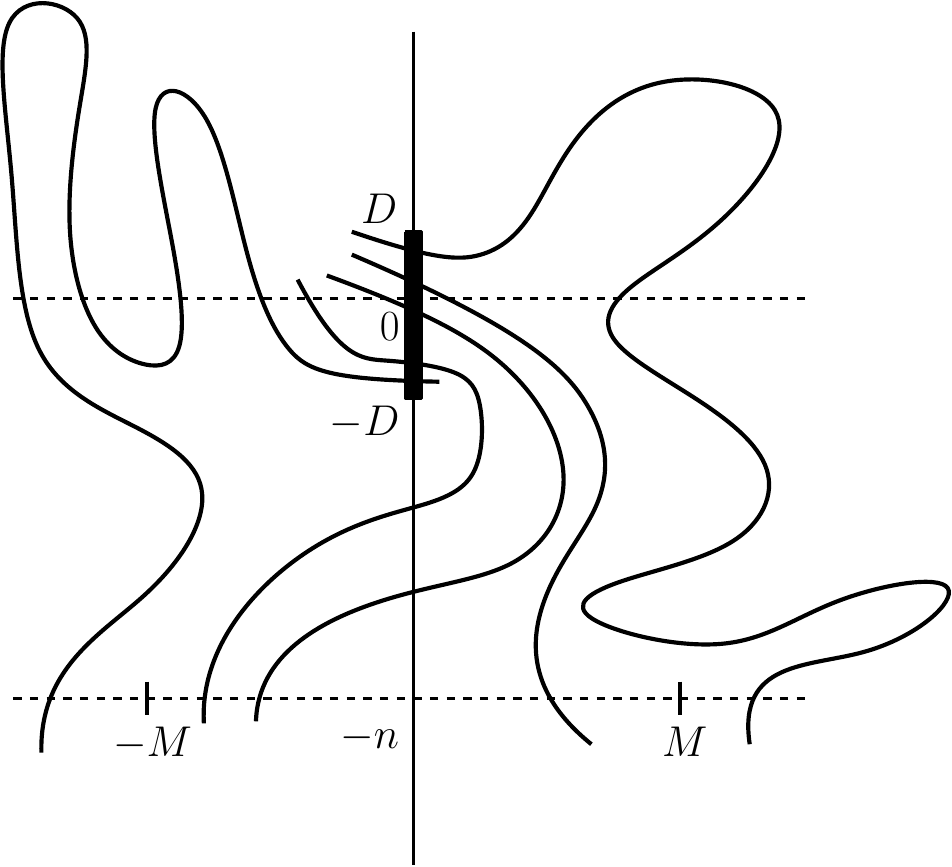}
\caption{a) The event $E^{\pi/2}_D \cap E^{-\pi/2}_D$ and some $n>n_0$. b) This event, together with the event $F_N(M)$ (here with $N=3$) translated by $(0,-n)$ for some $n>n_0$, implies the existence of $N+2$ zeros on the segment $\{ 0 \} \times [-D,D]$. Note that the two unbounded lines in Figure a) could belong to the same component of $\{ f = 0 \}$, but in that case $F_N(M)$ translated by $(0,-n)$ could not hold.
\label{fig:half}}
\end{figure}

Proposition \ref{prop:half} has the following corollary, which is all we shall need in the sequel:

\begin{corollary}\label{cor:positiveinhalfplane}
Let $f$ satisfy Assumption \ref{as:main} and consider the half-plane $H=\R \times \R_+$. Then almost surely the sets $\{ f \geq 0 \} \cap H$ and $\{ f \leq 0 \} \cap H$ have no unbounded connected component that intersects $\partial H$. In particular, for every $\ell \in \R$, almost surely the set $\{ f = \ell \} \cap H$ has no unbounded connected component that intersects $\partial H$.
\end{corollary}
\begin{proof}
Assume for the sake of contradiction that $\{ f \geq 0 \} \cap H$ has an unbounded connected component that intersects $\partial H$ with positive probability. Then by ergodicity this probability equals $1$. Moreover, by symmetry, this is also the case if we replace $\{ f \geq 0 \}$ by $\{ f \leq 0 \}$. By Lemma \ref{l:smoothC1mani}, this implies that almost surely the set $\{ f = 0 \} \cap H$ has an unbounded connected component that intersects $\partial H$ (in order to prove this, just note that there exists a point $x \in \partial H$ that is at the boudary of both an unbounded component of $\{ f > 0\}$ and an unbounded component of $\{f<0\}$, and notice that the nodal line starting from $x$ necessarily escapes $H$ at infinity by the Jordan curve theorem), which contradicts  Proposition \ref{prop:half}.
\end{proof}

\subsection{Application to saddle delocalisation: Proof of \eqref{e:sigrec}}

We now show that \eqref{e:sigrec} follows from Corollaries \ref{cor:4arm} and \ref{cor:positiveinhalfplane}. We begin with the following topological lemma:

\begin{lemma}\label{lem:saddles}
Let $\mathfrak{D}=(D,\calF,A)$ be a crossing domain where $D$ is a rectangle, let $S_0$ and $S_2$ be its distinguished sides, and recall that $S_1$ and $S_3$ denote the two components of $\partial D \setminus (\overline{S}_0 \cup \overline{S}_2)$. For $i \in \Z$, we let $S_i:=S_{i \! \text{ mod } 4}$. Then the following holds almost surely: \begin{itemize}
\item If $\sad_A \notin \partial D$ there are four disjoint (except at $\sad_A$) injective paths $\gamma_0,\ldots,\gamma_3$ included in $\{ f = -\thr_A \} \cap D$ from $\sad_A$ to $\partial D$ such that, for each $i \in \{0,\ldots,3\}$, the end point of $\gamma_i$ belongs to $\overline{S}_i \cup \overline{S}_{i+1}$.
\item Let $i_0 \in \{0,\ldots,3\}$. If $\sad_A \in \overline{S}_{i_0}$ there are two disjoint (except at $\sad_A$) injective paths $\gamma^-,\gamma^+$ included in $\{ f = -\thr_A \} \cap D$ from $\sad_A$ to $\partial D$ such that the end point of $\gamma^\pm$ belongs to $\overline{S}_{i_0\pm 1} \cup \overline{S}_{i_0\pm 2}$ (note that one of these paths might consist of a single point in the case that $\sad_A$ is an endpoint of the interval $S_{i_0}$).
\end{itemize}
\end{lemma}
\begin{proof}
By Lemma \ref{lem:crossing_at_the_threshold} there exists a path $\eta$, passing through $\sad_A(f)$, connecting $\overline{S}_0$ and $\overline{S}_2$ in $D \cap (\{f+\thr_A(f) > 0\} \cup \{ \sad_A(u) \})$. Replacing $S_0$ and $S_2$ by $S_1$ and $S_3$ in the definition of $A$, we obtain $A'$ another set of functions such that $(D,\calF,A')$ is a crossing domain with distinguished sides $S_1$ and $S_3$. The set $A'$ has the property that $\thr_{A'}(-f)=-\thr_A(f)$ (to prove this, one can for instance apply the analogue of Lemma \ref{l:smoothC1mani} at all rational levels $\ell$) and $\sad_{A'}(-f)=\sad_A(f)$. By Lemma \ref{lem:crossing_at_the_threshold}, we obtain a path $\eta'$, passing through $\sad_A(f)$, connecting $\overline{S}_1$ and $\overline{S}_3$ in $D \cap (\{f+\thr_A(f) < 0\} \cup \{ \sad_A(u) \})$.

Assume that $\sad_A(f)\notin\partial D$ and fix $i\in\{0,1,2,3\}$. The paths $\eta$ and $\eta'$ contain two paths $\eta_i$ and $\eta_{i+1}$, connecting $\sad_A(f)$ to $\overline{S}_i$ and $\overline{S}_{i+1}$ respectively, with the property that $f+\thr_A(f)$ is positive on $\eta_i \setminus \{ \sad_A(u) \}$ and negative on $\eta_{i+1}\setminus \{ \sad_A(u) \}$ or vice versa. Let $x_i$ and $x_{i+1}$ be the endpoints of these two paths that belong to $\partial D$ and let $I_i$ be the topological segment in $\overline{S}_i\cup\overline{S}_{i+1}$ bounded by $x_i$ and $x_{i+1}$. Then, the bounded region of the plane bounded by $\eta_i$, $\eta_{i+1}$ and $I_i$ must contain a path $\gamma_i$ connecting $\sad_A(f)$ to $I_i\subset \overline{S}_i\cup\overline{S}_{i+1}$ in $\{f+\thr_A(f)=0\}$ (to prove this rigorously, one can use the first item of Lemma \ref{lem:normal_form}). Since this is true for any $i\in\{0,1,2,3\}$, this covers the first point of the lemma.

For the second point, one reasons analogously except that either $\eta$ or $\eta'$ now has $\sad_A(f)$ as an endpoint (see Figure \ref{fig:saddles}). We omit any further details.
\end{proof}

\begin{figure}[h!]
\centering
\includegraphics[scale=0.6]{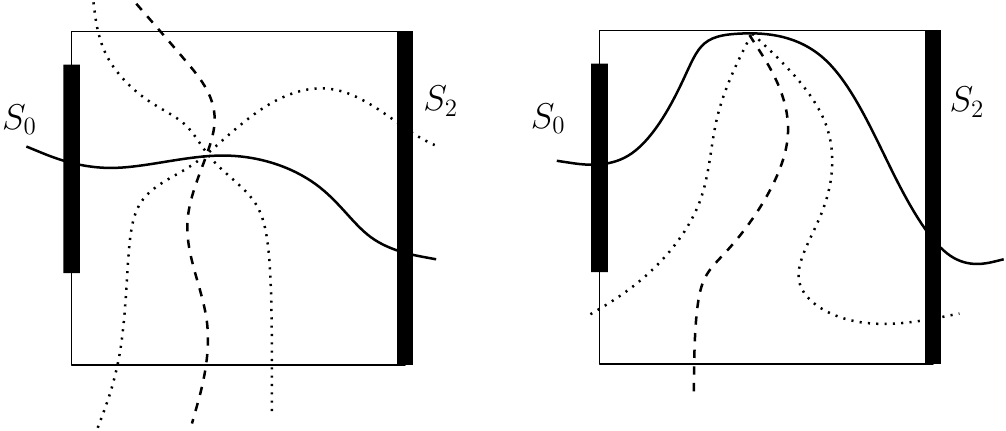}
\caption{The paths from (the proof of) Lemma \ref{lem:saddles}. The paths $\eta$ and $\eta'$ are respectively in full and dashed lines. The paths included in $\{f = -\thr_A\}$ (i.e.\ the $\gamma_i$'s or the $\gamma^\pm$'s) are in dotted lines.}
\label{fig:saddles}
\end{figure}

We state the following corollary of Lemma \ref{lem:saddles}, which links the lemma to the quantity $d_0$ from Proposition \ref{prop:delocalization}:

\begin{corollary}\label{cor:faces}
Let $\mathfrak{D}=(D,\calF,A)$ be as in Lemma \ref{lem:saddles}, and let $d_0$ be defined as in Proposition~\ref{prop:delocalization}. Then the diameter of the connected component of $\{ f = -\thr_A \} \cap D$ that contains $\sad_A$ is at least $d_0$.
\end{corollary}

We are now ready to prove \eqref{e:sigrec}, which completes the proof of Proposition \ref{prop:delocalization} since \eqref{e:sigann} was proven in Section \ref{ss:threshold_delocalizes}. Recall the definition of $\sigma_r(A)$ in \eqref{eq:sigma}.

\begin{proof}[Proof of \eqref{e:sigrec} (and hence of Proposition  \ref{prop:delocalization})]
By the union bound it suffices to prove the result for $r=1$. Recall that $H = \R \times \R_+$. By Corollary \ref{cor:positiveinhalfplane}, there exists a sequence $M_s \to \infty$ as $s \to \infty$ such that, with probability tending to $1$ as $s \to \infty$, neither $\{ f \geq 0 \} \cap H$ nor $\{ f \leq 0 \} \cap H$ has a connected component of diameter larger than $s$ that intersects $[-M_s,M_s] \times \{ 0 \}$. Let us fix such a sequence. We can (and will) assume that $M_s \geq 1$ for every $s$.

\smallskip
Consider a ball $B_x(1)$ that intersects the rectangle $D$, and let $I_x = \{ y \in \partial D \, : \, \exists z \in B_x(1), \, |y-z| \leq M_{d_0} \}$. By Lemma \ref{lem:saddles} and Corollary \ref{cor:faces}, if $\sad_A \in B_x(1)$ then we are in one of the two following cases:
\begin{itemize}
\item[(i)] $\sad_A$ is an $M_{d_0}$-saddle point;
\item[(ii)] $\{ f = -\thr_A \} \cap D$ has a connected component of diameter at least $d_0$ that intersects $I_x$.
\end{itemize}
Item (i) has probability $O(M_{d_0}^{-1})$ by Corollary \ref{cor:4arm}. To deal with item (ii), note that there exists a universal integer $N>0$ such that $I_x$ is included in a union of at most $N$ segments of length at most $2M_{d_0}$ (here we use that $M_{d_0}\geq 1$). As a result, item (ii) has probability that tends to $0$ as $d_0 \to \infty$ by the definition of $M_{d_0}$. In both cases the probability is bound below by a function of $d_0$ that goes to $0$ and depends only on the field, which gives the claim.
\end{proof}

\medskip
\appendix

\section{Basic properties of smooth Gaussian fields}
\label{s:app}

\subsection{Counting critical points}

Since we work with $C^1$-smooth fields, the Kac--Rice formula provides an estimate for the number of critical points of $f$ lying on circles:

\begin{lemma}[Number of critical points on circles]\label{lem:O(n)}
Let $f$ satisfy Assumption \ref{as:main}. There exists $c>0$ such that, for each $R>0$, the expectation of the number of critical points of $f_{|\partial B_0(R)}$ is less than $cR$.
\end{lemma}
\begin{proof}
By the Kac--Rice formula (see \cite[Theorem 6.2 and Proposition 6.5]{azais_wschebor}), the number $\textup{Cr}(R)$ of critical points of $f|_{\partial B_0(R)}$ has mean
\[ \E[\textup{Cr}(R)]=\int_0^{2\pi}\frac{1}{\sqrt{2\pi|\partial_\theta^2\kappa(0)|}}\E \big[|\partial_\theta^2f(Re^{i\theta})|\, |\, \partial_\theta f(Re^{i\theta})=0 \big]Rd\theta ,\]
where $\partial_\theta$ is the derivative along the tangent direction to $\partial B_R(0)$ at $Re^{i\theta}$. In particular
\[ \E[\textup{Cr}(R)]\leq c_1\br{\frac{\sup_{|u|=1}|\partial_u^4\kappa(0)|}{\inf_{|u|=1}|\partial_u^2\kappa(0)|}}^{1/2}\times R , \]
where $c_1>0$ is a universal constant and the denominator is positive by Assumption \ref{as:main}.
\end{proof}

\subsection{The reproducing property of the covariance kernel}

We next recall the reproducing property of the covariance kernel of a Gaussian field; a more general treatment can be found in \cite[Chapter 8, Section 4]{janson}. Let $T$ be any set, let $H$ be a finite-dimensional vector space of functions on $T$  equipped with a scalar product $\langle\cdot,\cdot\rangle$ and let $g$ be the standard Gaussian vector on the Euclidean space $(H,\langle\cdot,\cdot\rangle)$. 

\begin{lemma}
\label{l:rkhs}
Let $K$ denote the covariance function of $g$. Then, for each $x\in T$, $K(x,\cdot)\in H$, and for each $u\in H$,
\begin{equation}\label{eq:reproducing_kernel}
\langle K(x,\cdot),u\rangle=u(x) .
\end{equation}
(In other words, $K$ is the reproducing kernel of the Hilbert space $(H,\langle\cdot,\cdot\rangle)$.)
\end{lemma}

\begin{proof} 
There exists an orthonormal basis $e_1,\dots,e_k$ of $H$, and i.i.d.\ standard normals $\xi_1,\dots,\xi_k$, such that $\xi_1 e_1+\dots \xi_k e_k$ has the law of $g$. In particular, $K(x,y)=\sum_{j=1}^k e_j(x)e_j(y)$, which implies that $K(x,\cdot)\in H$. Now, for each $u\in H$, there exist $u_1,\dots,u_k\in\R$ such that $u=u_1 e_1+\dots +u_k e_k$ and so
\begin{equation*}
\langle K(x,\cdot),u\rangle=\sum_{j=1}^ku_je_j(x)=u(x). \qedhere
\end{equation*}
\end{proof}

\medskip
\section{Stratified Morse functions}\label{sec:app_morse}

In this section, we prove successively that stratified perfect Morse functions are generic in the sense of probability under suitable assumptions (see Lemmas \ref{lem:cond_implies_morse} and \ref{lem:as_main_implies_cond}), that they are stable in the $C^2$-topology (see Lemma \ref{lem:morse_is_open}) and that, on crossing domains, crossings occur at the threshold height (see Lemma \ref{lem:crossing_at_the_threshold}).
\smallskip

Throughout, we fix a stratified domain $(D,\calF)$ (see Definition \ref{def:strat_domain}) and recall the set $\calM(D,\calF)$ of perfect Morse functions (Definition \ref{d:morse}).

\subsection{Stratified perfect Morse functions are generic}\label{ssec:generic}

We prove that, under Assumption~\ref{as:main}, $f$ is generically a perfect Morse function. We split the proof into two parts: first we show that a general Gaussian field satisfying a certain set of conditions (stated in Lemma~\ref{lem:cond_implies_morse} below) is in $\calM(D,\calF)$ almost surely; then we verify that Assumption~\ref{as:main} implies this set of conditions. We do this to isolate the properties that are essential for $f  \in \calM(D,\calF)$ from unrelated conditions in Assumption \ref{as:main} (such as stationarity and decay of correlations). Recall the notation $D^+$ from Definition \ref{def:strat_domain}.

\begin{lemma}[See {\cite[Lemma 2.10]{ri19}} for a similar result]\label{lem:cond_implies_morse}
Let $f : D^+ \to \mathbb{R}$ be a Gaussian field satisfying:
\begin{enumerate}
\item $f \in C^3(D^+)$ almost surely;
\item For each $x,y\in D$ distinct, the random vector 
\[ (f(x),f(y),\nabla_xf,\nabla_yf) \in \R^6 \]
 is non-degenerate;
\item For each $x\in D$, the random vector 
\[ (\nabla_x f,\nabla^2_x f) \in \R^2\times\textup{Sym}_2(\R)  \] 
is non-degenerate.
\end{enumerate}
Then $f\in\calM(D,\calF)$ almost surely. Moreover, if $(f_n)_{n\in\N}$ is a sequence of Gaussian fields on $D^+$ that are almost surely $C^3(D^+)$ and that converge almost surely to $f$ in $C^2(D^+)$ then $f_n$ satisfies the three assumptions of the lemma for sufficiently large $n$.
\end{lemma}

\begin{proof}
The conditions of the lemma imply that, for each $F,F_1,F_2\in\calF$, the random field $x \mapsto (\nabla_{x}(f|_{F}),\det(\nabla^2_{x}(f|_{F})))$ has a bounded density on $F$ and the random field $(x,y) \mapsto (f(x)-f(y),\nabla_{x}(f|_{F_1}),\nabla_{y}(f|_{F_2}))$ has a bounded density on any compact subset of $\{ (x,y) \in F_1\times F_2 \, : \, x \neq y \}$, and moreover the field $(\partial^\alpha (f|_{F_1}),\partial^\beta(f|_{F_2}))_{|\alpha|,|\beta|\leq 2}$ is $C^1$. The fact that $f\in\calM(D,\calF)$ follows by applying \cite[Lemma 11.2.10]{adler_taylor} to the following fields:
\begin{itemize}
\item $x\mapsto (\nabla_x(f|_F),\det(\nabla^2_x(f|_F)))$ defined over $F$, and $F$ ranges over $\calF$;
\item $(x,y)\mapsto (f(x)-f(y),\nabla_x (f|_{F_1})|,\nabla_y(f|_{F_2}))$ defined over a compact exhaustion of $\{ (x,y) \in F_1\times F_2 \, : \, x \neq y \}$, where $F_1$ and $F_2$ range over $\calF$;
\item $x\mapsto|\textup{pr}_{T_x\overline{F}_2} \nabla_x f|^2$ defined over $F_1$ where $(F_1,F_2)$ ranges over the set of pairs of faces such that  $F_1\neq F_2$ and $F_1\subset\overline{F_2}$ ($\textup{pr}_{T_x\overline{F}_2}$ is the orthogonal projection onto $T_x\overline{F}_2$). 
\end{itemize}
For the second part of the lemma, we observe that if $(f_n)_n$ converges a.s.\ in $C^2$ towards $f$ then the covariance kernels of $f_n$ converge in $C^{2,2}$ to the covariance kernel of $f$. Since $D$ is compact, the conditions of Lemma \ref{lem:cond_implies_morse} are open with respect to the $C^{2,2}$ topology on the covariance, which gives the result.
\end{proof}

\begin{lemma}
\label{lem:as_main_implies_cond}
Let $f$ satisfy Assumption \ref{as:main}. Then $f$ satisfies the conditions in Lemma \ref{lem:cond_implies_morse}.
\end{lemma}

\begin{proof}
The first and second conditions follow from the first and second points of Assumption~\ref{as:main} combined with stationarity, so it remains to show the third condition. By stationarity, $\nabla_x f$ is independent from $\nabla^2_xf$ and by the second point $\nabla_x f$ is non-degenerate. If $\nabla^2_xf$ is degenerate, then by Lemma A.1 of \cite{bmm19}, the spectral measure of $f$ is supported on two lines $L_1,L_2$ through the origin. As a result, $f$ is the sum of two (independent) Gaussian fields $g_1,g_2$ such that $g_1$ (resp. $g_2$) is constant on each line orthogonal to $L_1$ (resp. $L_2$). Without loss of generality we can assume that $L_1 \neq L_2$. Let $u$ and $v$ be two unit vectors orthogonal to $L_1$ and $L_2$ respectively. For every $x \in \R^2$ we have
\[
(\partial_uf(0),\partial_vf(0),\partial_uf(x),\partial_vf(x))=(\partial_u g_2(0),\partial_v g_1(0),\partial_u g_2(x),\partial_v g_1(x)).
\]
Moreover, if $x=u$ then $\partial_v g_1(0) = \partial_v g_1(x)$, so the above Gaussian vector is degenerate, which contradicts the second point of Assumption~\ref{as:main}.
\end{proof}

\subsection{Stratified perfect Morse functions are stable}
We prove a stability result for perfect Morse functions. Recall the notation $T_x\overline{F}$ from Definition \ref{d:morse}.

\begin{lemma}\label{lem:morse_is_open}
The set $\calM(D,\calF)$ is open in $C^2(D^+)$.
\end{lemma}
\begin{proof}
We show that the complement is closed. Let $(u_n)_n$ be a sequence of functions belonging to $C^2(D^+)\setminus\calM(D,\calF)$ which converges to a $u\in C^2(D^+)$. Up to extraction we may assume that one of the three following conditions holds for all $n\in\N$: (i) the function $u_n$ has two distinct stratified critical points $x_n\in F_1$ and $y_n\in F_2$ with the same critical values, where $F_1,F_2\in\calF$ do not depend on $n$; (ii) the function $u_n$ restricted to some $F\in\calF$ (independent of $n$) has a critical point on which its Hessian is degenerate; (iii) there exist $F,F'\in\calF$ such that $F\subset\overline{F'}$, $F\neq F'$, $u_n$ has a stratified critical point $x_n\in F$ such that $\nabla_{x_n} u_n$ is orthogonal to $T_{x_n}\overline{F'}$. We deal with each case separately:
\begin{itemize}
\item[(i)] Up to extraction the sequences $(x_n)_n$ and $(y_n)_n$ converge to limits $x$ and $y$. Assume that $x$ belongs to some $F_3\in\calF$ different from $F_1$. Then $F_3\subset \overline{F_1}$ (this may be checked in each case of Definition \ref{def:strat_domain}), and $\nabla_x u$ is orthogonal to $T_x\overline{F_1}$ so $u\notin\calM(D,\calF)$. From now on let us assume that $x\in F_1$ and $y\in F_2$. Suppose that $x=y$. Then $F_1=F_2$ and $u_n|_{F_1}$ has two critical points close to each other and the Hessian of $u$ at $x$ must be degenerate. Finally, if $x\neq y$, $u$ has two distinct critical points $x$ and $y$ for which $u(x)=u(y)$.
\item[(ii)] Up to extraction we may assume that $(x_n)_n$ converges to some $x\in D$. Reasoning as above, we may assume that $x\in F$, and so the Hessian of $u$ at $x$ is degenerate.
\item[(iii)] The argument is analogous to (ii). \qedhere
\end{itemize}
\end{proof}

\subsection{Level set topology for stratified perfect Morse functions}

Fix a crossing domain $(D,\calF,A)$ (see Definition \ref{def:conn_domain}), with distinguished sides $S_0,S_2\subset\partial D$ if $D$ is a rectangle. We prove the following lemma, which implies that $u+\thr_A(u)\in A$ for any $u\in\calM(D,\calF)$ (i.e.\ crossings occur at the threshold height):

\begin{lemma}\label{lem:crossing_at_the_threshold}
For every $u\in\calM(D,\calF)$, the set of $\ell\in\R$ for which $u+\ell\in A$ is closed. In other words,
\begin{itemize}
\item If $D$ is a rectangle, there exists a continuous path in $D\cap (\{u+\thr_A(u) > 0\} \cup \{ \sad_A(u) \} )$ (recall that $u(\sad(u))+\thr_A(u)=0$), passing through $\sad_A(u)$, connecting $\overline{S}_0$ and $\overline{S}_2$.
\item If $D$ is an annulus, there exists a circuit in $D\cap (\{u+\thr_A(u) > 0\} \cup \{ \sad_A(u) \} )$ separating the inner disc from infinity.
\end{itemize}
\end{lemma}
To prove Lemma \ref{lem:crossing_at_the_threshold} we make use of an explicit description of the neighbourhood of a stratified critical point:
\begin{lemma}\label{lem:normal_form}
Let $u\in \calM(D,\calF)$. Assume that $0\in D$ is a stratified critical point of $u$ that belongs to a face $F\in\calF$ and has critical value $u(0)=0$.
\begin{itemize}
\item If $\textup{dim}(F)=2$, there exist $\alpha,\beta\in\{-1,+1\}$ and a local diffeomorphism $\psi$ of $\R^2$ at $0$ with $\psi(0)=0$ such that $u\circ\psi(x_1,x_2)=\alpha x_1^2+\beta x_2^2$.
\item If $\textup{dim}(F)=1$, there exist $\gamma\in\R$, $h\in C^1(\R^2)$ with $h(0)\neq 0$, and a local diffeomorphism $\psi$ of $\R^2$ at $0$ with $\psi(0)=0$ mapping a neighbourhood of $0$ in $\{0\}\times\R$ to a neighbourhood of $0$ in $F$, such that
\[u\circ\psi(x_1,x_2)=h(x_1,x_2)x_1+\gamma x_2^2 .\]
Moreover, $\psi$ may be chosen to map the unique two-dimensional face (near $0$) onto (a neighbourhood of $0$ in) $(0,+\infty)\times\R$.
\item If $\textup{dim}(F)=0$, $F$ is contained in the closure of two faces of dimension one, which we denote by $F'$ and $F''$ and $\nabla_0u$ is not orthogonal to $F'$ or $F''$.
\end{itemize}
\end{lemma}
\begin{proof}
The first point is the classical Morse lemma. The third point is a consequence of the definition of $\calM(D,\calF)$.
This leaves the second point of the lemma. To begin with, we note that $0$ remains a non-degenerate stratified critical point if $u$ is composed on the right by a local diffeomorphism. In particular, we may assume that $F$ is an open interval in $\{0\}\times\R$. Next, since $0$ is a non-degenerate critical point of $u(0,\cdot)$, we apply the Morse lemma to this function and deduce the existence of a local diffeomorphism $\psi$ such that
\[
u\circ\psi(0,x_2)=\gamma x_2^2
\]
for some $\gamma\in\{\pm 1\}$. We then define $h(x_1,x_2)=\frac{u\circ\psi(x_1,x_2)-u\circ\psi(0,x_2)}{x_1}$ near $0$ and obtain the required expression.
\end{proof}

We are now ready to prove Lemma \ref{lem:crossing_at_the_threshold}.

\begin{proof}[Proof of Lemma \ref{lem:crossing_at_the_threshold}]
Without loss of generality, we may assume that $\ell=0$, that $\sad_A(u)=0\in D$ and that $u$ has no stratified critical points at which it takes values in $(-1,1)\setminus\{0\}$. We focus first on the case where $D$ is a rectangle. For each $\delta>0$, by using locally the implicit function theorem, one can find $\eps>0$ and a continuous map $H:(D\cap\{u+\eps\geq 0\})\times [0,1]\to D\cap\{u+\eps\geq 0\}$ such that $H(\cdot,0)$ is the identity and $H(\cdot,1)$ takes values in $D\cap(\{u-\eps\geq 0\}\cup B(0,\delta))$. In addition, one may ask that each face in $\calF$ is mapped into itself by $H(\cdot,1)$.
\smallskip

Since $\eps>0=\thr_A(u)$, one can find a path $\gamma$ in $D\cap\{u+\eps\geq 0\}$ connecting $\overline{S}_0$ and $\overline{S}_2$. The path $H(\gamma,1)$ connects these two faces in $D\cap(\{u - \eps \geq 0\}\cup B(0,\delta))$. Note that $H(\gamma,1)$ intersects $B(0,\delta)$ because there is no path in $D\cap\{u-\eps\geq 0\}$ connecting $\overline{S}_0$ and $\overline{S}_2$. In order to conclude, it is enough to show that for some (small enough) value of $\delta>0$, $(\{u>0\}\cap B(0,\delta)\cap D) \cup \{0\}$ is path-connected. This follows immediately from Lemma \ref{lem:normal_form} above. 
\smallskip

In the case where $D$ is an annulus, the same proof holds by considering a circuit separating the inner disk from infinity in place of $\gamma$.
\end{proof}

\medskip
\section{RSW theory (by Laurin Köhler-Schindler)}\label{a:rsw}

In this section, we prove the RSW results Propositions \ref{p:gc1} and \ref{p:gc2}, based on a quantitative improvement of the approach in \cite{tas16} that has been developed together with V.~Tassion.

\smallskip
Recall that $\Cross_\ell(R,S)$ denotes the rectangle crossing event that $\{f \ge -\ell\} \cap [0,R] \times [0, S]$ contains a path from $\{0\} \times [0,S]$ to $\{R\} \times [0,S]$. Let us begin by observing that, by self-duality and symmetry, squares are crossed with probability equal to $1/2$ at level $\ell = 0$:

\begin{lemma}
\label{l:sc}
For all $R > 0$, $\prob[\Cross_0(R,R)] = 1/2$.
\end{lemma}

\begin{proof}
Let $\Cross^*_0(R,R)$ denote the event that there is a top-bottom crossing of $[0,R]^2$ by a path included in $\{ f \leq 0 \} \cap [0,R]^2$. By Lemma \ref{l:smoothC1mani}, $\Cross_0(R,R)$ and $\Cross^*_0(R,R)$ partition the probability space up to a null set. Since they also have equal probability (by $D_4$-symmetry and the equality in law of $f$ and $-f$) the result follows.
\end{proof}

Recall next the crossing domain $\mathfrak{D}(R; a,b) = (D(R), \mathcal{F}(R;a,b),A(R;a,b))$ for which $D(R)=[0,R]^2$, and $\Cross_\ell(A(R;a,b))$ is the event that there is a path in $\{f \ge -\ell\} \cap [0,R]^2$ from $\{0\} \times [0,R]$ to $\{R\} \times [a,b]$ (see Figure \ref{f:hevent}). To ease notation, we shall abbreviate $\Cross_\ell(R;a,b) = \Cross_\ell(A(R;a,b))$. Note also that $\Cross_\ell(R; 0, R) = \Cross_\ell(R, R)$. 

Let us also note that in the definition of crossing domains, the sides are intervals of dimension $1$ and not points. However, in this section we also consider the events $\Cross_\ell(R;a,a)$, which are defined in the obvious way, i.e.\ $\Cross_\ell(R;a,a)$ is the event that there is a path in $\{ f \ge -\ell \} \cap [0,R]^2$ from $\{0\} \times [0,R]$ to $\{(R,a)\}$.

\smallskip
The probability $\prob[\Cross_0(R; a,b)]$ is a continuous non-increasing function of $a$ and a continuous non-decreasing function of $b$. 
This allows us to define, for all $R>0$,
\begin{equation*}
 \alpha_R  = \min \left\{ \alpha \in [0,R] : \prob\left[ \Cross_0\left(R; \frac{R-\alpha}{2},\frac{R+\alpha}{2}\right) \right] \ge 1/4 \right\},  
  \end{equation*}
which is well-defined thanks to Lemma \ref{l:sc}.
As a direct consequence of this definition,
 \begin{equation}
 \label{e:ar}
\prob\left[ \Cross_0\left(R;\frac{R}{2},\frac{R+\alpha_R}{2}\right) \right] \ge 1/8, \quad \text{and} \quad \prob\left[ \Cross_0\left(R;\frac{R+\alpha_R}{2},R\right) \right] \ge 1/8,
 \end{equation}
where the bounds are obtained by reflection symmetry along the horizontal axis and Lemma \ref{l:sc} if $\alpha_R > 0$, and by inclusion if $\alpha_R = 0$.
We observe that $\alpha_R$ tends to infinity. While this might not be true in the general setting of \cite{tas16}, in the setting of this paper it follows directly from Corollary \ref{cor:positiveinhalfplane}. So we have the following.
\begin{lemma}
\label{l:ar}
As $R \to \infty$, $\alpha_R \to \infty$.
\end{lemma}
 
We next present three deterministic `gluing' constructions which form the core of the argument. In the following lemma and two propositions, we replace $\{f \ge - \ell\}$ with an arbitrary deterministic set $\mathcal{C} \subset \mathbb{R}^2$, and write $\Cross(R;\alpha,\beta)$ and $\Cross(a, b)$ to denote the analogous connection events with respect to this closed set. (Note that here the word ``event'' is an abuse of terminology.) For brevity we shall refer to arbitrary translations, rotations by $\pi/2$, and reflections in the vertical and horizontal axes of these events as `copies'. In addition, we introduce the $X$-event that will be important to construct crossings of long rectangles by connecting two copies of $\Cross(R;a,b)$. The event $X(R;c)$ for $c\ge 0$ denotes the occurrence of a path connected component of $\calC \cap [0,R]\times \R$ that intersects $\{0\}\times(-\infty,0]$, $\{0\}\times[c,+\infty)$,  $\{R\}\times(-\infty,0]$, and $\{R\}\times[c,+\infty)$ (see Figure \ref{f:X-e}).

\begin{figure}[h!]
	\begin{subfigure}[b]{0.5\textwidth}
		\includegraphics[width=\textwidth]{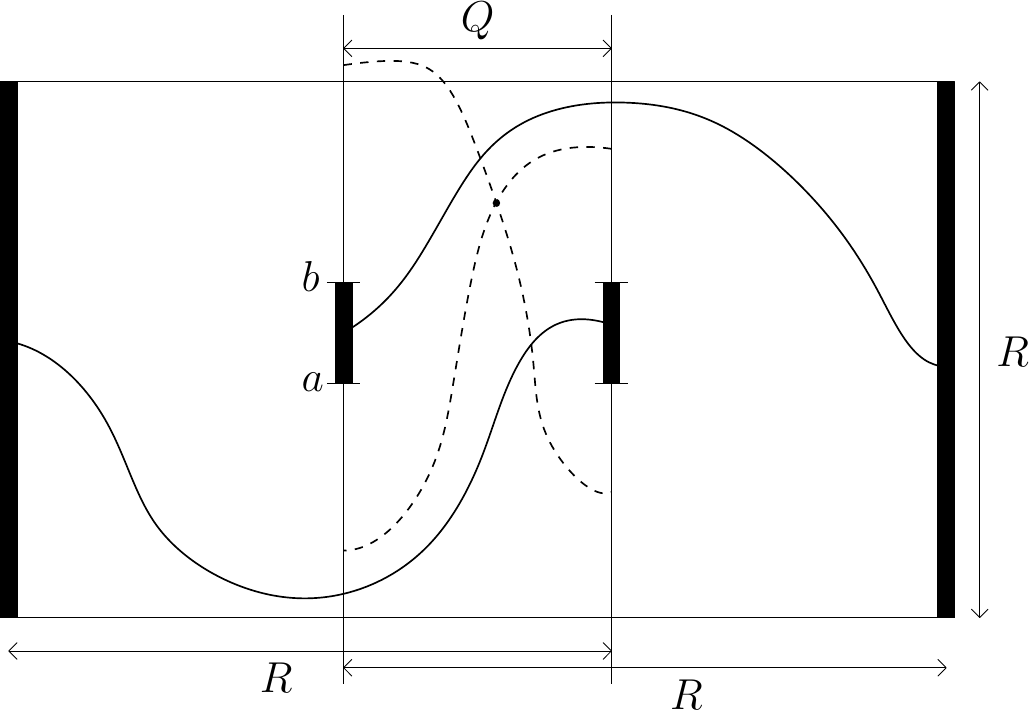}
		\caption{}
		\label{f:CXC_0}
	\end{subfigure}
	\quad
	\begin{subfigure}[b]{0.2\textwidth}
		\includegraphics[width=\textwidth]{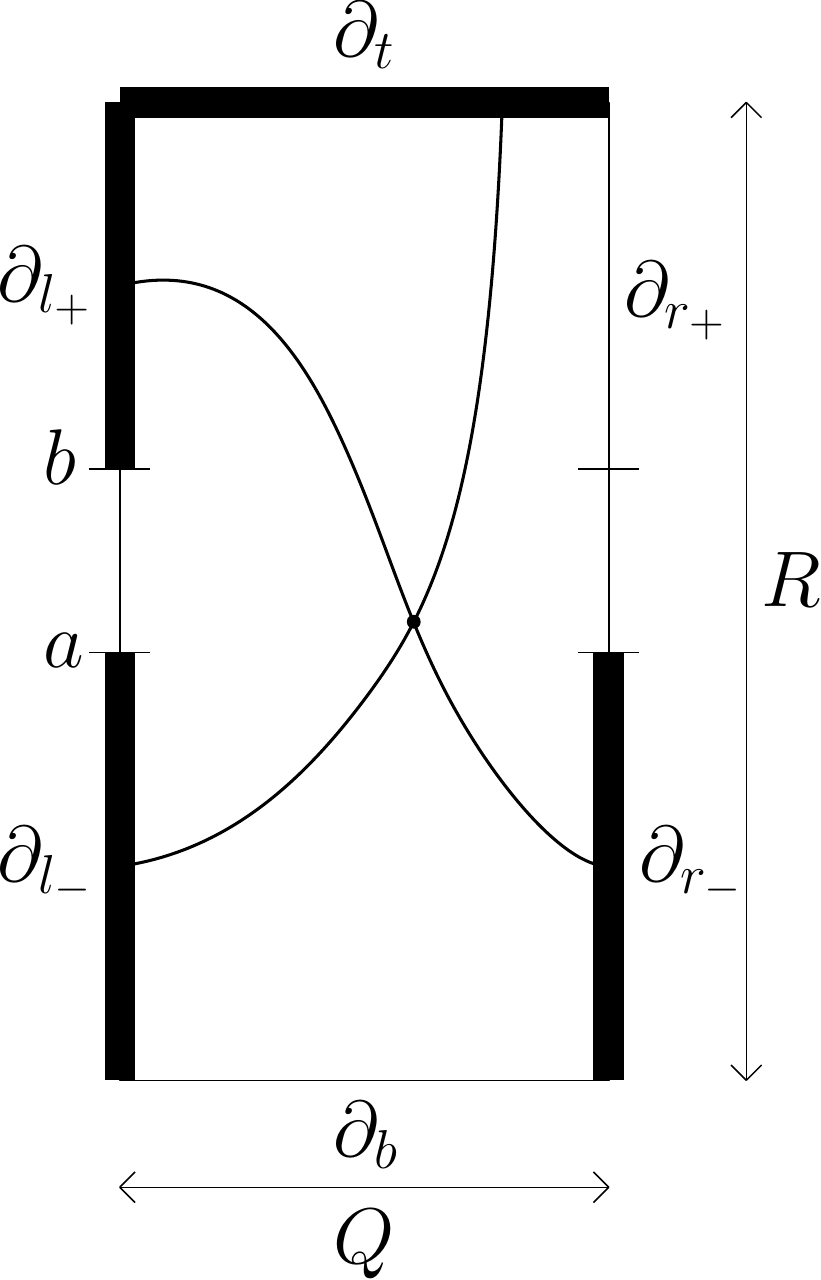}
		\caption{}
		\label{f:CXC_1}
	\end{subfigure}
	\quad	
	\begin{subfigure}[b]{0.2\textwidth}
		\includegraphics[width=\textwidth]{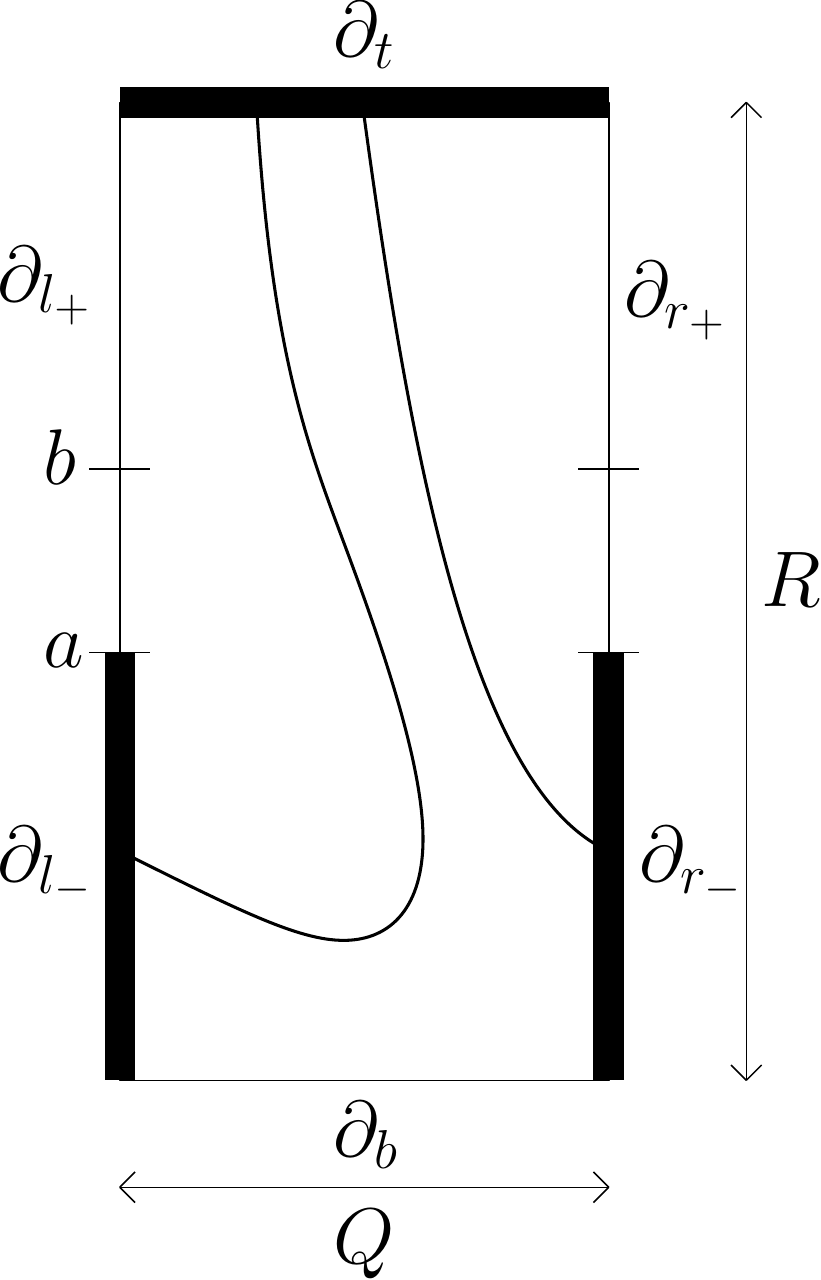}
		\caption{}
		\label{f:CXC_2}
	\end{subfigure}
	\caption{ (a) $\Cross$-$X$-$\Cross$ construction, the events (b) $E^{l_+,r_-}_{l_-,t}$, and (c) $E^{t,r_-}_{l_-,t}$.}
\end{figure}

 \begin{lemma}[$\Cross$-$X$-$\Cross$ construction]
\label{l:gc3}
Let $0 \le Q\le R$ and $a \le b$ with $a,b \in [0,R]$. There are two copies of the event $\Cross(R;a,b)$ and one copy of the event $X(Q;b-a)$, whose intersection is contained in 
\begin{equation*}
	\Cross(2R-Q,R).
\end{equation*}
\end{lemma}

\begin{proof}
	For an event $E$, let us denote by $x + E$ its translation by some $x \in \R^2$ and by $v \boldsymbol{\cdot} E$ its reflection along the vertical axis. We will show (see Figure \ref{f:CXC_0})
	\begin{equation}
		\Cross(R;a,b)  \cap \Big(x_0 + X(Q;b-a)\Big) \cap \Big( x_1 + \big(v\boldsymbol{\cdot} \Cross(R;a,b)\big)\Big)  \subset \Cross(2R-Q,R),
		\label{e:cxc}
	\end{equation}
	where $x_0 :=(R-Q,a)$ and $x_1:=(2R-Q,0)$, and we start by observing that
	\begin{align*}
		\Big(x_0 + X(Q;b-a)\Big) \subset & \Big\{ \exists \text{ path in } [R-Q,R] \times [0,R] \text{ from } \partial_{l_+} \cup \partial_t \text{ to }  \partial_b \cup \partial_{r_-}   \Big\}\\ \cap&
		\Big\{ \exists \text{ path in } [R-Q,R] \times [0,R] \text{ from } \partial_{l_-} \cup \partial_b \text{ to }  \partial_t \cup \partial_{r_+}   \Big\},
	\end{align*}
	where $\partial_{l_-} := \{R-Q\}\times[0,a]$, $\partial_{l_+} := \{R-Q\}\times[b,R]$, $\partial_{r_-} := \{R\}\times[0,a]$, $\partial_{r_+} := \{R\}\times[b,R]$, $\partial_t := [R-Q,R]\times\{R\}$, and $\partial_b := [R-Q,R]\times\{0\}$ denote segments of the boundary. From there, it is sufficient to argue \eqref{e:cxc} separately for all 16 events 
	\begin{equation*}
		E^{g,h}_{i,j} := \Big\{ \exists \text{ paths in } [R-Q,R] \times [0,R] \text{ from } \partial_{g} \text{ to }  \partial_h \text{, and from }  \partial_{i} \text{ to }  \partial_{j}   \Big\},
	\end{equation*}
	with $g \in \{l_+,t\}$, $h \in \{b,r_-\}$, $i \in \{l_-,b\}$, and $j \in \{r_+,t\}$.
	
	On the 9 events $E^{t,b}_{i,j}$, $E^{g,h}_{t,b}$, $E^{l_+, b}_{l_-, t}$, and $E^{t, r_-}_{b, r_+}$, there is a path in $[R-Q,R]\times[0,R]$ from top to bottom and so \eqref{e:cxc} is clearly satisfied. 
	
	On the 5 events $E^{l_+,r_-}_{l_-,r_+}$, $E^{l_+,r_-}_{l_-,t}$, $E^{l_+,r_-}_{b,r_+}$, $E^{l_+,b}_{l_-,r_+}$, and $E^{t,r_-}_{l_-,r_+}$, the two paths must intersect and so there exists a component connecting all four segments of the boundary (see Figure \ref{f:CXC_1}). In each case, one sees that the path in $\Cross(R;a,b)$ must intersect this component, and the same holds for $( x_1 + (v\boldsymbol{\cdot} \Cross(R;a,b)))$ by symmetry. This implies \eqref{e:cxc}.
	
	Finally, we consider the two events $E^{t,r_-}_{l_-,t}$ and $E^{l_+,b}_{b,r_+}$. The argument is identical, and so we present it for $E^{t,r_-}_{l_-,t}$ (see Figure \ref{f:CXC_2}). On $\Cross(R;a,b) \cap E^{t,r_-}_{l_-,t} \cap ( x_1 +( v\boldsymbol{\cdot} \Cross(R;a,b)))$, let
	\begin{align*}
		\gamma_1 & \text{ be a path in } [0,R]^2 \text{ from } \{0\} \times [0,R] \text{ to }  \{R\}\times [a,b], \\
		\gamma_2 & \text{ be a path in } [R-Q,R] \times [0,R] \text{ from } \partial_{l_-} \text{ to }  \partial_t, \\     
		\gamma_3 & \text{ be a path in } [R-Q,R] \times [0,R] \text{ from } \partial_{t} \text{ to }  \partial_{r_-} \text{, and } \\    
		\gamma_4 & \text{ be a path in } [R-Q,2R-Q] \times [0,R] \text{ from } \{R-Q\} \times [a,b] \text{ to }  \{2R-Q\}\times [0,R].        
	\end{align*}
	If the path $\gamma_1$ intersects the path $\gamma_2$, this implies \eqref{e:cxc} since the paths $\gamma_2$ and $\gamma_4$ always intersect. But if $\gamma_1$ and $\gamma_2$ are disjoint, then $\gamma_1 \cup \gamma_3$ must itself contain a path from $\partial_{l_-}$ to $\partial_t$ in $[R-Q,R]\times[0,R]$. Hence, $\gamma_1 \cup \gamma_3$ intersects the path $\gamma_4$, which implies \ref{e:cxc} and concludes the proof.
\end{proof}

\begin{figure}[h!]
	\centering
	\begin{minipage}{.5\textwidth}
		\centering
		\includegraphics[width=.3\linewidth]{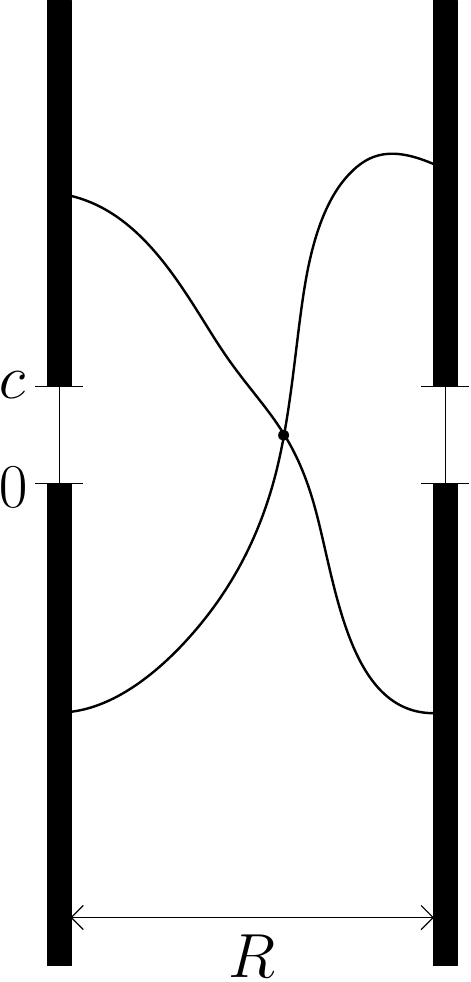}
		\captionof{figure}{The event $X(R;c)$.}
		\label{f:X-e}
	\end{minipage}%
	\begin{minipage}{.5\textwidth}
		\centering
		\includegraphics[width=.8\linewidth]{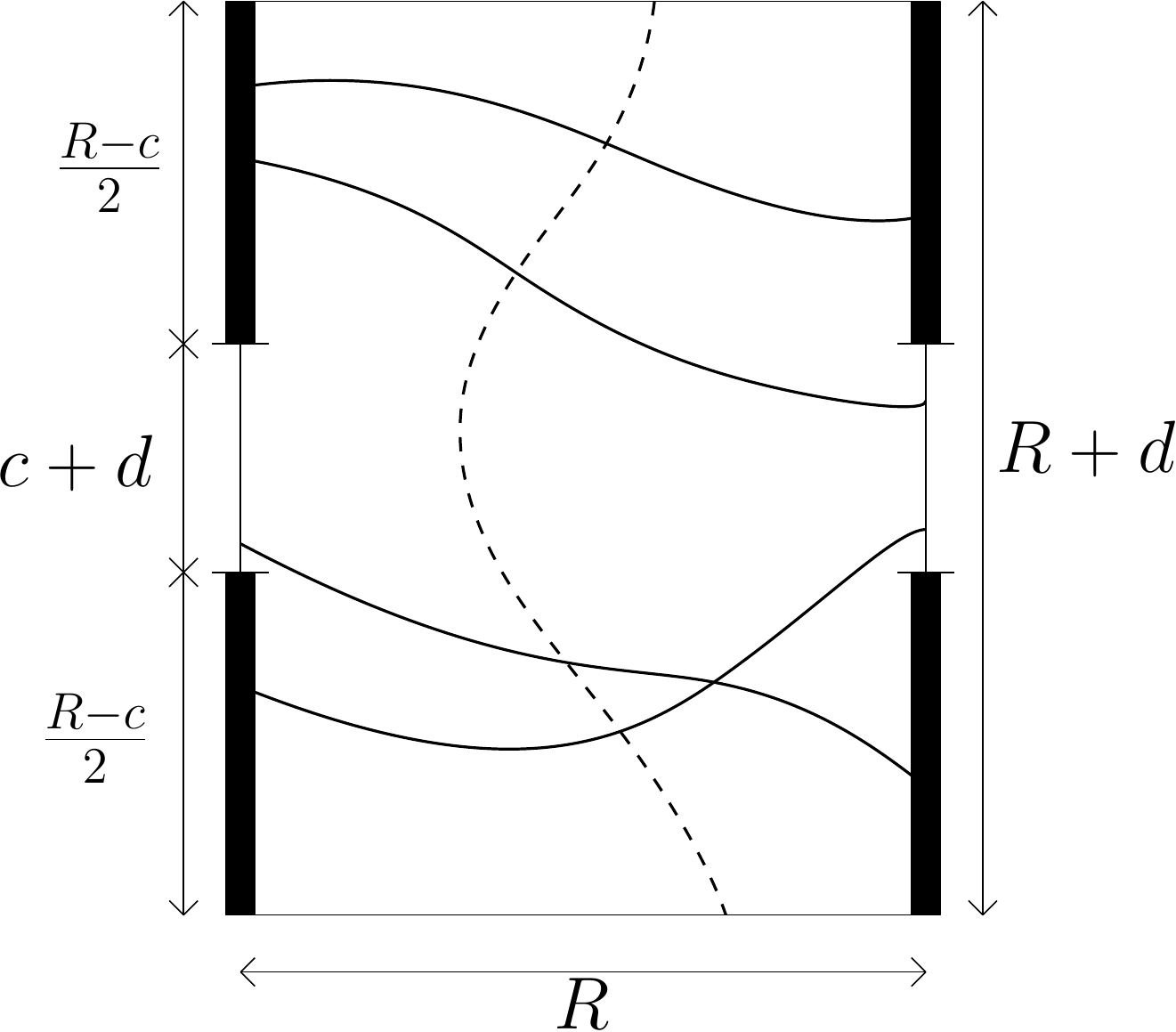}
		\captionof{figure}{Illustration of Proposition \ref{p:gc4}.}
		\label{f:X-cs}
	\end{minipage}
\end{figure}

\begin{proposition}
\label{p:gc4}
Let $R \ge 0$, $c \in [0,R]$, and $d\ge 0$. There are four copies of the event $\Cross(R;\frac{R+c}{2},R)$ and one copy of the event $\Cross(R+d,R)$, whose intersection is contained in
\begin{equation*}
X(R;c+d).
\end{equation*}
\end{proposition}
The proof follows directly from considering a $\pi/2$-rotated copy of the event $\Cross(R+d,R)$ and we refer to Figure \ref{f:X-cs} for an illustration.

\begin{proposition}
\label{p:gc5}
Let $R \ge 0$. There are ten copies of the event $\Cross(R;5R/8,R)$ and five copies of the event $\Cross(R,R)$, whose intersection is contained in
\begin{equation*}
	\Cross\left(5R/4,R \right).
\end{equation*}
\end{proposition}

\begin{proof}
By intersecting two copies of $\Cross(R;5R/8,R)$ with a copy of $\Cross(R,R)$ as shown on the left in Figure \ref{f:long}, we deduce the existence of a path inside an $R \times R$ square that surrounds a centred boundary segment of length $R/4$. Then, as on the right in Figure \ref{f:long}, the intersection of $5$ copies of this event induces a crossing of a $5R/4  \times R$ rectangle. 
\end{proof}

\begin{figure}[h!]
	\centering
	\begin{minipage}{.5\textwidth}
		\centering
		\includegraphics[width=.7\linewidth]{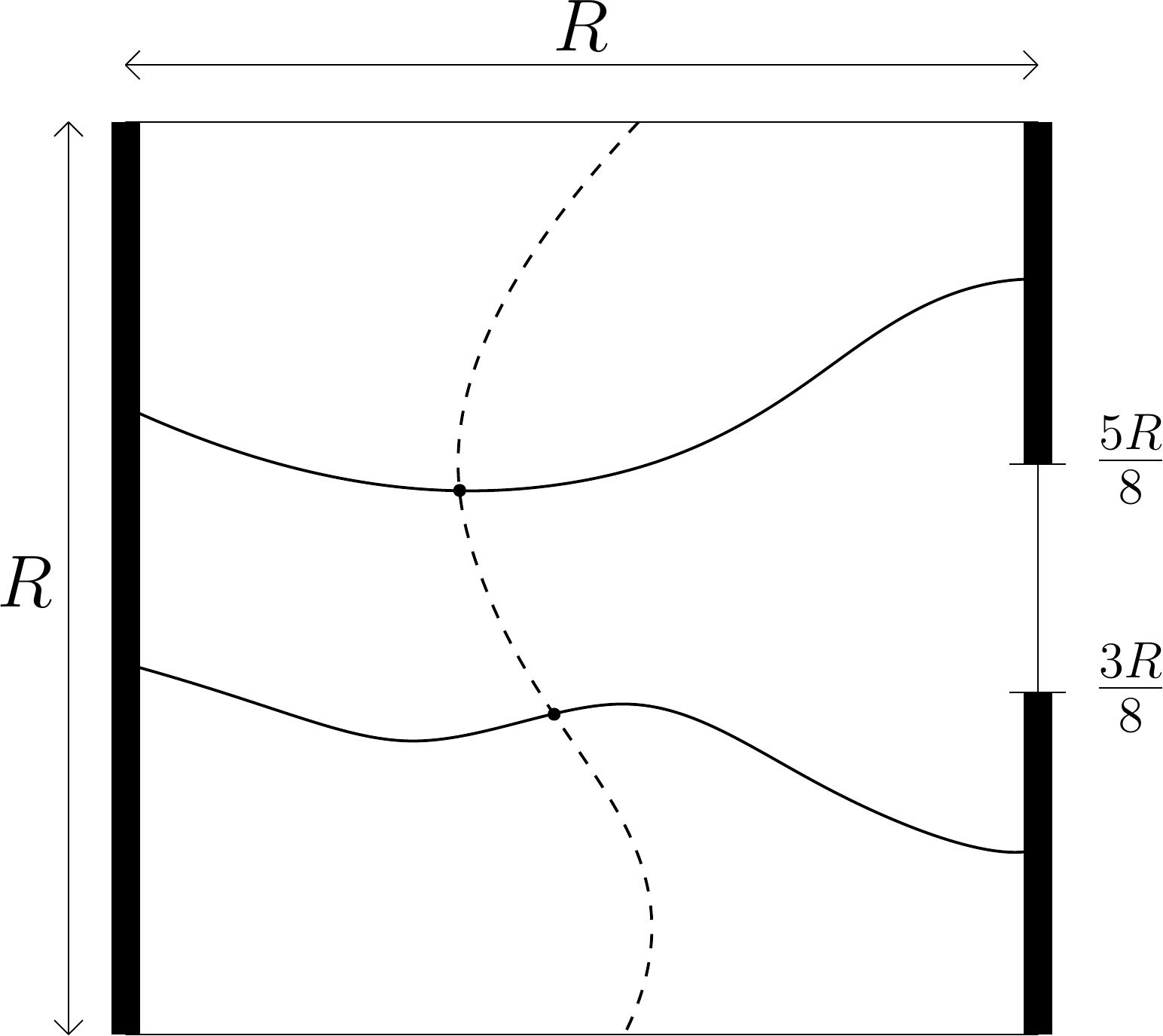}
	\end{minipage}%
	\begin{minipage}{.5\textwidth}
		\centering
		\includegraphics[width=.7\linewidth]{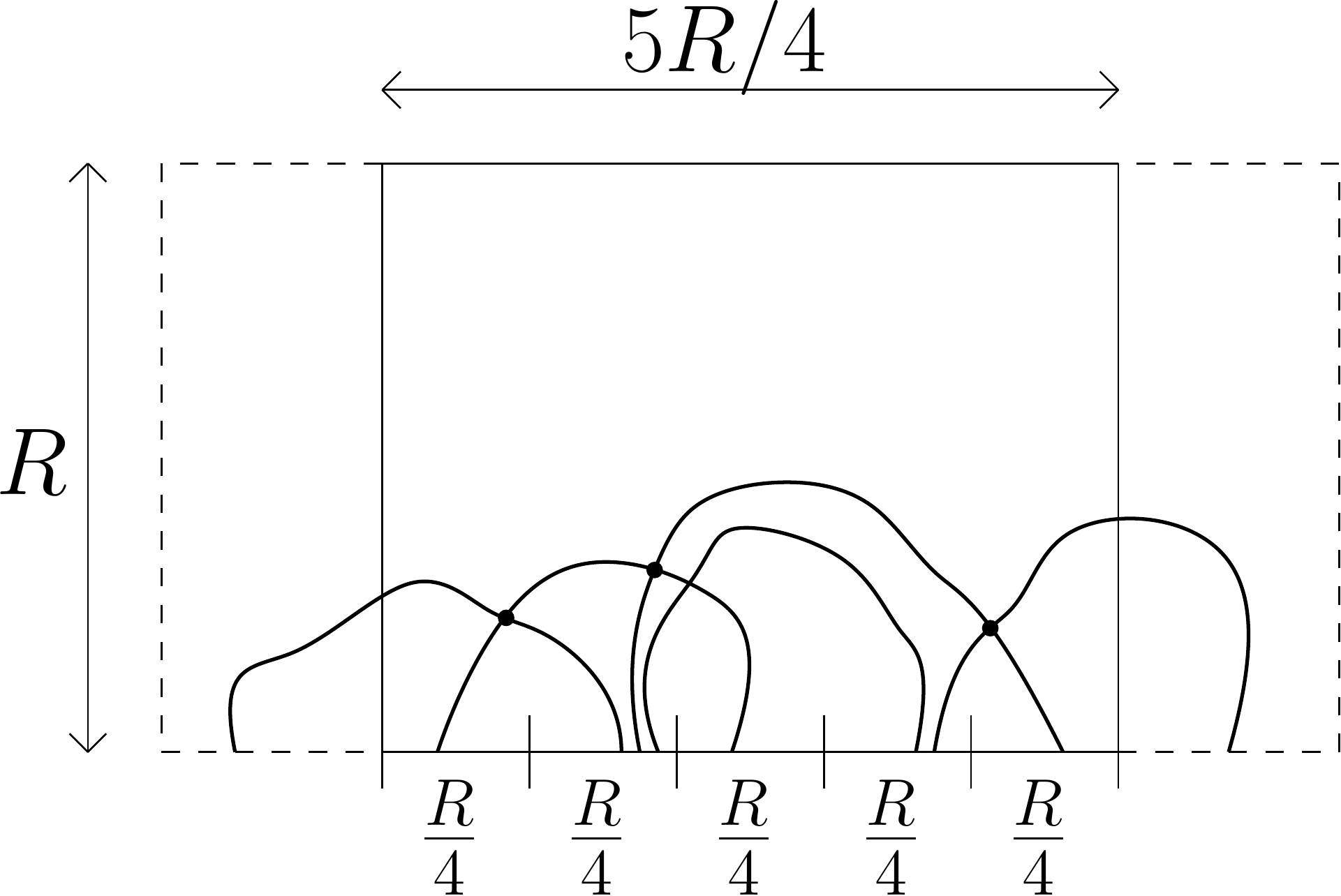}
	\end{minipage}
	\caption{Illustrations of two constructions used in the proof of Proposition \ref{p:gc5}.}
	\label{f:long}
\end{figure}

Importantly, in Lemma \ref{l:gc3}, Proposition \ref{p:gc4}, and Proposition \ref{p:gc5} the choice of copies does not depend on $\mathcal{C}$, and so in particular when $\mathcal{C} = \{f \ge -\ell\}$ it depends neither on $\ell$ nor on the realisation of~$\{f \ge -\ell\}$.

\smallskip
We are now ready to prove Propositions \ref{p:gc1} and \ref{p:gc2}.

\begin{proof}[Proof of Proposition \ref{p:gc1}]
 Let $R > 0$ be such that
 \begin{equation}
 \label{e:good1}
   \alpha_R  \le  2 \alpha_{3R/4}  .
   \end{equation}
We set $Q:= 3R/4$. By Proposition \ref{p:gc4}, there are four copies of $\Cross_\ell(Q;\frac{Q+\alpha_Q}{2},Q)$ and one copy of $\Cross_\ell(Q,Q)$, whose probabilities are, for $\ell=0$, at least $1/8$ by \eqref{e:ar} and whose intersection is contained in 
\begin{equation*}
	X_\ell\left(Q;\alpha_{Q}\right) \subset X_\ell\left(Q;\alpha_{R}/2\right),
\end{equation*}
where $X_\ell(\cdot,\cdot)$ is the obvious analogue of $X(\cdot,\cdot)$ for $\{ f \ge -\ell \}$. Note that the inclusion is a direct consequence of  \eqref{e:good1}. Applying the $\Cross$-$X$-$\Cross$ construction (Lemma \ref{l:gc3}) with two copies of $\Cross_\ell(R;\frac{R}{2},\frac{R+\alpha_R}{2})$ and one copy of  $X_\ell(Q;\alpha_{R}/2)$, we conclude that there exits a collection of seven events whose probabilities are, for $\ell=0$, all at least $1/8$ and whose intersection is contained in 
\begin{equation*}
	\Cross_\ell\left(2R-Q,R\right) = \Cross_\ell\left(5R/4,R\right).
\end{equation*}
For arbitrary $\lambda > 0$, it follows from standard gluing constructions that there exists a possibly larger, but finite collection of copies of the above described events, whose intersection is contained in $\Cross_\ell(\lambda R, R)$.

We have exhibited a collection of crossing domains with the required properties for any $R > 0$ satisfying \eqref{e:good1}. Given Lemma \ref{l:ar}, it only remains to observe that \eqref{e:good1} is satisfied on an unbounded sequence $(R_n)_{n\ge 0}$. This is clear since, if $\alpha_R \ge 2 \alpha_{3R/4}$ for all $R \ge R_0$, then $\alpha_R$ must grow at least stretched-exponentially on a subsequence, contradicting the fact that $\alpha_R \le R$.
\end{proof}

In the proof of Proposition \ref{p:gc1}, we have constructed crossings in long rectangles along an unbounded sequence of scales $(R_n)_{n\ge 0}$. It is natural to ask for a quantitative bound on the maximal distance between two subsequent `good' scales $R_{n-1}$ and $R_{n}$. If we choose $R_0$ sufficiently large to ensure $\alpha_R \ge 1$ for all $R \ge R_0$, then condition \eqref{e:good1} implies a polynomial bound $R_{n} \le (R_{n-1})^{c}$ with $c=\frac{\log(2)}{\log(3/2)}$. 
To prove Proposition \ref{p:gc2} which provides a much stronger bound on the maximal distance between subsequent `good' scales, we will introduce a weaker condition compared to \eqref{e:good1} and show that this condition is still sufficient for the construction of crossings in long rectangles. While the quantitative bound in Proposition \ref{p:gc2} suffices for the purposes of this paper, we remark that the procedure can be iterated to obtain even stronger bounds.

\begin{proof}[Proof of Proposition \ref{p:gc2}]
Let $R > 0$ be such that one of the following holds:
\begin{align}
\label{e:good2.1}
&\text{There exists a scale } Q \in [R/4, R] \text{ such that } \alpha_{Q} \ge  Q/4. \\
\label{e:good2.2}
&\text{There exists a scale } Q \in [R/4, R/2] \text{ such that } \alpha_{Q + \alpha_R} \le 2 \alpha_Q.
\end{align}
Depending on which condition is satisfied, we split the proof into two cases:
\begin{itemize}
	\item[\eqref{e:good2.1}] If \eqref{e:good2.1} holds, fix $Q \in [R/4, R]$ satisfying $\alpha_{Q} \ge  Q/4$. By Proposition \ref{p:gc5}, there are ten copies of the event $\Cross_\ell(Q;5Q/8,Q)$ and five copies of $\Cross_\ell(Q,Q)$, whose intersection is contained in 
	\begin{equation*}
		\Cross_\ell(5Q/4,Q),
	\end{equation*}
	and whose probabilities are, for $\ell=0$, at least $1/8$ by the assumption $\alpha_{Q} \ge  Q/4$ together with \eqref{e:ar}. For arbitrary $\lambda > 0$, it follows from standard gluing constructions that there exists a finite collection of copies of the above described events, whose intersection is contained in $\Cross_\ell(\lambda R, R)$.
	
	\item[\eqref{e:good2.2}] If only \eqref{e:good2.2} holds, fix $Q \in [R/4,R/2]$ satisfying $\alpha_{Q + \alpha_R} \le 2 \alpha_Q$ and set $Q' = Q + \alpha_R$. We consider the family of events 
	\begin{equation*}
		\mathcal{S}= \left\{ \Cross_\ell\left(Q';\frac{Q'}{2},\frac{Q'+\alpha_{Q'}}{2}\right), \Cross_\ell\left(Q;\frac{Q+\alpha_Q}{2},Q\right), \Cross_\ell(Q,Q) \right\},
	\end{equation*}
	and note that, for $\ell=0$, the probability of  any event in $\mathcal{S}$ is at least $1/8$ by \eqref{e:ar}. By Proposition \ref{p:gc4}, there are four copies of  $\Cross_\ell(Q;\frac{Q+\alpha_Q}{2},Q)$ and one copy of $\Cross_\ell(Q,Q)$, whose intersection is contained in 
	\begin{equation*}
		X_\ell\left(Q;\alpha_{Q}\right) \subset X_\ell\left(Q;\alpha_{Q'}/2\right).
	\end{equation*}
	Applying the $\Cross$-$X$-$\Cross$ construction (Lemma \ref{l:gc3}) with two copies of the event $\Cross_\ell(Q';\frac{Q'}{2},\frac{Q'+\alpha_{Q'}}{2})$ and one copy of $X_\ell(Q,\alpha_{Q'}/2)$, we conclude that there exist seven copies of events in $\mathcal{S}$ whose intersection is contained in
	\begin{equation*}
		\Cross_\ell\left(2Q'-Q,Q'\right) = \Cross_\ell\left(Q' + \alpha_R,Q'\right).
	\end{equation*}
	Now, consider the family of events 
	\begin{equation*}
		\mathcal{S}'= \mathcal{S} \cup \left\{ \Cross_\ell\left(R;\frac{R-\alpha_R}{2},\frac{R+\alpha_R}{2}\right), \Cross_\ell(Q';Q'/2,Q') \right\},
	\end{equation*}
	and note again that, for $\ell=0$, the probability of  any event in $\mathcal{S}'$ is at least $1/8$.
	Applying Proposition \ref{p:gc4}, there are four copies of $\Cross_\ell(Q';Q'/2,Q')$ and one copy of $\Cross_\ell(Q' + \alpha_R,Q')$, whose intersection is contained in 
	\begin{equation*}
		X_\ell\left(Q';\alpha_{R}\right).
	\end{equation*}
	Another $\Cross$-$X$-$\Cross$ construction (Lemma \ref{l:gc3}) using two copies of the event $\Cross_\ell(R;\frac{R-\alpha_R}{2},\frac{R+\alpha_{R}}{2})$ and one copy of $X_\ell(Q',\alpha_{R})$ implies that there exists a finite collection of copies of events in $\mathcal{S}'$ whose intersection is contained in
		\begin{equation*}
		\Cross_\ell\left(2R-Q',R\right) \subset \Cross_\ell\left(5/4 R,R\right).
	\end{equation*}
	For arbitrary $\lambda > 0$, it follows from standard gluing constructions that there exists a possibly larger, but finite collection of copies of the events in $\mathcal{S}'$, whose intersection is contained in $\Cross_\ell(\lambda R, R)$.
	
\end{itemize}
In all cases we have exhibited a collection of crossing domains with the required properties for any $R > 0$ satisfying \eqref{e:good2.1} or \eqref{e:good2.2}. Given Lemma \ref{l:ar}, it only remains to argue that, for any $k \in \N$, \eqref{e:good2.1} or \eqref{e:good2.2} is satisfied on a  sequence $R_n \to \infty$ such that eventually
\[ R_{n+1} \le R_n \log^{(k)} R_n   .\]
Applying Lemma \ref{l:goodscales} below to the function $R \mapsto \max\{1,\alpha_R\}$ (which by Lemma \ref{l:ar} coincides with $\alpha_R$ eventually), we conclude the proof.
\end{proof}

Recall the definition of the \textit{base-$b$ iterated logarithm} 
\begin{equation}
\label{e:iteratedlog}
\log_b^\ast(x) = \min\{ k \in \N : \log_b^{(k)}(x) \le 1 \} ,
\end{equation}
where $\log_b^{(k)} = \log_b \log_b \cdots \log_b$ denotes the $k$-fold iteration of the base-$b$ logarithm.

\begin{lemma}
\label{l:goodscales}
Let $\alpha(x) \in [1,x]$ and define
\begin{align*}
	 G = &\big\{ x : \text{there exists }  y \in [x/4, x] \text{ with } \alpha(y) \ge y/4 \big\}\\
	 \cup &\big\{ x : \text{there exists }  y \in [x/4, x/2] \text{ with } \alpha \big(y + \alpha(x) \big) \le 2 \alpha(y) \big\}.
\end{align*}
Then there exists an unbounded subsequence $(x_n)_{n \ge 1} \subset G$ satisfying
\[ x_{n+1}\le x_n 4^{\log^\ast_{2^{1/4}} (x_n)} . \]
In particular, for any $k \in \N$ this sequence satisfies eventually
\[x_{n+1} \le x_n\log^{(k)} (x_n) . \]
\end{lemma}
\begin{proof}
Consider $x \notin G$. Then iterating the bound 
\[  \alpha\big(y + \alpha(x)\big) > 2 \alpha(y)  \ , \quad \text{for all } y \in [x/4, x/2] , \]
along the subsequence $(y_n)_{n \ge 0} = (x/4 + n\alpha(x))_{n \ge 0}$, we have
\begin{equation}
\label{e:gs1}
\alpha\big(x/4 + \lceil  (x/4)/\alpha(x) \rceil \alpha(x) \big) > 2^{(x/4)/\alpha(x)} \alpha(x/4) . 
\end{equation}
Since $\alpha(x) \le x/4$, the left-hand side of \eqref{e:gs1} is at most $x/4$, and so we have
\begin{equation}
\label{e:h}
x/4 > (2^{1/4})^{x/\alpha(x)} \alpha(x/4) \ \Longleftrightarrow \ \frac{x}{\alpha(x)} < \log_{2^{1/4}} \Big(  \frac{x/4}{\alpha(x/4)} \Big) \  \Longleftrightarrow \ f(x) <  \log_b  f(x/4) 
\end{equation}
where in the last equivalence we abbreviated $b =2^{1/4}$ and $f(x) = x / \alpha(x) \in [4, x]$. 
 
\smallskip
Now fix $x_0 > 0$, define $I_{x_0} =[x_0,4^{\log^\ast_b(x_0)}x_0]$, and suppose for the sake of contradiction that $I_{x_0} \cap G$ is empty, and so in particular \eqref{e:h} holds for all $x \in I_{x_0}$. By iterating \eqref{e:h} along the subsequence $(4^n x_0)_{n \ge 1}$ we have
\[  f(4^{\log_b^\ast(x_0)} x_0) < \log_b^{(\log^\ast_b(x_0))} \left(f(x_0)\right) \le  \log_b^{(\log^\ast_b(x_0))}(x_0) \le 1 ,     \]
where the second inequality follows from $f(x_0) \le x_0$, and the last inequality follows by the definition of the iterated logarithm \eqref{e:iteratedlog}. Since $f(x) \geq 1$ for all $x > 0$, this is a contradiction and we conclude that $ I_{x_0} \cap G$ is non-empty for every $x_0 > 0$. The result follows by induction.
\end{proof}

\medskip
\bibliographystyle{alpha}
\bibliography{paper}

\end{document}